%% file: smith-warsaw-Xi-arxiv.tex
\pdfoutput=1
\documentclass{amsart}
\usepackage{amsmath,amsthm}
\usepackage{amssymb}
\usepackage{graphicx}
\usepackage[T1]{fontenc}
\usepackage[final]{microtype}
\usepackage{cmtiup}
\usepackage[x11colors]{xcolor}
\usepackage{tikz}
\usetikzlibrary{decorations.pathreplacing,arrows,positioning,matrix}
\tikzset{>=stealth}
\usepackage[scr=boondox]{mathalfa}
\usepackage{amssymb}
\usepackage[colorlinks]{hyperref}

\numberwithin{equation}{section}                                 
\newtheorem{thm}[equation]{Theorem}

\newtheorem{lemma}[equation]{Lemma}
\newtheorem{cor}[equation]{Corollary}
\theoremstyle{definition}
\newtheorem{defn}[equation]{Definition}
\theoremstyle{remark}
\newtheorem{rmk}[equation]{Remark}

\newcommand{\lhk}{\mathbin{\hbox{\vrule height1.4pt width4pt depth-1pt
\vrule height4pt width0.4pt depth-1pt}}} 

\newcommand{\poisson}[1]{\ensuremath{\boldsymbol{\{}#1\boldsymbol{\}}}}
\newcommand{\pair}[1]{\ensuremath{\left\langle #1 \right\rangle}}

\newcommand{\mapsfrom}{\mathrel{\reflectbox{\ensuremath{\mapsto}}}}

\DeclareMathOperator{\coker}{coker}
\DeclareMathOperator{\rank}{rank}

\DeclareMathOperator{\Gr}{Gr}
\DeclareMathOperator{\Pol}{Pol}
\DeclareMathOperator{\Sym}{Sym}
\DeclareMathOperator{\Var}{Var}
\DeclareMathOperator{\B}{B}
\DeclareMathOperator{\Cone}{\mathscr{C}}
\DeclareMathOperator{\Eik}{\mathscr{E}}
\DeclareMathOperator{\Wu}{\mathbf{W}}
\DeclareMathOperator{\End}{End}

\DeclareMathOperator{\Sec}{Sec}
\DeclareMathOperator{\Hom}{Hom}

\DeclareMathOperator{\im}{im}

\newcommand{\taut}{\boldsymbol{\gamma}}
\definecolor{rd}{rgb}{0.0, 0.0, 0.9}
\definecolor{io}{rgb}{1.0, 0.65, 0.0}
\definecolor{lb}{rgb}{0.48, 0.45, 0.8}
\definecolor{c1}{rgb}{0.82, .10, 0.26}
\definecolor{c2}{rgb}{0.6, 0.4, 0.8}
\definecolor{c3}{rgb}{0.0, 0.0, 1.0}
\definecolor{c4}{rgb}{ 0.55, 0.71,0}
\definecolor{darkred}{rgb}{0.56,0,0.11}
\newcommand{\la}{\ensuremath{\textcolor{c1}{c_1}}}
\newcommand{\lb}{\ensuremath{\textcolor{c2}{c_2}}}
\newcommand{\lc}{\ensuremath{\textcolor{c3}{c_3}}}
\newcommand{\ld}{\ensuremath{\textcolor{c4}{c_4}}}
\newcommand{\ma}{\ensuremath{\textcolor{c1}{d_1}}}
\newcommand{\mb}{\ensuremath{\textcolor{c2}{d_2}}}
\newcommand{\mc}{\ensuremath{\textcolor{c3}{d_3}}}
\newcommand{\md}{\ensuremath{\textcolor{c4}{d_4}}}

\def \CP1{
   \draw (-1,0) arc (180:360:1cm and 0.5cm);
    \draw[dashed] (-1,0) arc (180:0:1cm and 0.5cm);
    \draw (0,1) arc (90:270:0.5cm and 1cm);
    \draw[dashed] (0,1) arc (90:-90:0.5cm and 1cm);
    \draw (0,0) circle (1cm);
    \draw (0,0) -- (1.25,0);
    \draw (0.05,0.12) -- (0.05,-0.11) -- (-0.05, -0.13) -- (-0.05, 0.12) -- cycle;
}


\title[Involutive Tableaux]{Involutive Tableaux, Characteristic Varieties, and Rank-one Varieties in the Geometric Study of PDEs}
\author[A. D. Smith]{Abraham D. Smith}
\address{Department of Mathematics, Statistics and Computer Science\\University of Wisconsin-Stout\\Menomonie, Wisconsin  54751-2506, USA}
\date{\today}
\email{smithabr@uwstout.edu}
\begin{document}

\begin{abstract}
This expository monograph cuts a short path from the common, elementary
background in geometry (linear algebra, vector bundles, and algebraic ideals) to
the most advanced theorems about involutive exterior differential systems:
(1) The incidence correspondence of the characteristic variety,
(2) Guillemin normal form and Quillen's thesis,
(3) The Integrability of Characteristics by Guillemin, Quillen, Sternberg, and 
Gabber, and (4) Yang's Hyperbolicity Criterion.
To do so, the geometric theory of PDEs is reinterpreted as the study of smooth
sub-bundles of the Grassmann bundle, whereby the rank-1 variety is emphasized.
The primary computational tool is an enhanced formulation of Guillemin normal
form that is equivalent to involutivity of tableaux. 
\end{abstract}

\maketitle

\microtypesetup{protrusion=false} 
\tableofcontents 
\microtypesetup{protrusion=true} 

\setcounter{section}{-1}
\section{Introduction and Overview}

Given a system of partial differential equations [PDEs] over a manifold, does the
system of PDEs have any local
solutions to the Cauchy initial-value problem?  That is, given initial conditions on a
locally-defined hypersurface, can we produce a local solution that satisfies
those initial conditions and also satisfies the PDEs? 
More generally, which initial hypersurfaces admit such solutions?  Can we do
this iteratively by solving a sequence of initial-value problems from dimension
0 to 1, 1 to 2, and so on to build solutions through any point? 

These questions are the heart of exterior differential systems [EDS], a
powerful specialist approach to the geometric study of PDEs.  EDS 
typically present as ideals of exterior differential forms over a manifold.

Some deeper questions are: What is the shape of the family of local solutions
obtained in this way?  How can we determine whether two systems of PDEs are
``the same'' up to local coordinate transformations?  Does the space of all
PDEs (up to local coordinate transformation) have any meaningful shape or
structure of its own?

These deeper questions are answered by analyzing the characteristic variety of
an EDS.  The original motivation for the characteristic variety is to see where
the Cauchy initial-value problem becomes ambiguous.  That is, given an initial
condition for our PDEs on a local submanifold of dimension $n{-}1$, when would
the $n$-dimensional solutions for that initial condition fail to be unique?

When analyzing the characteristic variety of various EDS, one discovers that
the characteristic variety is an exquisitely subtle structure that reveals far
more than originally anticipated.  The characteristic variety dictates the
internal geometry of the solutions of the original PDEs, while also controlling
the parameter space of all such solutions.  Under reasonable hypotheses, this
means that EDS or PDEs can be understood up to coordinate equivalence as
``parametrized families of solution manifolds with associated characteristic
geometry.'' 

This is beautiful and important, but it has been a difficult topic for
researchers to access, because the foundations of EDS have not yet entered the
common curriculum of graduate students.  Fluency
with differential ideals remains a relatively rare skill, practiced 
in a handful of schools worldwide.
Indeed, it is common for researchers first encountering the subject to become trapped in
an endless cycle of translating systems from local jet coordinates to
differential forms and back again, without gaining any new geometric insights
and without using the most powerful theoretical ideas in EDS.  In
particular, it can take many years for researchers to appreciate the central
role that the characteristic variety plays in uncovering geometric insights.
However---despite the name---differential forms are not themselves the core
idea behind exterior differential systems.  Differential forms are merely a
concise language. Rather, the core idea is
to recognize that these questions are more geometric than analytic, and that
ideals (that is, conditions defined by functions) and varieties (that is,
shapes cut out by functions) must come into play.  To describe families of
solutions, we need the geometric language of bundles, schemes, and moduli.

Therefore, the goal of this monograph is to cut the shortest-possible
expository path from the common curriculum in geometry (linear algebra, vector
bundles, and algebraic ideals) to several key results regarding the
characteristic variety.  
These key results are
\begin{enumerate}
\item the incidence correspondence of the characteristic and rank-1 varieties,
and its relationship to eigenspace decomposition,
\item Guillemin normal form, its enhancements, and Quillen's thesis, 
\item the Integrability of Characteristics (Guillemin, Quillen, Sternberg,
Gabber), and
\item Yang's Hyperbolicity Criterion.
\end{enumerate}
The required common curriculum is
\begin{enumerate}
\item graduate-level linear algebra
(short-exact sequences, dual spaces, the rank-nullity theorem, tensor products,
generalized eigenspaces, as in Artin's \emph{Algebra} \cite{Artin1991}),
\item the fundamentals of smooth manifolds (tangent spaces, Sard's theorem, bundles, as in Milnor's
\emph{Topology from the Differential Viewpoint} \cite{Milnor1997}), and
\item  basic algebraic geometry (projective space, ideal, variety, scheme, as
in Harris' \emph{Algebraic Geometry, a first course} \cite{Harris1992}).  
\end{enumerate}

To accomplish this, the subject of exterior differential systems is
reinterpreted as the study of smooth sub-bundles of the Grassmann bundle over a
smooth manifold. 
In doing so, the role of exterior differential forms becomes obscured, in favor
of tableaux (vector spaces of homomorphisms) and symbols (varieties of endomorphisms).    
Specifically, Guillemin normal form for tableaux and symbols plays the central
computational role, not differential forms.  This is because most humans---and
their computer algebra systems---are more comfortable with matrices than with
exterior algebra. 
Exterior differential ideals are not introduced until absolutely needed.
This is because many of the essential lemmas depend only on the geometry of the
Grassmann bundle, which is the variety of the trivial exterior differential
system.
This reformulation allows elementary versions of those key results (in fact,
almost all the lemmas are restatements of the rank-nullity theorem), and it
becomes possible to outline how these results could be used to push the
frontiers of the subject.

While the audience is assumed to have a general
interest and cultural awareness of PDEs or EDS in some form, all the required definitions are
provided when needed. 
Even so, it is wise always to have Bryant, Chern, Gardner, Goldschmidt, and Griffiths's \emph{Exterior Differential
Systems} \cite{BCGGG} and Ivey and Landsberg's \emph{Cartan for Beginners}
\cite{Ivey2003} nearby.  They are cited for comparison frequently.  Another
excellent reference is McKay's \emph{Introduction to Exterior Differential
Systems} \cite{McKay2018}.
A note for EDS experts: the results in these pages can be found in
numerous places in the literature in some form or other, and I have indicated my
favorite sources throughout.  The only innovation here is in presentation.
Most notably, in contrast to the vast majority of expositions, 
the central topic is the $C^\infty$ characteristic variety, not the
$C^\omega$ Cartan--K\"ahler theorem.  This is because the key open question is 
``what does the family of all involutive PDEs look like?'' not ``how do I solve
this particular involutive PDE?''

This monograph is organized into four parts, each containing several sections.
Part~\ref{part:back} covers the structure of tableau (subspaces of a space of
homomorphisms) and the differential geometry of the Grassmann
variety.  Because the results are elementary---almost trivially so---they
provide a good foundation for building from the common curriculum to the
central topic.  Part~\ref{part:manifolds} converts those elementary results to
the language of bundles, PDEs and EDS.  That language allows 
an enhanced version of Guillemin normal form that is equivalent to
involutivity.
Part~\ref{part:Xi} achieves the key purpose of this monograph,
as a triumvirate is formed by the characteristic scheme, the
rank-1 cone, and the mutual eigenvector problem for symbols. 
Part~\ref{part:eikonal} examines the integrability of the characteristic
variety in various guises, and offers a general dogma (that the characteristic scheme
knows all coordinate-invariant data about a system of PDEs) 
that suggests possible future developments in the theory of
EDS.

This monograph was developed to support a series of lectures at the Institute
of Mathematics at the Polish Academy of Sciences in September 2016, as part of
a Workshop on the Geometry of Lagrangian Grassmannians and Nonlinear PDEs.

\part{Matrices and Subspaces}
\label{part:back}
\section{Tableaux and Symbols}\label{sec:tableau}
 Tableaux are very simple objects; every undergraduate encounters the example
 ``$r \times n$ matrices form a vector space using the usual matrix
 operations,'' and a tableau is any subspace of that vector space.

Given vector\footnote{When it becomes appropriate to do so, at \eqref{eqn:mybundles} in
Section~\ref{sec:eds}, we switch from vector spaces to
complex projective spaces for algebraic convenience.}
spaces $W$ and $V$ with $\dim W = r$ and $\dim V = n$,  a \emph{tableau} is a
linear subspace of $A \subset \Hom(V,W)$.  We use the notation $W \otimes V^*$ and
$\Hom(V,W)$ interchangeably.

Being a subspace, any tableau $A$ is the kernel of some linear map $\sigma$, called the \emph{symbol}, whose
range is written as $H^1(A)$.  We have a short exact sequence of spaces:
\begin{equation}
0 \to A \to W \otimes V^* \overset{\sigma}{\to} H^1(A) \to 0,
\label{eqn:H1}
\end{equation}
where $H^1(A)$ is just notation for $(W \otimes V^*)/A$.
Let $\dim A = s$ and $\dim H^1(A) = t = nr-s$.

For example, let $W = \mathbb{R}^3$ and $V = \mathbb{R}^3$, and consider the
5-dimensional tableau $A \subset W \otimes V^*$ described in the standard bases by 
\begin{equation}
\left\{ 
\begin{pmatrix}
\alpha_0 & \alpha_1 & \alpha_2 \\
\alpha_1 & \alpha_2 & \alpha_3 \\
\alpha_2 & \alpha_3 & \alpha_4 
\end{pmatrix} ~:~ \alpha_i \in \mathbb{R} \right\}.\label{eqn:example_tab}\end{equation}
If $\pi \in W \otimes V^*$ is a $3 \times 3$ matrix with entries $\pi^a_i$, then 
the symbol $\sigma$ defining $A$ consists of four conditions:
\begin{equation}
\begin{split}
0 &= \pi^2_3 - \pi^3_2,\\
0 &= \pi^1_3 - \pi^3_1,\\
0 &= \pi^2_2 - \pi^3_1,\\
0 &= \pi^1_2 - \pi^2_1.
\end{split}
\label{eqn:example_sym}\end{equation}

\subsection{Rank-one ideal}
The fundamental theorem of linear algebra states that any homomorphism $\pi \in W
\otimes V^*$ has a well-defined rank.  Thus, for any tableau $A
\subset W \otimes V^*$, we could ask how $\rank(\pi)$ varies across $\pi \in A$.
For our purposes, the most interesting\footnote{There is a good reason that the
rank-1 case is most interesting: the varieties of higher-rank matrices 
are determined algebraically by the varieties of lower-rank matrices, so the geometry of
$\rank(\pi)$ across $\pi \in A$ comes down to the rank-1 case.}
case is $\rank(\pi) = 1$.

The space $W \otimes V^*$ admits the \emph{rank-1 ideal}, $\mathscr{R}$,
which is irreducible and generated by all $2\times 2$ minors $\left\{ 0= \pi^a_i\pi^b_j -
\pi^a_j\pi^b_i\right\}$ in any basis.  This is a homogeneous ideal, so we may
consider the rank-1 cone in vector space or the rank-1 variety in
projective space.  (The vertex of the rank-1 cone is the rank-0 matrix.)

For any $A$, consider the ideal $A^\perp + \mathscr{R}$, which 
defines $\Cone \subset A$ as the variety $\Cone = A \cap
\Var(\mathscr{R})$.  The variety $\Cone$ is the set of matrices in $A$ that are also
rank-1; it is a linear section of the rank-1 cone defined by $\mathscr{R}$.

In the example \eqref{eqn:example_tab}, $\Cone$ can be parametrized as matrices of the form
\begin{equation}
\begin{pmatrix}
\kappa^4 & \kappa^3\tau & \kappa^2\tau^2 \\
\kappa^3\tau & \kappa^2\tau^2 & \kappa\tau^3 \\
\kappa^2\tau^2 & \kappa\tau^3 & \tau^4 
\end{pmatrix}= 
\begin{pmatrix}
\kappa^2 \\ \kappa\tau \\ \tau^2
\end{pmatrix}\otimes
\begin{pmatrix}
\kappa^2 & \kappa\tau & \tau^2
\end{pmatrix},
\label{eqn:C}\end{equation}
which can be interpreted as the rational normal Veronese curve\footnote{
For more on Veronese curves and the more general Segre embeddings and determinantal varieties, see
\cite{Harris1992,Shafarevich1994}.},  \begin{equation}
[\kappa^4: \kappa^3\tau : \kappa^2\tau^2: \kappa\tau^3 : \tau^4] 
\cong \mathbb{P}^1 \subset \mathbb{P}^4 \cong \mathbb{P}A.\end{equation}
Moreover, the projection of $\Cone$ to $\mathbb{P}V^*$ is another rational
normal curve,
 \begin{equation}
[\kappa^2: \kappa\tau : \tau^2] 
\cong \mathbb{P}^1 \subset \mathbb{P}^2 \cong \mathbb{P}V^*.\end{equation}
This toy example plays a crucial role in applications for hyperbolic and
hydrodynamically integrable PDEs \cite{Ferapontov2009, Smith2009}.

\subsection{Generic Bases}\label{sec:regular} 
We would like to find a ``good'' basis in which to express a tableau $A$ and
study its properties.

First, an analogy.   When studying a single homomorphism $F:\mathbb{C}^n \to
\mathbb{C}^r$, or $F \in \mathbb{C}^r \otimes \mathbb{C}^{n*}$, there are various ``good'' bases of the domain and co-domain to express $F$.  A basis of
$\mathbb{C}^{n*}$ is ``generic'' for $F$ if the first $\rank(F)$ columns are
independent.  A basis of $\mathbb{F}^{r}$ is ``generic'' for $F$ if the first
$\rank(F)$ rows of $F$ are independent in that basis.  Among the generic bases,
we can construct particularly ``good'' bases for writing $F$. When $F$ is written 
in a ``good'' basis, we say it is in a ``normal form,'' and the normal form
allows us readily to study properties of $F$.  
For example:
\begin{itemize}
\item Use Gaussian elimination\footnote{Algorithmically, this is accomplished
using improved Gram-Schmidt or Householder triangularization.  See
\cite{Trefethen1997} for a discussion of stability of row-reduction.} to place $F$ in reduced row-echelon form.  Then, the
rank, kernel, and image of $F$ are immediately apparent.  The fundamental
theorems in linear algebra depend on this normal form.

\item Apply a polar/unitary decomposition to find the singular-value decomposition of
$F$.  Then, the norm of $F$ and its action with respect to the Hermitian inner
products of $\mathbb{C}^n$ and $\mathbb{C}^r$ are immediately apparent.
Important theorems in metric geometry and multivariate statistics depend on
this normal form.

\item Solve a sequence of eigenvalue problems in the case $n=r$ to find Jordan
normal form.  Then, the eigenspace structure of $F$, and the commutative
algebra of matrices to which it belongs are immediately apparent.  The theory
of Lie groups and Lie algebras depends on this normal form.
\end{itemize} 

Given a tableau $A \subset W \otimes V^*$ with symbol $\sigma$, we are curious
whether we can construct bases that are ``good'' simultaneously for all
homomorphisms in the tableau.  This situation is considerably more complicated
than the situation of a single homomorphism, and it turns out that it is most
important to focus on the symbol maps, but we arrive at a satisfying
answer in Section~\ref{sec:inveig}.  By the above analogy, it is convenient to establish a notion of ``generic'' bases
formulated in terms of independence.  This is done as follows.

In any bases of $V^*$ and $W$, the tableau $A$ is a space of $r \times n$
matrices only $s$ of whose entries are linearly independent.  That is, in a
given basis, we can consider the entries $\pi \mapsto \pi^a_i$ as elements of
$A^*$, just as we think of the components $v \mapsto v^i$ of vectors in $V$ as being linear
functions on $v \in \mathbb{R}^n$, using the dual basis of $V^*$.

Across all bases of $V^*$, there is a maximum number of independent entries
that can occur in column 1; call that number $s_1$.  (In a measure-zero set of
bases of $V^*$, the number of actual independent entries in the first column
may be less than $s_1$.)  Once those independent entries are accounted for,
there is a maximum number $s_2$ of new independent entries that can occur in
the second column. (In a measure-zero set of bases of $V^*$ that achieve $s_1$
in column 1, the number of actual independent entries in columns 1 and 2 may be
less than $s_1+s_2$.)  Once those independent entries are accounted for, there
is a maximum number $s_3$ of new independent entries that can occur in
column 3. (In a measure-zero set of bases of $V^*$ that achieve $s_1+s_2$
in columns 1 and 2, the number of actual independent entries in columns 1, 2,
and 3 may be less than $s_1+s_2+s_3$.) Continuing in this way, we have $s_i$ as
the number of new independent entries in the $i$th column achieved for
almost-all bases of $V^*$.  (In a  measure-zero set of bases of $V^*$ that
achieve $s_1+s_2+\cdots+s_{i-1}$ in columns 1 through $i{-}1$, the number of
actual independent entries in columns 1 through $i$ may be less than
$s_1+\cdots+s_i$.) Eventually, for such a basis, there is a column $\ell$ where
we have reached $s_1+s_2 + \cdots + s_\ell= s$, so there is some maximum column
$\ell \leq n$ such that $s_\ell >0$, where the last independent entry appears.
So, \begin{equation} \begin{split} s &= s_1 + s_2 + \cdots + s_{\ell} +
s_{\ell+1} + \cdots  +s_n\\ &= s_1 + s_2 + \cdots + s_\ell + 0 + \cdots + 0.
\end{split} \end{equation} The index $\ell$ is called the \emph{character} of
$A$, and the number $s_\ell$ is called the \emph{Cartan integer} of $A$.  The
tuple $(s_1, \ldots, s_\ell)$ gives the \emph{Cartan characters} of $A$.  Note
that $s_1 \geq s_2 \geq \cdots \geq s_{\ell}$, since otherwise the maximality
condition would have been violated in an earlier column.

Permanently reserve the index ranges 
\begin{equation}
\begin{split}
i,j &\in \{1, \ldots, \ell, \ell+1, \ldots, n\},\\
\lambda, \mu &\in \{1, \ldots, \ell\phantom{, \ell+1, \ldots, n}\},\\
\varrho, \varsigma &\in \{\phantom{1, \ldots, \ell, }\ell+1, \ldots, n\},\ \text{and}\\
a,b &\in \{1, \ldots, r\}.
\label{eqn:index}
\end{split}
\end{equation}

A basis\footnote{We follow the notation of differential geometry. This notation indicates an ordered basis of co-vectors, not a vector.  Each
$u^i$ is an element of $V^*$.} $(u^i) = (u^1, \ldots, u^n)$ of $V^*$ is called
\emph{generic} if its Cartan characters achieve the lexicographical maximum value
$(s_1, s_2, \ldots, s_n)$.  As seen in the previous paragraph, almost all bases
of $V^*$ are generic in this sense.  Given a basis $(u^i)$ of $V^*$, a
basis\footnote{We follow the notation of differential geometry. This notation indicates an ordered basis of vectors, not a co-vector.  Each
$z_a$ is an element of $W$.} $(z_a) = (z_1, \ldots, z_r)$ of $W$ is called
\emph{generic} if the \emph{first} $s_i$ independent entries in column $i$ are
independent.  

Choose generic a basis $(u^i) = (u^1, \ldots, u^n)$ for $V^*$, and let $(u_i) =
(u_1, \ldots, u_n)$ be its dual basis for $V$.  Choose a generic basis $(z_a) =
(z_1, \ldots, z_r)$ for $W$, and let $(z^a) = (z^1,\ldots, z^n)$ be its dual
basis for $W^*$.   
An element of the tableau is written as \begin{equation}\pi = \pi^a_i (z_a
\otimes u^i) \in W \otimes V^*,\end{equation} and the upper-left entries
$\pi^a_\lambda$ for $a \leq s_\lambda$ form a basis of $A^*$. 

Because the bases are generic, the symbol map $\sigma$ can be written as 
\begin{equation}
\Big\{ 0 = \pi^a_i - B^{a,\lambda}_{i,b} \pi^b_\lambda ~:~ 1 \leq i \leq n,\
s_i < a \leq r\Big\}.
\label{eqn:symrels}
\end{equation}
It is implicit that $B^{a,\lambda}_{i,b} = 0$ if $a \leq s_i$ or $b \geq
s_\lambda$ or $i<\lambda$.
That is, entries to the lower-right (unshaded) are written as linear combinations of the
entries in the upper-left (shaded) using the coefficients $B^{a,\lambda}_{i,b}$, as in Figure~\ref{fig:figtab}.

Consider the example \eqref{eqn:example_sym}, which is not written in generic
bases.  If we exchange columns  $2\leftrightarrow 3$ and rows
$1\leftrightarrow 3$, then it
becomes generic with $(s_1, s_2, s_3) = (3,2,0)$, seen here:
\begin{equation}
\left\{ 
\begin{pmatrix}
\alpha_2  & \alpha_4 & \alpha_3 \\ 
\alpha_1  & \alpha_3 & \alpha_2 \\
\alpha_0  & \alpha_2 & \alpha_1 \\
\end{pmatrix} \right\}
=
\left\{ 
\begin{pmatrix}
\colorbox{io}{$\pi^1_1$}  & \colorbox{io}{$\pi^1_2$} & \pi^2_2 \\
\colorbox{io}{$\pi^2_1$}  & \colorbox{io}{$\pi^2_2$} & \pi^1_1 \\
\colorbox{io}{$\pi^3_1$}  & \pi^1_1 & \pi^2_1 \\
\end{pmatrix} \right\}
.\label{eqn:example_tab_generic}\end{equation}
Equation~\eqref{eqn:symrels} becomes:
\begin{equation}
\begin{split}
0 &= \pi^3_2 - 1 \pi^1_1 - 0 \pi^2_1 - 0 \pi^3_1 - 0 \pi^1_2 - 0 \pi^2_2,\\
0 &= \pi^1_3 - 0 \pi^1_1 - 0 \pi^2_1 - 0 \pi^3_1 - 0 \pi^1_2 - 1 \pi^2_2,\\
0 &= \pi^2_3 - 1 \pi^1_1 - 0 \pi^2_1 - 0 \pi^3_1 - 0 \pi^1_2 - 0 \pi^2_2,\\
0 &= \pi^3_3 - 0 \pi^1_1 - 1 \pi^2_1 - 0 \pi^3_1 - 0 \pi^1_2 - 0 \pi^2_2.\\
\end{split}
\label{eqn:example_sym_generic}\end{equation}

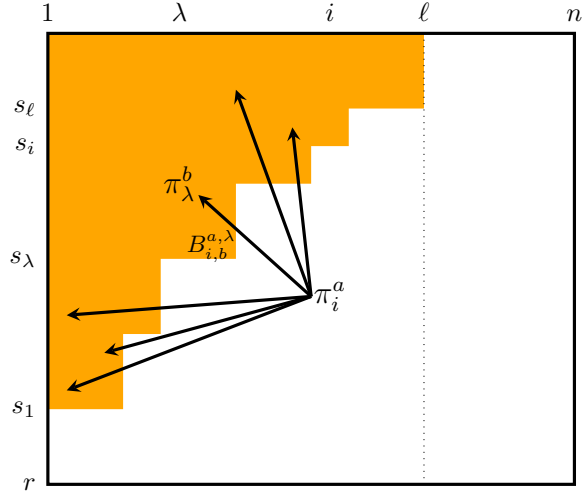
\begin{figure}
\begin{center}
\begin{tikzpicture} [scale=0.5]
\fill[io] (0,0) -- (10, 0) -- (10,-2) -- ( 9,-2) -- ( 9,-2) -- ( 8,-2) -- ( 8,-3) -- ( 7,-3) -- ( 7,-4) -- ( 6,-4) -- ( 6,-4) -- ( 5,-4) -- ( 5,-6) -- ( 4,-6) -- ( 4,-6) -- ( 3,-6) -- ( 3,-8) -- ( 2,-8) -- ( 2,-10) -- ( 1,-10) -- ( 1,-10) -- ( 0,-10) -- cycle;
\draw[very thick]     (0,0) -- (0,-12) -- (14,-12) -- (14,0) -- cycle;
\draw[dotted]     (10,0) -- (10,-12);
\draw (0,-12) node [left=1pt,black] {$r$};
\draw (0,-2) node [left=1pt,black] {$s_\ell$};
\draw (0,-10) node [left=1pt,black] {$s_1$};
\draw (0,-6) node [left=1pt,black] {$s_\lambda$};
\draw (0,-3) node [left=1pt,black] {$s_i$};
\draw (0,0) node [above=1pt,black] {$1$};
\draw (3.5,0) node [above=1pt,black] {$\lambda$};
\draw (7.5,0) node [above=1pt,black] {$i$};
\draw (10,0) node [above=1pt,black] {$\ell$};
\draw (14,0) node [above=1pt,black] {$n$};
\draw  (3.5,-4) node {\Large $\pi^b_\lambda$};
\draw  (7.5,-7) node {\Large $\pi^a_i$};
\draw[very thick, ->,left] (7.0,-7) to node {\small $B^{a,\lambda}_{i,b}~$} (4,-4.3);
\draw[very thick, ->] (7.0,-7) -- (5,-1.5);
\draw[very thick, ->] (7.0,-7) -- (0.5,-7.5);
\draw[very thick, ->] (7.0,-7) -- (1.5,-8.5);
\draw[very thick, ->] (7.0,-7) -- (0.5,-9.5);
\draw[very thick, ->] (7.0,-7) -- (6.5,-2.5);
\end{tikzpicture}
\end{center}
\caption{A tableau in generic bases. Image from \cite{Smith2014a}.}
\label{fig:figtab}
\end{figure}
\begin{figure}
\begin{center}
\begin{tikzpicture} [scale=0.5]
\fill[io] (0,0) -- (10, 0) -- (10,-2) -- ( 9,-2) -- ( 9,-2) -- ( 8,-2) -- ( 8,-3) -- ( 7,-3) -- ( 7,-4) -- ( 6,-4) -- ( 6,-4) -- ( 5,-4) -- ( 5,-6) -- ( 4,-6) -- ( 4,-6) -- ( 3,-6) -- ( 3,-8) -- ( 2,-8) -- ( 2,-10) -- ( 1,-10) -- ( 1,-10) -- ( 0,-10) -- cycle;
\draw[very thick]     (0,0) -- (0,-12) -- (14,-12) -- (14,0) -- cycle;
\draw[dotted]     (10,0) -- (10,-12);
\draw (0,-12) node [left=1pt,black] {$r$};
\draw (0,-2) node [left=1pt,black] {$s_\ell$};
\draw (0,-10) node [left=1pt,black] {$s_1$};
\draw (0,-6) node [left=1pt,black] {$s_\lambda$};
\draw (0,-3) node [left=1pt,black] {$s_i$};
\draw (0,0) node [above=1pt,black] {$1$};
\draw (3.5,0) node [above=1pt,black] {$\lambda$};
\draw (7.5,0) node [above=1pt,black] {$i$};
\draw (10,0) node [above=1pt,black] {$\ell$};
\draw (14,0) node [above=1pt,black] {$n$};
\draw  (3.5,-3.5) node {\small $\Wu^-_\lambda$};
\draw  (7.5,-1.5) node {\small $\Wu^-_i$};
\draw[very thick] (7,0) -- (7,-12) -- (8,-12) -- (8,0) -- cycle;
\draw[very thick] (3,0) -- (3,-6) -- (4,-6) -- (4,0) -- cycle;
\draw[dotted] (7,-3) -- (8,-3);
\draw[very thick, ->] (4.0,-3.5) to node[above] {$0$} (6.9,-2);
\draw[very thick, ->] (4.0,-3.5) to node[below] {$\B^\lambda_i$ } (6.9,-4.5);
\end{tikzpicture}
\end{center}
\caption{The map $\B^\lambda_i$ for a tableau in generic bases.  Image from
\cite{Smith2014a}.}
\label{fig:figB}
\end{figure}
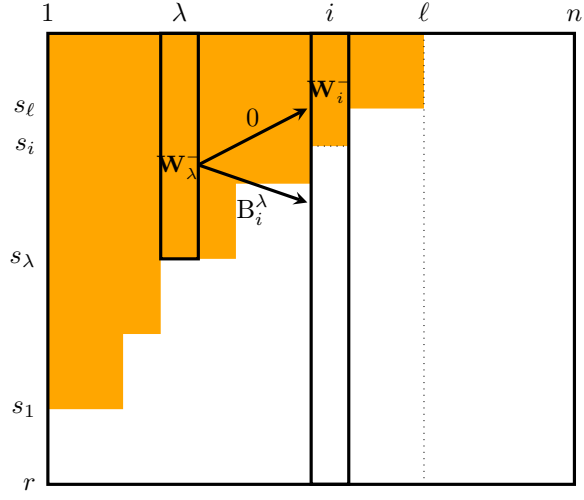

One can take the dual perspective, wherein the symbol coefficients
$B^{a,\lambda}_{i,b}$ define a 
map from the upper-left independent entries to the lower-right entries.
That is, consider the map 
\begin{equation}
\B \in V^* \otimes V \otimes W \otimes W^* \cong \End(V^*)
\otimes \End(W)\end{equation}
defined by 
\begin{equation}
\sum_{a \leq s_i} \delta^\lambda_i \delta^a_b (z_a \otimes z^b) \otimes (u^i
\otimes u_\lambda) +   
\sum_{a>s_i} B^{a,\lambda}_{i,b} (z_a \otimes z^b) \otimes (u^i \otimes
u_\lambda).
\label{eqn:IB}
\end{equation}
Equation~\eqref{eqn:IB} is the formal inclusion $A \to W \otimes V^*$ in the defining
exact sequence~\eqref{eqn:H1}.
By fixing $\varphi \in V^*$ and $v \in V$, we obtain an endomorphism
$\B(\varphi)(v):W \to W$ defined by the column relations of $(\pi^a_i)$, as in Figure~\ref{fig:figB}.
We use the shorthand $\B^\lambda_i$ for $\B(u^\lambda)(u_i)$, but note that
this is not quite the same as $B^{a,\lambda}_{i,b} z_a \otimes z^b$ because of
the identity term in Equation~\eqref{eqn:IB}.
That is, $\B^\lambda_\lambda = \sum_{a \leq s_\lambda} \delta^a_b (z_a \otimes z_b)$ for all $\lambda \leq \ell$.  

For the example
\eqref{eqn:example_tab_generic}--\eqref{eqn:example_sym_generic}, the maps
$\B^\lambda_i :W \to W$ are:
\begin{equation}
\begin{matrix}
  \B^1_1 = \begin{pmatrix} 1 & 0 & 0\\  0 & 1 & 0\\ 0 & 0 & 1 \end{pmatrix}
& \B^1_2 = \begin{pmatrix} 0 & 0 & 0\\  0 & 0 & 0\\ 1 & 0 & 0 \end{pmatrix}
& \B^1_3 = \begin{pmatrix} 0 & 0 & 0\\  1 & 0 & 0\\ 0 & 1 & 0 \end{pmatrix}\\
& \B^2_2 = \begin{pmatrix} 1 & 0 & 0\\  0 & 1 & 0\\ 0 & 0 & 0 \end{pmatrix}
& \B^2_3 = \begin{pmatrix} 0 & 1 & 0\\  0 & 0 & 0\\ 0 & 0 & 0 \end{pmatrix}.
\label{eqn:example_sym_array}
\end{matrix}
\end{equation}
So, if $\varphi = \varphi_i u^i \in V^*$ and $v = v^j u_j \in V$, the endomorphism
$\B(\varphi)(v) :W \to W$ is  
\begin{equation}
\B(\varphi)(v) = 
\begin{pmatrix} \varphi_1v^1 + \varphi_2v^2& \varphi_2v^3 & 0\\  \varphi_1v^3 &
\varphi_1v^1 +\varphi_2v^2& 0\\ \varphi_1v^2 & \varphi_1v^3 & \varphi_1v^1
\end{pmatrix}.
\label{eqn:example_sym_join}
\end{equation}

Using our generic basis $(u_i)$ for $V$ and its dual basis $(u^i)$
 for $V^*$, define decompositions $V = U \oplus Y$ and $V^* =
Y^\perp \oplus U^\perp$ using our index convention \eqref{eqn:index} as follows:
\begin{equation}
\begin{split}
V &= \pair{u_1,\ldots, u_{\ell}, u_{\ell+1}, \ldots,u_n} = \pair{u_i},\\
U &= \pair{u_1,\ldots, u_{\ell}\phantom{, u_{\ell+1}, \ldots,u_n}} =
\pair{u_\lambda},\\
Y &= \pair{\phantom{u_1,\ldots, u_{\ell},} u_{\ell+1}, \ldots,u_n} =
\pair{u_\varrho},
\end{split}
\label{eqn:VUY}
\end{equation}
and
\begin{equation}
\begin{split}
V^* &= \pair{u^1,\ldots, u^{\ell}, u^{\ell+1}, \ldots,u^n} = \pair{u^i},\\
U^* \cong Y^\perp &= \pair{u^1,\ldots, u^{\ell}\phantom{, u^{\ell+1}, \ldots,u^n}} =
\pair{u^\lambda},\\
Y^* \cong U^\perp &= \pair{\phantom{u^1,\ldots, u^{\ell},} u^{\ell+1}, \ldots,u^n} =
\pair{u^\varrho}.
\end{split}
\label{eqn:VUYdual}
\end{equation}
The isomorphisms $U^* \cong Y^\perp$ and $Y^* \cong U^\perp$ depend on the
basis; they are non-canonical but sometimes useful.

It is apparent from \eqref{eqn:IB} that $\B(\varphi) = \B(\tilde{\varphi})$ if
$\varphi - \tilde{\varphi} \in U^\perp$; that is if $\varphi_\varrho =
\tilde{\varphi}_\varrho$ for all $\varrho \geq \ell+1$, so we usually consider $\B(\varphi)$
only for $\varphi \in Y^\perp$.

Thus, in generic bases, we have a collection $\B^\lambda_i$ of endomorphisms of
$W$.
For our purposes of constructing a normal form, a ``good'' basis is one which
makes the endomorphisms $\B^\lambda_i$ as structurally similar as possible.
Section~\ref{sec:endo} imposes additional conditions on the images of these
endomorphisms for this purpose.

\subsection{Endovolutive Tableaux}\label{sec:endo}
Suppose $(u^i)$ and $(z_a)$ are generic bases for $A$.  For any $i$, define a decomposition
$W = \Wu^-_i \oplus \Wu^+_i$ by
\begin{equation}
\begin{split}
W &= \pair{z_1,\ldots, z_{s_i}, z_{s_i+1}, \ldots,z_r} = \pair{z_a}\\
\Wu^-_i &= \pair{z_1,\ldots, z_{s_i}\phantom{, z_{s_i+1}, \ldots,z_r}}\\ 
\Wu^+_i &= \pair{\phantom{z_1,\ldots, z_{s_i},} z_{s_i+1}, \ldots,z_r}
\end{split}
\end{equation}
By \eqref{eqn:IB}, the map $\B^\lambda_i:W \to W$ has support $\Wu^-_\lambda
\subset W$, and its image lies in $\Wu^+_i \subset W$. 

More generally, for any $\varphi \in V^*$, let
$\Wu^-(\varphi) = \Wu^-_{\underline{\lambda}}$ and 
$\Wu^+(\varphi) = \Wu^+_{\underline{\lambda}}$, 
where $\underline{\lambda}$
is the minimum index such that $\varphi_{\underline{\lambda}} \neq 0$.
(For general $\varphi$, we have $\dim \Wu^-(\varphi) = s_1$.)

A tableau $A$ expressed in bases $(u^i)$ and $(z_a)$ is called
\emph{endovolutive}\footnote{The term \emph{endovolutive} was coined in 
\cite{Smith2014a}, but the phenomenon was described previously in \cite[Chapter
IV\S5]{BCGGG}, \cite{Yang1987}, and it is certainly familiar to anyone who has
manipulated tableaux of linear Pfaffian systems.}
if $B^{a,\lambda}_{i,b} = 0$ for all $a > s_\lambda$.  That
is, endovolutive means that $\B^\lambda_i$ is an endomorphism of
$\Wu^-_\lambda \subset W$, as in Figure~\ref{fig:figBend}.

\begin{figure}
\begin{center}
\begin{tikzpicture} [scale=0.5]
\fill[io] (0,0) -- (10, 0) -- (10,-2) -- ( 9,-2) -- ( 9,-2) -- ( 8,-2) -- ( 8,-3) -- ( 7,-3) -- ( 7,-4) -- ( 6,-4) -- ( 6,-4) -- ( 5,-4) -- ( 5,-6) -- ( 4,-6) -- ( 4,-6) -- ( 3,-6) -- ( 3,-8) -- ( 2,-8) -- ( 2,-10) -- ( 1,-10) -- ( 1,-10) -- ( 0,-10) -- cycle;
\draw[very thick]     (0,0) -- (0,-12) -- (14,-12) -- (14,0) -- cycle;
\draw[dotted]     (10,0) -- (10,-12);
\draw (0,-12) node [left=1pt,black] {$r$};
\draw (0,-2) node [left=1pt,black] {$s_\ell$};
\draw (0,-10) node [left=1pt,black] {$s_1$};
\draw (0,-6) node [left=1pt,black] {$s_\lambda$};
\draw (0,-3) node [left=1pt,black] {$s_i$};
\draw (0,0) node [above=1pt,black] {$1$};
\draw (3.5,0) node [above=1pt,black] {$\lambda$};
\draw (7.5,0) node [above=1pt,black] {$i$};
\draw (10,0) node [above=1pt,black] {$\ell$};
\draw (14,0) node [above=1pt,black] {$n$};
\draw  (3.5,-3.5) node {\small $\Wu^-_\lambda$};
\draw  (7.5,-2) node {\small $\Wu^-_i$};
\draw  (8.5,-4.5) node[rotate=90] {\small $\Wu^+_i \cap \Wu^-_\lambda$};
\draw  (8.5,-9.5) node[rotate=90] {\small $\Wu^+_i \cap \Wu^+_\lambda$};
\draw[very thick] (7,-6) -- (7,-12) -- (8,-12) -- (8,-6) -- cycle;
\draw[very thick] (3,0) -- (3,-6) -- (4,-6) -- (4,0) -- cycle;
\draw[very thick] (7,0) -- (7,-6) -- (8,-6) -- (8,0) -- cycle;
\draw[dotted] (7,-3) -- (8,-3);
\draw[very thick, ->] (4.0,-3.5) to node[above] {$0$} (6.9,-2);
\draw[very thick, ->] (4.0,-3.5) to node[below] {$\B^\lambda_i$ } (6.9,-4.5);
\draw[very thick, ->] (4.0,-3.5) to node[below] {$0$} (6.9,-9);
\end{tikzpicture}
\end{center}
\caption{The map $B^\lambda_i$ for an endovolutive tableau.}
\label{fig:figBend}
\end{figure}
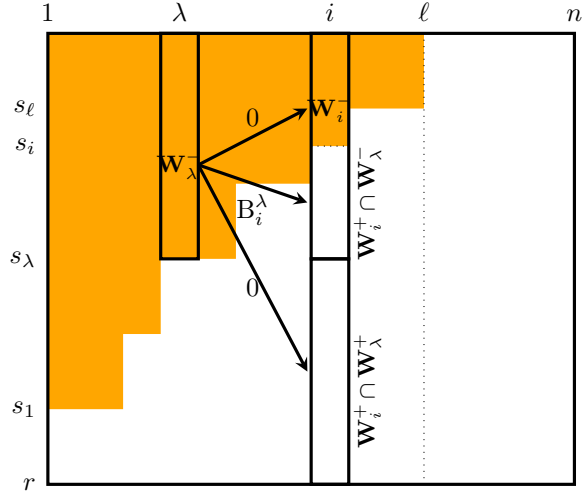

Note that the example \eqref{eqn:example_sym_array} is endovolutive because $s_2=2$ and 
$\B^2_2$ and $\B^2_3$ have non-zero entries only in the upper-left $2\times 2$ part.

In this way, when considering endovolutive tableaux, it useful to arrange the symbol
endomorphisms as an $\ell \times n$ array of $r \times r$ matrices:
\begin{equation}
\begin{bmatrix}
\B^1_1 & \B^1_2  & \B^1_3   & \B^1_4 & \cdots & \B^1_\ell & \cdots & \B^1_n \\
0      & \B^2_2  & \B^2_3   & \B^2_4 & \cdots & \B^2_\ell & \cdots & \B^2_n \\
0      & 0       & \B^3_3   & \B^3_4 & \cdots & \B^3_\ell & \cdots & \B^3_n \\
0      & 0       & 0        & \B^4_4 & \cdots & \B^4_\ell &  \cdots & \B^4_n\\    
       &         &          &        & \ddots &         & \B^\lambda_i & \vdots   \\
0      & 0       & 0        & 0      & 0      & \B^\ell_\ell & \cdots & \B^\ell_n 
\end{bmatrix}.
\label{eqn:bigB}
\end{equation}
Each ``diagonal'' entry $\B^\lambda_\lambda$ is the $r\times r$ matrix for which the
non-zero upper-left part is an $s_\lambda \times s_\lambda$ identity matrix,
$I_{s_\lambda}$.
Endovolutivity means that $\B^\lambda_i$, which is the $r \times r$ matrix in
row $\lambda$ of \eqref{eqn:bigB}, is zero outside the upper-left $s_\lambda
\times s_\lambda$ part.

If a tableau is endovolutive in
certain bases for $W$ and $V^*$, then it is also endovolutive under any
upper-triangular change-of-basis for $u^i \mapsto g^i_j u^j$.  Under such a
basis change, the columns of $(\pi^a_i)$ are linear combinations of the ones to
their right, and the sub-matrices in \eqref{eqn:bigB} change by the
corresponding conjugation.  Endovolutivity is a property of the flag
generated by the basis of $V^*$.

\subsection{Mutual Eigenvectors and Rank}
For endovolutive bases, each $\B^\lambda_i$ is an endomorphism of a particular vector space, 
so it is sensible to consider an eigenvector problem for these maps:
For any $\lambda$, let 
\begin{equation}
\Wu^1(u^\lambda) = \left\{ 
w \in \Wu^-_\lambda ~:~ \B^\lambda_\mu w = \delta^\lambda_\mu w, \quad \forall \mu 
\leq \ell\right\}.
\end{equation}
That is, we want to find the vectors that are preserved by $\B^\lambda_\lambda
= I_{s_\lambda}$
but are annihilated by all $\B^\lambda_\mu$ for $\mu \neq \lambda$.
More generally, let
\begin{equation}
\Wu^1(\varphi) = \left\{ 
w \in \Wu^-(\varphi) ~:~ 
\left( \sum_\lambda \varphi_\lambda \B^\lambda_\mu - \varphi_\mu I  \right) w =
0,\quad 
 \forall \mu
\leq \ell
\right\}.
\label{eqn:Wup}
\end{equation}

Equation~\eqref{eqn:Wup} can be rewritten as a mutual eigenvector
problem on the $\ell$ endomorphisms $\B(\varphi)(u_1)$, \ldots, $\B(\varphi)(u_\ell)$:
\begin{equation}
\Wu^1(\varphi) = \left\{ 
w \in \Wu^-(\varphi) ~:~ 
 \B(\varphi)(u_\mu)\,w = \varphi_\mu w,\quad  \forall \mu
\leq \ell
\right\}.
\label{eqn:Wupeig}
\end{equation}
Alternatively, because $B^\mu_\mu = I_{s_\mu}$, equation~\eqref{eqn:Wup} says that $\B(\varphi)(\cdot)w$ is rank-1 when
restricted to $Y^\perp$, so we can rewrite it as 
\begin{equation}
\Wu^1(\varphi) = \left\{ 
    w \in \Wu^-(\varphi) ~:~
w \otimes \varphi  + J^a_\varrho(z_a \otimes u^\varrho) \in A_e,\quad \exists
J \in W \otimes U^\perp
\right\}.
\label{eqn:Wupinv}
\end{equation}
This space is the focus of \cite{Guillemin1968}, and it plays an important part
in our story.   Unlike $\Wu^-(\varphi)$, its definition
does not depend on the basis; its definition depends only on the splitting $V =
U \oplus Y$.
Its dimension is an important invariant.
\begin{lemma}
Suppose that the tableau $A$ admits endovolutive bases.
For generic $\varphi$, $\dim \Wu^1(\varphi) = s_\ell$.
\label{lem:dimW1}
\end{lemma}
Lemma~\ref{lem:dimW1} is the result of a quick rank computation 
using \eqref{eqn:Wup}--\eqref{eqn:Wupeig}.  See \cite{Smith2014a}.

Our ``good'' basis and normal form will be built on the requirement that the maps $\B^\lambda_i$ commute on
certain combinations of these spaces (and thus the maps share generalized eigenspaces
and Jordan-block normal form there).  That is, we are aiming for
something like the commutative subalgebras seen in \cite{Gerstenhaber1961} and
\cite{Guralnick2000}.
Endovolutivity allows surprisingly
direct computation of the desired conditions.
For more detail on
endovolutivity, see \cite{Smith2014a} and the references therein.  We return to
this topic in Section~\ref{sec:edsinv}, but before that we must introduce the
geometry of subspaces.

\section{Grassmann and Universal Bundles}\label{sec:grass}
The Grassmann variety is the set $\Gr_n(X)$ of $n$-planes in an
$(n{+}r)$-dimensional vector space $X$.  It is a smooth projective variety and
a smooth manifold of dimension $nr$.  An $n$-plane $e\in \Gr_n(X)$ is called an
\emph{element}.

\subsection{Tangent and Arctangent}\label{sec:tan}
Depending on one's favorite notation, there are sev\-eral ways to see that the
tangent space of $\Gr_n(X)$ at $e$ is $(X/e) \otimes e^*$.

First, for any $e \in \Gr_n(X)$, choose a basis
$(u_i) = (u_1,
\ldots, u_n)$ for $e$, and choose $(z_a) = (z_1, \ldots, z_r)$ so as to complete a
basis of the entire vector space $X$.
Any $n$-plane $\tilde{e}$ near $e$ admits a basis $(\tilde{u}_i) =
(\tilde{u}_1, \ldots, \tilde{u}_n)$ that
we may assume is related by a matrix in reduced column-echelon form:
\begin{equation}
\begin{pmatrix}
\tilde{u}_1 & \ldots & \tilde{u}_n
\end{pmatrix}
=
\begin{pmatrix}
u_1 & \ldots & u_n & z_1 & \cdots & z_r
\end{pmatrix}
\begin{pmatrix}
1 &  \cdots & 0\\
  &  \ddots &  \\
0 & \cdots & 1\\
K^1_1 & \cdots & K^1_n\\
\vdots & \ddots & \vdots\\
K^r_1 & \cdots & K^r_n
\end{pmatrix}.
\end{equation}
More succinctly, using the summation convention:
\begin{equation}
\tilde{u}_i = u_i + z_a K^a_i = u_j \delta^j_i + z_a K^a_i.
\label{eqn:elementbasis}
\end{equation}
That is, $(\tilde{u}_i)$ and $(u_i)$ are related by an $(n+r)\times n$ matrix of rank
$n$ whose image $\pair{\tilde{u}_1, \ldots, \tilde{u}_n} = \tilde{e}$ is determined uniquely by the
$r\times n$ submatrix $(K^a_i)$.  
In this sense, $T_e \Gr_n(X)$ is isomorphic
to the space of $r \times n$ matrices, which is isomorphic to $(X/e)\otimes
e^*$.
This is easy and computational, but this isomorphism is not \emph{canonical} for an
abstract vector space (without metric) because it depends on a choice of
splitting $X = e \oplus (X/e)$ by choosing the complementary basis $(z_a)$.

Alternatively, to see $T_e \Gr_n(X) = (X/e) \otimes e^*$ and avoid splitting, we can use the dual\footnote{Recall that $(X/e)^*$ is canonically
isomorphic to $e^\perp$: if $[v]= \{ u + e \} \in X/e$, then
$\varphi([v]) = \varphi(v) + 0$ is well-defined for all $\varphi \in e^\perp$.} short-exact sequences
\begin{equation}
\begin{split}
0 \to e \to X \to X/e \to 0,\\
0 \to e^\perp \to X^* \to e^* \to 0.
\end{split}
\label{eqn:shortexact}
\end{equation}
Choose any basis $(\theta^a) = (\theta^1, \ldots, \theta^r)$ of the annihilator space
$e^\perp = (X/e)^*$, and let $(z_a) = (z_1, \ldots, z_r)$ be the corresponding dual
basis of $(X/e)$.  Then, we may take the coefficients $K^a_i$ of 
\begin{equation}
K = z_a \otimes
K^a_i(\tilde{e}) = z_a \otimes \theta^a(\tilde{u}_i) \in (X/e) \otimes e^*
\end{equation} as
$nr$ coordinates on $T_e \Gr_n(X)$; that is, $K^a_i$ gives a basis of
$T^*_e\Gr_n(X)$.

More abstractly, an explicit choice of bases $(u_i)$ for $e$ and $(\theta^a)$ for
$e^\perp$ is unnecessary.  Instead, we need only the abstract homomorphism $K
\in (X/e) \otimes e^*$, because the space\footnote{Note that $v+K(v)$ is not
well-defined in $X$ for any particular $v \in e$, but the span over all such $v$ is
well-defined.} 
\begin{equation}
\tilde{e} = \pair{\tilde{u}_i} = \pair{u_i + K(u_i)}
= \pair {v + K(v) ~:~ v\in e}
\end{equation}
is invariant
under $GL(n)$ transformations on $(u_i)$ and $(\tilde{u}_i)$ as well as $GL(r)$
transformations on $\theta$.  That is, $\tilde{e}$ is the ``graph'' of $v
\mapsto v + K(v)$ over all $v \in e$.

As in Figure~\ref{fig:arctan}, the derivative map
$\Gr_n(X) \to (X/e) \otimes e^*$ near $e$ is a multidimensional
generalization of the tangent function, so the inverse map\footnote{The 
map $\arctan_e$ is analogous to exponential map $\exp_p:T_pM \to M$ from
Riemannian geometry or Lie group representation theory, except that 
this description of $\arctan_e$
does not make explicit use of a metric or group structure.}
is written 
\begin{equation}
\arctan_e: (X/e) \otimes e^* \to \Gr_n(X).
\end{equation}

The reader is encouraged to read \cite[\S5]{Milnor1974} and
\cite{Kobayashi1963} and to search for the terms \emph{Pl\"ucker embedding} and
\emph{Stiefel manifold} for more detail on this subject.

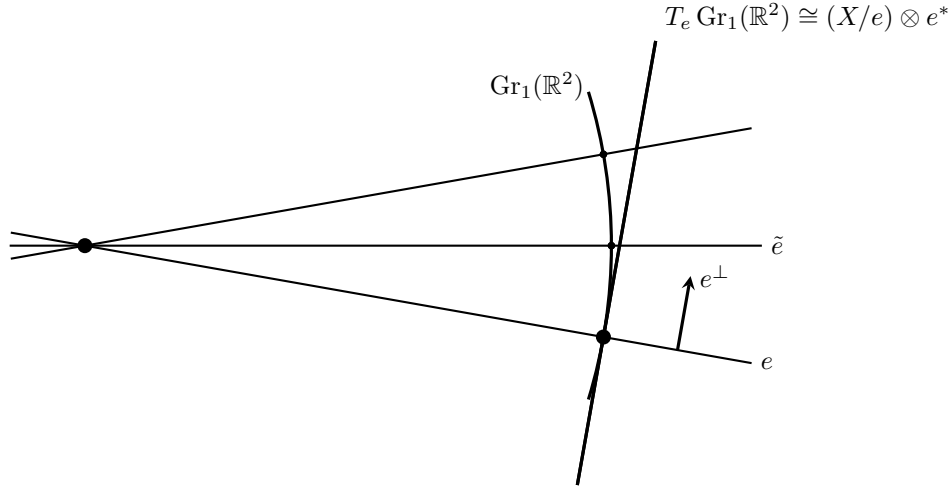
\begin{figure}
\begin{tikzpicture}[scale=2]
\fill (0, 0) circle (0.5mm);
\fill ({3.5*cos(-10)}, {3.5*sin(-10)}) circle (0.5mm);
\fill ({3.5*cos(0)}, {3.5*sin(0)}) circle (0.25mm);
\fill ({3.5*cos(10)}, {3.5*sin(10)}) circle (0.25mm);
\draw[very thick] ({3.5*cos(-10) + 1*sin(-10)}, {3.5*sin(-10)-1*cos(-10)}) -- ({3.5*cos(-10) - 2*sin(-10)}, {3.5*sin(-10) + 2*cos(-10)});
\fill ({3.5*cos(0)}, {3.5*sin(0)}) circle (0.25mm);
\fill ({3.5*cos(10)}, {3.5*sin(10)}) circle (0.25mm);
\draw[very thick]  ({3.5*cos(-17)}, {3.5*sin(-17)}) arc (-17:17:3.5); 
\draw[thick] ({-0.5*cos(-10)}, {-0.5*sin(-10)}) -- ({4.5*cos(-10)}, {4.5*sin(-10)});
\draw[thick] ({-0.5*cos( 10)}, {-0.5*sin( 10)}) -- ({4.5*cos( 10)}, {4.5*sin( 10)});
\draw[thick] (-0.5,0) -- (4.5,0);
\draw[very thick] ({3.5*cos(-10) + 1*sin(-10)}, {3.5*sin(-10)-1*cos(-10)}) -- ({3.5*cos(-10) - 2*sin(-10)}, {3.5*sin(-10) + 2*cos(-10)});
\node[above left] at ({3.5*cos(15)}, {3.5*sin(15)}) {$\Gr_1(\mathbb{R}^2)$};
\node[above right] at ({3.5*cos(-10) - 2*sin(-10)}, {3.5*sin(-10)+2*cos(-10)})
{$T_e \Gr_1(\mathbb{R}^2) \cong (X/e) \otimes e^*$};
\node[right] at ({4.5*cos(-10)}, {4.5*sin(-10)}) {$e$};
\node[right] at (4.5, 0) {$\tilde{e}$};
\draw[very thick, ->] ({4*cos(-10)}, {4*sin(-10)}) -- ({4*cos(-10) - 0.5*sin(-10)}, {4*sin(-10) + 0.5*cos(-10)});
\node[right] at ({4*cos(-10) - 0.5*sin(-10)}, {4*sin(-10) + 0.5*cos(-10)}) {$e^\perp$};
\end{tikzpicture}
\caption{From $e$, identify a nearby line $\tilde{e}$ in $\mathbb{R}^2$ with a relative angle.  The map from $T_e\Gr_1(\mathbb{R}^2) \cong (-\infty, \infty)$ to
the neighborhood $(-\pi/2, \pi/2)$ of $e$ in $Gr_1(\mathbb{R}^2)$ is $\arctan_e$.
Its inverse is $\tan_e$.}\label{fig:arctan}
\end{figure}

\begin{rmk}
Notice that any linear subspace of $(X/e) \otimes e^*$ is a tableau in the sense of
Section~\ref{sec:tableau}.
In some sense, it is the \emph{only} example, as arbitrary $V$ and $W$ could be
studied by setting $X = V \oplus W$ and $e = V+0$.
Moreover, any smooth submanifold $Z \subset \Gr_n(X)$ with tangent space $T_e Z
\subset T_e \Gr(X)$ at $e \in Z$ gives $T_e Z$ as a tableau in $(X/e) \otimes e^*$.
This observation is the heart of the entire subject of exterior differential
systems, and it reappears forcefully in Section~\ref{sec:eds}.
\label{rmk:subspacesaretableaux}
\end{rmk}

\subsection{Polar pairs}\label{sec:polarpairs}
The purpose of this subsection is to establish two results,
Lemmas~\ref{lem:ranknullity} and \ref{lem:connected}, that tie the 
algebraic geometry of intersecting subspaces to the differential geometry of the Grassmannian.
These lemmas are used in
Part~\ref{part:Xi} to demonstrate the correspondence between the characteristic variety (in the Cauchy problem
of a system of PDEs) and the rank-1 variety (of the tableau of an EDS) in
Lemma~\ref{lem:AseesXi}, thus
providing the foundation of the geometric theory of PDEs. 

Suppose that $e, \tilde{e} \in \Gr_n(X)$, and that they intersect along a
hyperplane.
That is, suppose $e' = e \cap \tilde{e}$  and $\dim e' = n-1$.  We call the
pair of $n$-planes
$e$ and $\tilde{e}$ a \emph{polar pair} because they are both polar
extensions\footnote{This is a classical terminology that reappears in
Section~\ref{sec:polarXi}.} of $e'$.  
For any $e \in \Gr_n(X)$, let  
\begin{equation}
\Pol_1(e) = \{ \tilde{e} \in \Gr_n(X) ~:~
\dim(\tilde{e} \cap e) = n-1\}.
\end{equation}
We say $\tilde{e} \in \Pol_1(e)$ is a \emph{polar pair of $e$}.
This relationship is symmetric---hence the unqualified term polar \emph{pair}---as
$\tilde{e} \in \Pol_1(e)$ if and only if $e \in \Pol_1(\tilde{e})$, but this
relationship is not an equivalence relation, as it fails both reflexivity and
transitivity.  

\begin{figure}
\begin{tikzpicture}[scale=0.7]
\draw (-3,0) -- (3,0) -- (4.1, 1.5) --  (1.5,1.5);
\draw (0,-3) -- (0,3) -- (1.5, 4.1) --  (1.5,1.5);
\draw (0,-3) -- (1.5,-1.1) -- (1.5, 0);
\draw (-3,0) -- (-1.1,1.5) -- (0,1.5);
\draw[thick] (-1, -1) -- (2,2);
\draw[->] (0.75,0.75) -- (1.75,0.75);
\draw[->] (0.75,0.75) -- (0.75,1.75);
\node[above] at (1.75,0.75) {\small $v$};
\node[right] at (0.75,1.75) {\small $\tilde{v}$};
\node[above left] at (3,0) {$e$};
\node[below right] at (0,3) {$\tilde{e}$};
\node[left] at  (-1,-1) {$e'$};
\end{tikzpicture}~\includegraphics[width=0.5\textwidth]{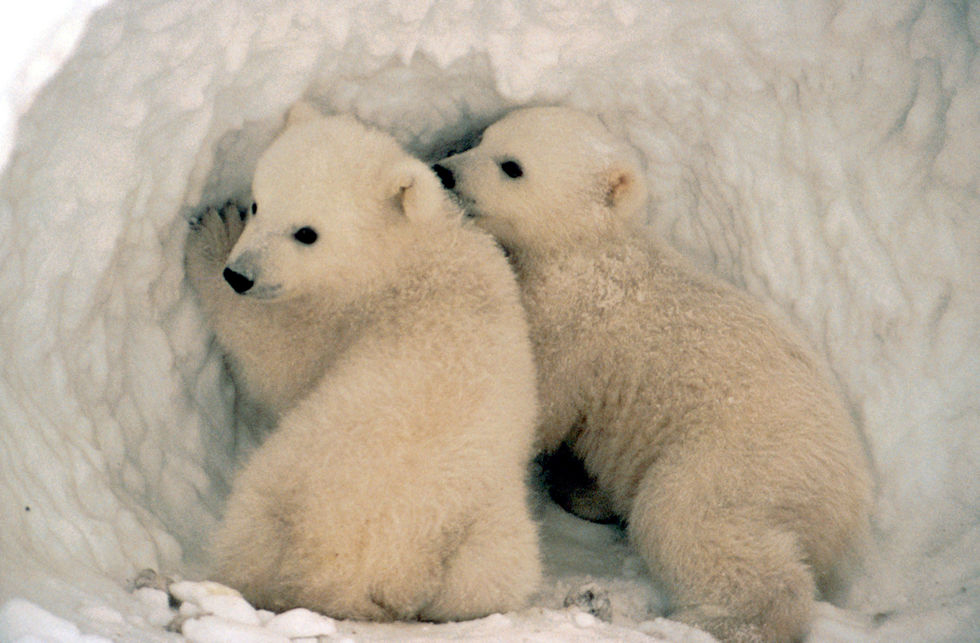}
\caption{Polar pairs.}\label{fig:pp}
\end{figure}
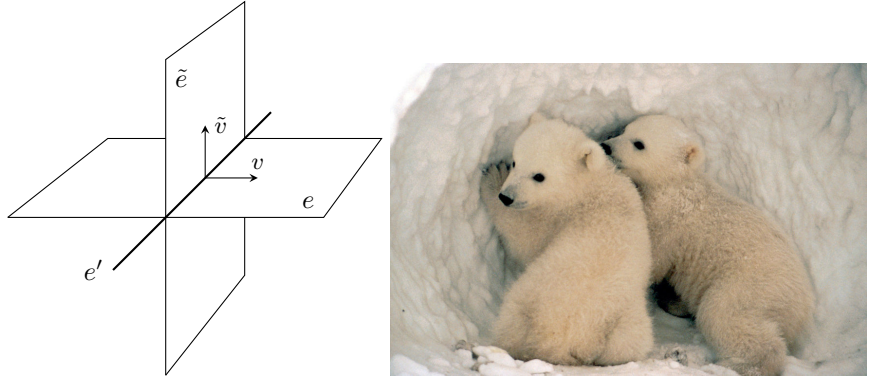

Within the image of $\arctan_e$, 
Lemma~\ref{lem:polarpairs1} ties the notion of polar pairs
to lines in the tangent space $T_e \Gr_n(X)$, 

\begin{lemma}
Suppose $e \in \Gr_n(X)$ and $\tilde{e} = \arctan_e(K)$ for $K \in (X/e)
\otimes e^*$.
Then $\rank(K) = 1$ if and only if $\tilde{e} \in \Pol_1(e)$.
\label{lem:polarpairs1}
\end{lemma}
\begin{proof}
Suppose that $\tilde{e} \in \Pol_1(e)$.
Let $e' = e \cap \tilde{e}$, so $\dim e' = n-1$.
Let $(u_1, \ldots, u_{n-1})$ be a basis for $e'$, and
extend that basis to a basis $(u_1, \ldots, u_{n-1}, v)$ for $e$ and to a basis
$(u_1, \ldots, u_{n-1}, \tilde{v})$ for $\tilde{e}$.  Writing
\eqref{eqn:elementbasis} in this case, it is apparent that only the
$n$th column of $\left(K^a_i\right)$ is nonzero.
That is, the tangent homomorphism $K \in (X/e) \otimes e^*$ is rank-1.  (It
cannot be the degenerate rank-0 unless $e = \tilde{e}$.)

Conversely, suppose that $K \in (X/e) \otimes e^*$ is rank-1.  Let $e' = \ker K
\subset e$, which is a subspace of $e$ of dimension $n{-}1$.  Any line in $e'$
is preserved by the map $e \to X$ defined by the matrix in
\eqref{eqn:elementbasis}; hence, the subspace $e'$ is also a subspace of
$\tilde{e}$.  (It cannot be the degenerate case  $e=\tilde{e}$ unless $K$ is
rank-0.)
\end{proof}

The concept of polar pairs generalizes to
co-dimensions $k$ other than $1$. 
For any $e \in \Gr_n(X)$, let
\begin{equation}
\Pol_k(e) = \{ \tilde{e} \in \Gr_n(X) ~:~
\dim(\tilde{e} \cap e) = n-k\}.
\end{equation}
Because $\dim X = n+r$ and $\dim e = n$, the set $\Pol_k(e)$ is
nonempty if and only if $k \leq r$, because $n+k = \dim(e+ \tilde{e}) \leq
n+r$.  The definition is trivial and fairly useless for $k=0$.
Again, the $k$-polar-pair relationship $\tilde{e} \in \Pol_k(e)$ is
symmetric but neither reflexive nor transitive for the interesting case $0<k\leq r$.

One can see immediately that Lemma~\ref{lem:polarpairs1} generalizes by replacing $1$ with any
rank $k$ to give Lemma~\ref{lem:polarpairs}.
\begin{lemma}
Suppose $e \in \Gr_n(X)$ and $\tilde{e} = \arctan_e(K)$ for $K \in (X/e)
\otimes e^*$.
Then $\rank(K) = k$ if and only if $\tilde{e} \in \Pol_k(e)$.
\label{lem:polarpairs}
\end{lemma}

Next, we can generalize Lemma~\ref{lem:polarpairs} to
Lemma~\ref{lem:ranknullity} by 
dropping the use of $\arctan$.  That is, we can consider a $k$-polar pair
$(e,\tilde{e})$ where $\tilde{e}$ lies outside the open image of $\arctan_e$.
From an algebraic perspective, Lemma~\ref{lem:ranknullity} can be seen as a Grassmannian version of the
rank-nullity theorem.  
Phrased in other ways, it is popular true/false homework question in
undergraduate linear algebra textbooks.

\begin{lemma}
Fix $e \in \Gr_n(X)$ and $\tilde{e} \in \Pol_k(e)$.
The canonical maps $\tilde{e} \mapsto \tilde{e}/e$ and $\tilde{e} \mapsto
(\tilde{e} \cap e)^\perp/e^\perp$ both have rank-$k$ images,  
yielding 
the incidence correspondence  
Figure~\ref{fig:polarincidence}.
\label{lem:ranknullity}
\end{lemma}
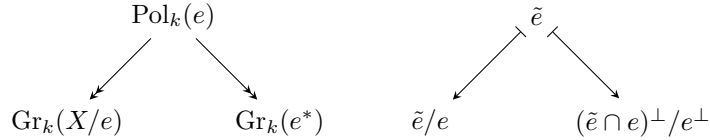
\begin{figure}[ht]
\centering\begin{tikzpicture}[node distance=2cm, auto]
  \node (Pol) {$\Pol_k(e)$};
  \node[below left of=Pol] (Xe) {$\Gr_k(X/e)$};
  \node[below right of=Pol] (ep) {$\Gr_k(e^*)$};
  \draw[->>, above right] (Pol) to node {} (Xe);
  \draw[->>, above left] (Pol) to node {} (ep);
  \node[right of=ep] (Xe1) {$\tilde{e}/e$};
  \node[above right of=Xe1] (Pol1) {$\tilde{e}$};
  \node[below right of=Pol1] (ep1) {$(\tilde{e}\cap e)^\perp/e^\perp$};
  \draw[|->, above right] (Pol1) to node {} (Xe1);
  \draw[|->, above left] (Pol1) to node {} (ep1);
\end{tikzpicture}
\caption{The incidence correspondence of polar pairs $e$ and $\tilde{e}$.}\label{fig:polarincidence}
\end{figure}

\begin{proof}
Let $e' = e \cap \tilde{e}$. 
Consider the short-exact sequences~\eqref{eqn:shortexact}, and apply the
rank-nullity theorem of those maps on $e'$, which has dimension $n-k$.
In the first short-exact sequence, $\tilde{e}/e = e'/e$ has dimension $k$ as a
subspace of $X/e$.
In the second short-exact sequence, 
the space $\left(\tilde{e}\cap e\right)^\perp/e^\perp =
(e')^\perp/e^\perp$ has dimension $k$ as a subspace of $e^*=X^*/e^\perp$.
In both cases, and such subspace can be constructed this way.
\end{proof}

Now, reconsider the case $k=1$ in light of Lemma~\ref{lem:ranknullity}. Then
each $\tilde{e} \in \Pol_1(e)$ yields a hyperplane $e' = \tilde{e} \cap e$. The
right image  $(e')^\perp / e^\perp$ in Figure~\ref{fig:polarincidence} is some
line $[\xi] \in \mathbb{P}e^*$.  The left image $\tilde{e}/e$ is some line $[w]
\in \mathbb{P}(X/e)$.  So, each $\tilde{e} \in \Pol_1(e)$ yields a rank-1
projective homomorphism $[w] \otimes [\xi] = [w \otimes \xi] \in
\mathbb{P}\left( (X/e) \otimes e^*\right)$.  Any element of $\mathbb{P}\left(
(X/e) \otimes e^*\right)$ could be obtained this way by appropriate choice of
$\tilde{e}$.

To see how this generalizes Lemma~\ref{lem:polarpairs1}, 
let us write $[w] \otimes [\xi]$ explicitly. 
Let \[(\omega^1, \ldots \omega^n, \theta^1, \ldots, \theta^r)\] be a basis for
$X^*$ such that $e = \ker \{ \theta^1, \ldots, \theta^r\}$ and 
$e' = \ker \{ \theta^1, \ldots, \theta^r, \xi\}$ for some $\xi = \xi_i
\omega^i$.   Then, $\tilde{e} = \ker\{ \tilde{\theta}^1, \ldots,
\tilde{\theta}^r\}$ for some $\tilde{\theta}^a = J^a_b \theta^b + K^a_i
\omega^i$.  Because $e' \subset \tilde{e}$, it must be that 
\begin{equation}
\begin{split}
\tilde{\theta}^a &\equiv 0\ \mod\{\theta^c, \xi\},\ \text{so}\\
J^a_b \theta^b + K^a_i \omega^i &\equiv 0\ \mod\{ \theta^c, \xi\},\ \text{so}\\
K^a_i \omega^i &\equiv 0\ \mod\{ \theta^c, \xi\},\ \text{so}\\
K^a_i \omega^i &\equiv 0\ \mod\{ \xi\}.
\end{split}
\end{equation}
Hence, each $K^a_i \omega^i$ is a multiple of $\xi$; call it $w^a \xi$.  
(Note
that $w^a=0$ for all $a$ if and only if $\tilde{e}=e$, which contradicts our
assumption $\dim e' = n-1$.)
We can
use this fact to build a rank-1 homomorphism:  Let $(z_a)$ be the basis of $X/e$ dual to $(\theta^a)$.  Let $(\omega^i)$ also denote
the basis of $e^*=X^*/e^\perp$ induced by $\omega^i \in X^*$, so that $\xi \in
e^*$ also
denotes the image of $\xi \in X^*$.
Let $w = w^a z_a$.
Then the induced homomorphism 
\begin{equation}
K = z_a \otimes K^a_i \omega^i = z_a \otimes w^a \xi = w \otimes \xi \in (X/e) \otimes e^*
\end{equation}
is rank-1.   Each of $w$ and $\xi$ is well-defined up to scale, so $K$ is
well-defined up to scale, yielding $[K] = \mathbb{P} \left((X/e)\otimes e^*\right)$. 

It may be that $\tilde{e}$ lies outside the open image of $\arctan_e$.
How then do we interpret $K$? Is there any relationship between $\tilde{e}$ and
$\arctan_e(K)$?  
From a differential geometric perspective, this is reminiscent of the failure of injectivity at large distances
for the exponential map in Riemannian geometry.
Lemma~\ref{lem:connected} shows that for any polar pair $\tilde{e}$ of $e$,
either $\tilde{e}$ lies in the curve $\arctan_e([K])$ or is the limit of the
curve.

\begin{lemma}
Suppose $e \in \Gr_n(X)$ and $\tilde{e} \in \Pol_1(e)$.  Then there is a
continuous path $\{e_\tau : 0 \leq 0 \leq 1\}$ in $\Gr_n(X)$
such that $e_0=e$, $e_1=\tilde{e}$, and 
$e_\tau \cap e = \tilde{e} \cap e = e_\tau \cap
\tilde{e}$ for all $0 < \tau < 1$.  
The rank-1 line $[K]$ induced by $e_\tau$ via from
Lemma~\ref{lem:ranknullity} is constant across $0<\tau\leq 1$.  Moreover, 
$e_\tau \in \arctan_e([K])$ for $0 \leq \tau < 1$.
\label{lem:connected}
\end{lemma}

\begin{proof}
Let $e' = e \cap \tilde{e}$.
For some independent vectors $v, w \in X$, we may write $e = e' + \pair{v}$ and $\tilde{e} = e'+
\pair{w}$ and define\footnote{If preferred, one can reparametrize from a linear interpolation to a circular interpolation by replacing
$\tau$ with $\cos\vartheta$
and $1-\tau$ with $\sin\vartheta$ for some angle $0 \leq \vartheta \leq \pi/2$.}
a curve from $e$ to $\tilde{e}$ in $\Gr_n(X)$ by 
\begin{equation}
e_\tau = e' + \pair{(1-\tau)v +\tau w)},\quad 0 \leq \tau \leq 1.
\label{eqn:etau}
\end{equation}
Note that $e' = e \cap  e_\tau = \tilde{e}\cap e_\tau$ for all $0<\tau<1$.
It is apparent from \eqref{eqn:etau} that 
$e_\lambda/e$ is the line $[\tau w] = [w]$, which is constant versus $\tau$.
It is also apparent from \eqref{eqn:etau}
that 
$(e_\tau \cap e)^\perp/e^\perp = (e')^\perp / e^\perp$ is the  line $[\xi]$,
which 
is constant versus $\tau$. 
Hence, all such $e_\tau$ have the same representative
rank-1 homomorphism, $[w \otimes \xi] = [K]$ in Lemma~\ref{lem:ranknullity}.

It may be that $\tilde{e}=e_1$ lies outside the
open image of $\arctan_e$.  However, comparison of \eqref{eqn:etau} and
\eqref{eqn:elementbasis} implies that  
all $e_\tau$ lie inside the image of $\arctan_e$ for all $\tau < 1$.
So, the image $\arctan_e([w \otimes \xi])$ contains
an open set of $\{e_\tau\}$ where $e_\tau \cap e = e'$.  
\end{proof}

Consider the example summarized in
Figure~\ref{fig:pp}, where
\begin{equation}
e = \pair{ \begin{pmatrix} 1 \\ 0 \\ 0\end{pmatrix}, 
 \begin{pmatrix} 0 \\ 1 \\ 0\end{pmatrix}},\ \text{and}\ 
\tilde{e} = \pair{ \begin{pmatrix} 1 \\ 0 \\ 0\end{pmatrix}, 
 \begin{pmatrix} 0 \\ 0 \\ 1\end{pmatrix}}.
\end{equation}
Note that $\tilde{e}$ is outside the open image of $\arctan_e$ because 
\eqref{eqn:elementbasis} breaks down as written in this basis.
But, $e_\tau$ is the family obtained by rotating from $e$
toward $\tilde{e}$ about the axis $e'$ through an angle
$\arctan(\frac{\tau}{1-\tau})$, which varies from $0$ to $\frac\pi2$.  
For all $0 \leq \tau < 1$, we have 
\begin{equation}
e_\tau = \pair{ \begin{pmatrix} 1 \\ 0 \\ 0\end{pmatrix}, 
 \begin{pmatrix} 0 \\ 1-\tau \\ \tau\end{pmatrix}}
=
\pair{ \begin{pmatrix} 1 \\ 0 \\ 0\end{pmatrix}, 
 \begin{pmatrix} 0 \\1  \\ \frac{\tau}{1-\tau}\end{pmatrix}}.
\end{equation}
Thus, the line of rank-1 matrices $[K]$ in $(X/e)\otimes e^*$ is written as  $[\begin{pmatrix} 0 &
1\end{pmatrix}]$ in this basis.  This line represented by every $e_\tau$ in 
a curve that converges to $\tilde{e}$ as $\tau \to 1$.
Indeed, up to a choice of basis, this is essentially the only example.

Overall, we have learned that any $k=1$ polar pair in $\Gr_n(X)$ is represented by
a line of rank-1 matrices in the tangent space, and vice-versa.
This is sufficient for our purposes, but those seeking a more detailed
understanding of polar pairs are encouraged to investigate \emph{Schubert
varieties}---for example in \cite{Robles2012a}---and the other outgrowths of
Hilbert's 15th problem.

\subsection{The Tautological Bundle}\label{sec:tautbundle}
Soon, we will consider algebraic equations defined on $e^*$.
To facilitate this, for any $e \in \Gr_n(X)$, we consider the complexified
projective space $\mathbb{X} = \mathbb{P}X \otimes \mathbb{C}$ and its subspace
$\mathbb{P}e \otimes \mathbb{C}$.    For standard complex projective space, we
write $\mathbb{P}^d$ for $\mathbb{CP}^{d} = \mathbb{P}(\mathbb{C}^{d+1})$.  That
is, $\mathbb{X} \cong \mathbb{P}^{n+r-1}$, and $\mathbb{P}e \otimes \mathbb{C}
\cong \mathbb{P}^{n-1}$.

If we consider all such spaces across all $e$
simultaneously, we obtain the \emph{tautological bundle}\footnote{
These are also called universal bundles or canonical bundles. They are analogous to
the sheaves $\mathscr{O}(-1)$ and $\mathscr{O}(1)$, respectively, for
varieties in projective space.}  $\taut$ over $\Gr_n(X)$ with fiber
\begin{equation} 
\taut_e  = \mathbb{P}e \otimes \mathbb{C},\quad \forall e \in \Gr_n(X),
\end{equation}
and its dual bundle $\taut^*$ over $\Gr_n(X)$ with fiber 
\begin{equation}
\taut^*_e  = \mathbb{P}e^* \otimes \mathbb{C}, \quad \forall e \in \Gr_n(X),
\end{equation}
and its annihilator bundle $\taut^\perp$ over $\Gr_n(X)$ with fiber
\begin{equation}
\taut^\perp_e  = \mathbb{P}e^\perp \otimes \mathbb{C}, \quad \forall e \in \Gr_n(X),
\end{equation}
and its cokernel bundle $\mathbb{X}/\taut$ over $\Gr_n(X)$ with fiber
\begin{equation}
(\mathbb{X}/\taut)_e  = \mathbb{P}(X/e) \otimes \mathbb{C}, \quad \forall e \in \Gr_n(X).
\end{equation}
See Figure~\ref{fig:tautological}.

There is a dual pair of short exact sequences of projective bundles, analogous to
\eqref{eqn:shortexact}.
\begin{equation}
\begin{split}
0 \to \taut_e \to \mathbb{X} \to (\mathbb{X}/\taut)_e \to 0,\\
0 \to \taut^\perp_e \to \mathbb{X}^* \to \taut^*_e \to 0.
\end{split}
\label{eqn:shortexact2}
\end{equation}
Hence, the complex projectivized tangent bundle $\mathbb{P}T\Gr(X) \otimes \mathbb{C}$ is isomorphic (canonically) to
$(\mathbb{X}/\taut) \otimes \taut^*$.
If we choose a splitting of these sequences, then we can use the dual bases to
establish a (non-canonical) decomposition $\mathbb{P}X \otimes \mathbb{C} \cong \taut_e \oplus
(\mathbb{X}/\taut)_e$ for any $e$.

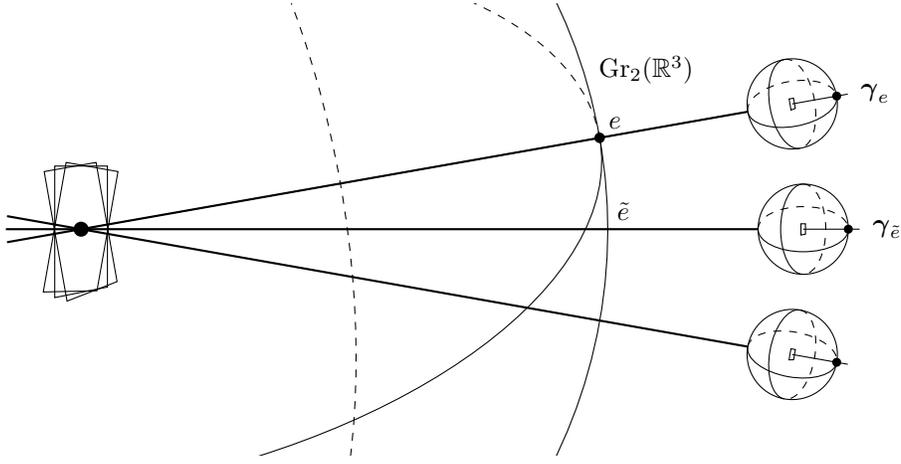
\begin{figure}
\begin{tikzpicture}[scale=2]
\clip (-0.5,1.5) -- (5.75,1.5) -- (5.75,-1.5)-- (-0.5,-1.5) -- cycle;
\begin{scope}[shift={(0,0)}, rotate=0, xscale=3.5,yscale=3.5,opacity=1.0] 
\draw (0.05,0.12) -- (0.05,-0.11) -- (-0.05, -0.13) -- (-0.05, 0.12) -- cycle;
\end{scope}
\begin{scope}[shift={(0,0)}, rotate=-10, xscale=3.5,yscale=3.5,opacity=1.0] 
\draw (0.05,0.12) -- (0.05,-0.11) -- (-0.05, -0.13) -- (-0.05, 0.12) -- cycle;
\end{scope}
\begin{scope}[shift={(0,0)}, rotate=10, xscale=3.5,yscale=3.5,opacity=1.0] 
\CP1
\fill (1,0) circle (0.01cm);
\end{scope}
\fill (0, 0) circle (0.5mm);
\draw[thick] ({-0.5*cos(-10)}, {-0.5*sin(-10)}) -- ({4.5*cos(-10)}, {4.5*sin(-10)});
\draw[thick] ({-0.5*cos( 10)}, {-0.5*sin( 10)}) -- ({4.5*cos( 10)}, {4.5*sin( 10)});
\draw[thick] (-0.5,0) -- (4.5,0);
\node[above right] at ({3.5*cos(15)}, {3.5*sin(15)}) {$\Gr_2(\mathbb{R}^3)$};
\node[above right] at ({3.5*cos(10)}, {3.5*sin(10)}) {$e$};
\node[above right] at (3.5, 0) {$\tilde{e}$};
\begin{scope}[shift={(4.8,0)}, rotate=0, xscale=0.30,yscale=0.30,opacity=1.0] \CP1 
\fill (1,0) circle (0.1cm);
\end{scope}
\begin{scope}[shift={({4.8*cos(10)},{4.8*sin(10)})}, rotate=10, xscale=0.30,yscale=0.30,opacity=1.0] \CP1 
\fill (1,0) circle (0.1cm);
\end{scope}
\begin{scope}[shift={({4.8*cos(-10)},{4.8*sin(-10)})}, rotate=-10, xscale=0.30,yscale=0.30,opacity=1.0] \CP1 
\fill (1,0) circle (0.1cm);
\end{scope}
\node[right] at ({5.2*cos(10)}, {5.2*sin(10)}) {$\taut_e$};
\node[right] at (5.2, 0) {$\taut_{\tilde{e}}$};
\end{tikzpicture}
\caption{A cartoon of the tautological bundle, $\taut$.  Here $e$ is a real 2-plane in
$\mathbb{R}^3$, which can be represented by a line because $Gr_2(\mathbb{R}^3)
\cong \mathbb{P}(\mathbb{R}^{3*})$. 
Each $\taut_e \cong \mathbb{P}(\mathbb{R}^2) \otimes \mathbb{C}
= \mathbb{P}^1$ is a Riemann sphere.
Thus, $\taut$ is depicted as a bundle of 2-spheres over a hemisphere.
}\label{fig:tautological}
\end{figure}

One can also consider the frame\footnote{Some authors might flip the names of
the frame and coframe bundles.  I tend to choose this notation because the frame
bundle is covariant with diffeomorphisms on the base space, and only
contravariant objects get a ``co-'' prefix.  The jargon for duality is always
frustrating.}
bundle $\mathcal{F}_{\taut}$ over $\Gr_n(X)$ associated to
$\taut$, whose fiber is all linear isomorphisms
\begin{equation}
\mathcal{F}_{\taut,e} 
= \{ (u^i): \taut_e \overset{\sim}{\to} \mathbb{P}^{n-1} \} 
= \{ \text{bases of $\taut^*_e$}\} 
\cong PGL(n),
\end{equation}
and the coframe bundle $\mathcal{F}_{\taut^*}$ over $\Gr_n(X)$ associated to
$\taut^*$, whose fiber is all linear isomorphisms
\begin{equation}
\mathcal{F}_{\taut^*,e} = \{ (u_i): \taut^*_e \overset{\sim}{\to} \mathbb{P}^{n-1} \} 
= \{ \text{bases of $\taut_e$}\} 
\cong PGL(n).
\end{equation}
To write homogeneous complex-algebraic ideals on $\taut^*_e$ that vary across $e \in \Gr_n(X)$, 
one can choose any section $(u_i)$ of $\mathcal{F}_{\taut^*}$ to give coordinates, and use the ring
\begin{equation}
S = C^\infty(\Gr_n(X))[u_1, \ldots, u_n].
\label{eqn:XiRing}
\end{equation}

\part{PDEs on Manifolds}\label{part:manifolds}
In this part, we build bundles whose fibers are the structures seen in
Part~\ref{part:back}.  This produces a satisfying language for describing a
system of PDEs on a manifold in Section~\ref{sec:eds}.

\section{Bundles upon Bundles}\label{sec:bunbun}
If $M$ is a smooth manifold of dimension $m=n+r$, then we can form the
smooth bundle $\Gr_n(TM)$ with fiber $\Gr_n(T_pM)$.  Let $\varpi:\Gr_n(TM) \to
M$ denote the bundle projection. 

Because \eqref{eqn:shortexact} holds for $X = T_pM$ at any $p \in M$, any
local section of $\Gr_n(TM)$ can be described by choosing its annihilator
section of $\Gr_r(T^*M)$, and vice-versa.
For every $p \in M$, the Grassmann variety $\Gr_n(T_pM)$ has a tautological
bundle $\taut(p)$ with fiber $\taut_e(p) = \mathbb{P}e \otimes
\mathbb{C}$, a dual bundle, and so on. 

The total space $\Gr_n(TM)$ is a manifold in its own
right, so we may consider $\taut$ as a bundle over the manifold
$\Gr_n(TM)$, which is itself a bundle over $M$.
In other words, we can
reinterpret all of Section~\ref{sec:tautbundle} in terms of bundles over $\Gr_n(TM)$ by using $\mathbb{X}$ denote the projective bundle over
$\Gr_n(TM)$ that has fiber $\mathbb{X}_e = \mathbb{P}T_pM \otimes \mathbb{C}$ at $e$ with $\varpi(e)=p$.
A complete description of some $v \in \taut$ would be $(p,e,v)$ where 
$v \in \mathbb{P}e \otimes \mathbb{C}$, and $e \in \Gr_n(T_pM)$, and $p \in M$.
A complete description of some $\varphi \in \taut^*$ would be $(p,e,\varphi)$ where 
$\varphi \in \mathbb{P}e^* \otimes \mathbb{C}$, and $e \in \Gr_n(T_pM)$, and $p \in M$.
See Figure~\ref{fig:bunbun}.
The same bundle-wise constructions hold for $\taut^\perp$,
$(\mathbb{X}/\taut)$, $\mathcal{F}_{\taut}$, and $\mathcal{F}_{\taut^*}$ from
Section~\ref{sec:tautbundle}.

Extending \eqref{eqn:XiRing} to write homogeneous complex-algebraic ideals on $\taut^*_e$ that vary across $e \in \Gr_n(TM)$, 
one can choose any section $(u_i)$ of $\mathcal{F}_{\taut^*}$ to give coordinates, and use the ring
\begin{equation}
S = C^\infty(\Gr_n(TM))[u_1, \ldots, u_n].
\label{label:XiRingM}
\end{equation}

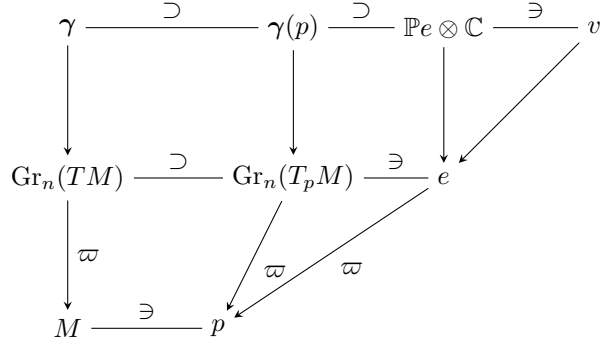
\begin{figure}
\begin{center}
\begin{tikzpicture}[node distance=2cm, auto]
\node (M) {$M$};
\node[right of = M] (p) {$p$};
\node[above of = M] (Gr) {$\Gr_n(TM)$};
\node[right of = Gr, node distance=3cm] (Grp) {$\Gr_n(T_pM)$};
\node[right of = Grp] (e) {$e$};
\node[above of = Gr] (g) {$\taut$};
\node[above of = Grp] (gp) {$\taut(p)$};
\node[above of = e] (pec) {$\mathbb{P}e \otimes \mathbb{C}$};
\node[right of = pec] (v) {$v$};
\draw[-, above] (Grp) to node {$\supset$} (Gr);
\draw[->] (g) to node {} (Gr);
\draw[->] (Gr) to node {$\varpi$} (M);
\draw[-, above] (pec) to node {$\supset$} (gp);
\draw[->] (gp) to node {} (Grp);
\draw[-, above] (gp) to node {$\supset$} (g);
\draw[-, above] (p) to node {$\owns$} (M);
\draw[-, above] (e) to node {$\owns$} (Grp);
\draw[->] (Grp) to node {$\varpi$} (p);
\draw[->] (e) to node {$\varpi$} (p);
\draw[->] (pec) to node {} (e);
\draw[-, above] (v) to node {$\owns$} (pec);
\draw[->] (v) to node { } (e);
\end{tikzpicture}
\end{center}
\caption{Tautological bundles over Grassmann bundles over manifolds. 
Vertical arrows are bundle projections.}\label{fig:bunbun}
\end{figure}

\subsection{The Contact Ideal}\label{sec:contact}
For any $e \in \Gr_n(TM)$, consider its annihilator subspace $e^\perp \subset T^*_pM$.
There is a corresponding subspace $J_e \subset T^*_e \Gr_n(TM)$, defined as 
\begin{equation}
J_e = \pair{ \zeta \circ \varpi_* ~:~ \zeta \in e^\perp} =
e^\perp \circ \varpi_*.
\end{equation}
as in Figure~\ref{fig:contact}.
If $(z^a)$ is a basis of $e^\perp$, then we let $\theta^a = z^a \circ \varpi_*$
for each $a$ to define a basis $(\theta^a)$ of $J_e$.

\begin{figure}[ht]
\begin{center}
\begin{tikzpicture}[scale=1.5]                                                  
\draw (1,0) node (M) {$M$};
\draw (0.4,0) node  {$p$};
\draw (0.7,0) node {$\in$};
\draw (3,0) node (TM) {$T_pM$};
\draw (1,2) node (G) {$\Gr_n(TM)$};
\draw (0,2) node {$e$}; 
\draw (0.3,2) node {$\in$};
\draw (3,2) node (TG) {$T_e\Gr_n(TM)$};
\draw (5,1) node (V) {$\mathbb{R}^{r}$};
\draw[->] (G) to node[left] {$\varpi$} (M);
\draw[->] (TG) to node[left] {$\varpi_*$} (TM);
\draw[->] (TM) to node[below right] {$\pair{(z^a)} = e^\perp$} (V);
\draw[->,dashed] (TG) to node[above right] {$\pair{(\theta^a)} = J_e$} (V);
\end{tikzpicture}
\end{center}
\caption{Contact forms on the Grassmann bundle of
$M$.}\label{fig:contact}
\end{figure}
In the exterior algebra $\Omega^\bullet\left(\Gr_n(TM)\right)$, consider the 
ideal $\mathcal{J}$ that is generated as $\pair{J,\mathrm{d}J} = \pair{\theta^a,
\mathrm{d}\theta^a}$.  This is called the  
\emph{contact ideal}, and it is the first example of an EDS as seen in
Section~\ref{sec:eds}.
Note that, for any (local) section $\epsilon:M \to \Gr_n(TM)$, the contact
ideal satisfies the universal reproducing property 
\begin{equation}
\epsilon^*(J) = \epsilon^*(\epsilon^\perp \circ \varpi_*) =
\epsilon^\perp \circ \varpi_* \circ \epsilon_* =  
\epsilon^\perp. 
\end{equation}
Because this property is universal, the subbundle $J$ is a submodule defined globally across
$\Gr_n(TM)$ even if topology forces any
particular section $\epsilon$ to be defined locally. 

If one were to choose local coordinates $(x^i, y^a)$ for $M$ and local fiber coordinates $(P^a_i)$ for $\Gr_n(TM)$
near a particular $n$-plane $e = \ker \{ \mathrm{d}y^a \}$, then $\mathcal{J}$
is the ideal
typically written as  
\begin{equation}
\begin{cases}
0 = \theta^a = \mathrm{d}y^a - P^a_i \mathrm{d}x^i,\\
0 = \mathrm{d}\theta^a = - \mathrm{d}P^a_i \wedge\mathrm{d}x^i,
\end{cases}
\label{eqn:contactP}
\end{equation}
where the functions $P^a_i$ depend on $\tilde{e}$ in an open neighborhood of $e$ in $\Gr_n(TM)$.

After reading Section~\ref{sec:immersions}, compare this coordinate description
to your favorite definition of jet space, $\mathbb{J}^1(\mathbb{R}^n,
\mathbb{R}^r)$.  Also, compare the local fiber coordinates $P^a_i$ to the tangent
coordinates $K^a_i$ from Section~\ref{sec:tan}; when restricting to the fiber
over a single basepoint $p \in
M$, they are
essentially identical.
For some highly amusing applications of the contact system, see
\cite{Gromov1986}.

\subsection{Immersions and Frame Bundles}\label{sec:immersions}
Fix an immersion $\iota:N \to M$ with $\dim N = n$.
For any $x \in N$ with $\iota(x)=p$, the push-forward derivative has image $\iota_*(T_xN)$, which
is an $n$-dimensional subspace of $T_pM$; hence, $\iota_*(T_xN) \in \Gr_n(TM)$.  
Define the map $\iota^{(1)}: N \to \Gr_n(TM)$ by 
\begin{equation}
\iota^{(1)}(x) = \iota_*(T_xN) = e \in \Gr_n(TM),
\end{equation}
and note that $\iota = \varpi \circ \iota^{(1)}$, so $\iota_* = \varpi_* \circ
\iota^{(1)}_*$.

It is obvious from the definition that $\iota^{(1)}$ is also an immersion.
Therefore, we can use it to pull-back the tautological bundle $\taut^*$ as
defined in Sections~\ref{sec:tautbundle} and \ref{sec:bunbun}.  Let $\taut^*_N =
\iota^{(1)*}\taut^*$, which has fiber 
\begin{equation}
\taut^*_{N,x} = \taut^*_e(p) = \mathbb{P}e^* \otimes \mathbb{C} =
\mathbb{P}\iota_*(T_xN)\otimes \mathbb{C};
\end{equation}
that is, $\taut^*_N$ is identified with $\mathbb{P} T^*N \otimes \mathbb{C}$
via $\iota_*$.  See Figure~\ref{fig:prolongation}.

The immersion $\iota^{(1)}$ is called the \emph{prolongation} of the immersion $\iota$.

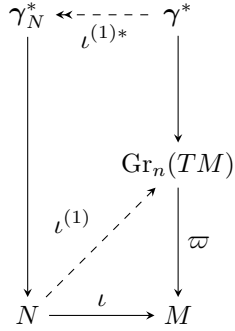
\begin{figure}
\begin{center}
\begin{tikzpicture}[node distance=2cm, auto]
\node (M) {$M$};
\node[above of = M] (Gr) {$\Gr_n(TM)$};
\node[above of = Gr] (g) {$\taut^*$};
\node[left of = M, distance=3cm] (N) {$N$};
\node[left of = g] (gN) {$\taut^*_N$};
\draw[->] (gN) to node {} (N);
\draw[->] (g) to node {} (Gr);
\draw[->] (Gr) to node[right] {$\varpi$} (M);
\draw[->] (N) to node {$\iota$} (M);
\draw[<<-, dashed] (gN) to node[below] {$\iota^{(1)*}$} (g);
\draw[->, dashed] (N) to node {$\iota^{(1)}$} (Gr);
\end{tikzpicture}
\end{center}
\caption{The dual tautological bundle $\taut^*$ pulls back to form a bundle
$\taut^*_N$ over an immersed
submanifold $N$.}\label{fig:prolongation}
\end{figure}

Now, consider the contact forms $(\theta^a) = (z^a \circ \varpi_*)$ forms from Section~\ref{sec:contact}.
For all $x \in N$ and all $v \in T_xN$, we have 
\begin{equation}
\iota^{(1)*}(\theta^a)(v) 
= 
\theta^a(\iota^{(1)}_*(v)) 
= z^a \circ \varpi_* \circ \iota^{(1)}_*(v) 
= z^a(\iota_*(v)) = 0,
\end{equation}
which ultimately gives the following lemma:
\begin{lemma}
If $\iota:N \to M$ is an immersion for $\dim N=n$, then $\iota^{(1)*}(\mathcal{J})=0$.   
Conversely, if $\iota^{\prime}:N \to \Gr_n(TM)$ is an immersion for $\dim N = n$
satisfying $\iota^{\prime *}(\mathcal{J})=0$ and such that
the image $\iota'_*(T_xN)$ is transverse to the fiber $\ker \varpi_*$ for all
$x \in N$, then there is some immersion $\iota:N \to M$ such that $\iota^{(1)}
= \iota^\prime$.
\end{lemma}

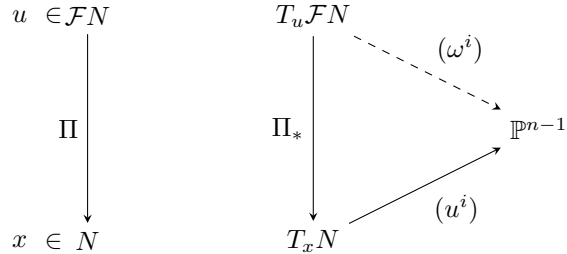
\begin{figure}[ht]
\begin{center}
\begin{tikzpicture}[scale=1.5]                                                  
\draw (1,0) node (N) {$N$};
\draw (0.4,0) node  {$x$};
\draw (0.7,0) node {$\in$};
\draw (3,0) node (TN) {$T_xN$};
\draw (1,2) node (F) {$\mathcal{F}N$};
\draw (0.4,2) node {$u$}; 
\draw (0.7,2) node {$\in$};
\draw (3,2) node (TF) {$T_u\mathcal{F}N$};
\draw (5,1) node (V) {$\mathbb{P}^{n-1}$};
\draw[->] (F) to node[left] {$\Pi$} (N);
\draw[->] (TF) to node[left] {$\Pi_*$} (TN);
\draw[->] (TN) to node[below right] {$(u^i)$} (V);
\draw[->,dashed] (TF) to node[above right] {$(\omega^i)$} (V);
\end{tikzpicture}
\end{center}
\caption{Tautological Form of the frame bundle of a manifold
$N$.}\label{fig:tautform}
\end{figure}

Moreover, recall that any manifold $N$ of dimension $n$ admits a projective frame bundle
$\Pi:\mathcal{F}N \to N$ with fiber 
\begin{equation}
\mathcal{F}_xN = \{ (u^i): \mathbb{P}T_xN \overset{\sim}{\to} \mathbb{P}^{n-1} \} 
= \{ \text{bases of $\mathbb{P}T^*_xN \otimes \mathbb{C}$}\} 
\cong PGL(n),
\end{equation}
The total space $\mathcal{F}N$ admits a tautological\footnote{In various
references, this 1-form is called the canonical, the Hilbert, and
the soldering 1-form.} 1-form 
$\omega:T
\mathcal{F}N \to \mathbb{P}^{n-1}$ defined by $\omega^i_u = u^i \circ \Pi_*$ as
in Figure~\ref{fig:tautform}.  It is characterized by its universal reproducing
property: for any (local) section $\eta:N \to \mathcal{F}N$:
\begin{equation}
\eta^*(\omega^i) = \eta^*(\eta^i \circ \Pi_*) = 
\eta^i \circ \Pi_* \circ \eta_* = \eta^i,
\end{equation}
or, more succinctly, $\eta^*(\omega) = \eta$.

Because this property is universal, the 1-form $\omega$ is defined globally across
$\mathcal{F}N$ even if topology forces any
particular 1-form $\eta$ to be defined locally. 

For any local diffeomorphism $f:N \to \tilde{N}$, there is an induced (covariant) map
on the frame bundles $f^\dagger:\mathcal{F}N \to \mathcal{F}\tilde{N}$ by
$f^\dagger:(u^i)
\mapsto (u^i) \circ (f_*)^{-1}$.  Using the universal property, it is easy to prove this
lemma, which shows that diffeomorphisms are characterized by their preservation
of the tautological
form on the frame bundle:
\begin{lemma}
If $f:N \to \tilde{N}$ is a diffeomorphism, then $(f^\dagger)^*(\tilde{\omega}) =
\omega$.  Conversely, if $F:\mathcal{F}N \to \mathcal{F}\tilde{N}$ is
$PGL(n)$-equivariant
diffeomorphism such that $F^*(\tilde{\omega}) = \omega$, then there exists a
unique diffeomorphism $f:N \to \tilde{N}$ such that $f^\dagger = F$.
\end{lemma}

Combining the universal properties of the $\mathcal{J}$ and $\omega$, we obtain
the following theorem telling us what information we can transfer from
$\Gr_n(TM)$ to an immersed submanifold:
\begin{thm}
If $\iota:N \to M$ is a smooth immersion, then
\begin{itemize}
\item $\iota^{(1)*}(\mathcal{J}) = 0$, and
\item $\mathcal{F}N = \iota^{(1)*}(\mathcal{F}_{\taut})$.
\end{itemize}
Conversely, if $\iota^\prime:N \to \Gr_n(TM)$ is a smooth immersion such that
\begin{itemize}
\item $\iota^{\prime*}(\mathcal{J}) = 0$, and
\item $\mathcal{F}N = \iota^{\prime*}(\mathcal{F}_{\taut})$,
\end{itemize}
then there exists a smooth immersion $\iota:N \to M$ such that
$\iota^{(1)}=\iota^\prime$.
\label{thm:equiv}
\end{thm}
That is, an immersed submanifold satisfies the contact ideal, which is
generated differentially by some annihilator 1-forms $(\theta^a)$ spanning
$\taut^\perp$, and its frame bundle is equipped with tautological 1-forms
$(\omega^i)$ spanning $\taut^*$.

\begin{rmk}
Note the similarity between the universal property of the contact ideal on the
Grassman bundle and the
universal property of the tautological 1-form on the frame bundle.
Exploitation of this interaction as in Theorem~\ref{thm:equiv} has a long and interesting history.

For example, consider the study of a Lie pseudogroup acting on a manifold $M$.
One option is to differentiate the coordinates of $M$ repeatedly using the
contact ideal until
differential syzygies of the Lie pseudogroup action can be found in prolonged
local coordinates, which are then converted to a coordinate-free description
using the pseudogroup action.  The other option is to work on the frame bundle
of $M$ immediately, where any expression on the tautological
1-form is automatically invariant, then prolong as necessary
to reveal the syzygies.   The latter is used often when the Lie pseudogroup
arises as equivalence of intrinsic $G$-structures, and the former is used often when the Lie
pseudogroup arises from an extrinsic action on some ambient coordinates.
For more on these fascinating and interconnected ideas, I encourage you to read
\cite{Clelland2016},  \cite{Olver1995}, \cite{Valiquette2013}, and
\cite{Gardner1989}---and the collected works of E. Cartan.
\label{rmk:contact}
\end{rmk}

\section{Exterior Differential Systems}\label{sec:eds}
Let $M$ be a smooth manifold of finite dimension $m$.  An \emph{exterior
differential system} [EDS] on $M$ consists of an ideal $\mathcal{I}$ in the
total exterior algebra $\Omega^\bullet(M)$ that is differentially closed and
finitely generated.  \emph{Differentially closed} means that
$\mathrm{d}\mathcal{I} \subset \mathcal{I}$.  \emph{Finitely generated} means
that in each degree $d$, the $d$-forms in the ideal,
$\mathcal{I}_{d}=\mathcal{I}\cap \Omega^d(M)$, form a finitely generated
$C^\infty(M)$-module.  We assume that $\mathcal{I}_0=0$; otherwise, one would
restrict to a subvariety of $M$ defined by those functions.  
A \emph{solution} or \emph{integral manifold} is an immersed manifold $\iota:N \to M$ such that
$\iota^*(\mathcal{I})=0$.
Optionally, we sometimes specify an independence condition as an $n$-form
$\boldsymbol{\omega} \in \Omega^n(M)$ that is not allowed to vanish on
solutions.  When an EDS represents a system of PDEs in local coordinates $x^1,
\ldots, x^n$, then $\boldsymbol{\omega} = \mathrm{d}x^1 \wedge \cdots \wedge
\mathrm{d}x^n$, meaning that we seek solutions $N$ on which those coordinates
are sensible, and $\iota:N \to M$ is a function that gives the dependent variables
in $M$ (those transverse to $\iota(N)\subset M$) as functions of the independent variables in $N$.

\begin{rmk}
Exterior differential systems are defined this way because 
the term ``PDE'' or ``system of PDEs'' is difficult to
pin down with geometric precision.  Colloquially, ``system of PDEs'' usually
means a finite set of (hopefully, smooth) equations on some local jet space.
In Section~\ref{sec:grass}, we explored the geometry of the bundle $\Gr_n(TM)$;
recall that the contact system $\mathcal{J}$ on $\Gr_n(TM)$ provides a
coordinate-invariant notion of jet space.  
So, a system of PDEs can be  thought of as a collection of equations on jet
$\Gr_n(TM)$.  Hopefully, those equations are smooth and respect the bundle structure coming
from the contact system (otherwise, derivatives misbehave).  
By virtue of the Pl\"ucker embedding $\Gr_n(TM) \to \mathbb{P}\wedge^n(TM)$, an
EDS provides precisely the structure to write an ideal whose variety is a
subvariety (in the bundle sense) of $\Gr_n(TM)$.  By taking smooth
subvarieties, we can apply Remark~\ref{rmk:subspacesaretableaux} and apply our
knowledge of tableaux from Part~\ref{part:back} to study EDS.
Even by this definition, an EDS could be rather wild;
however, in many practical applications, it happens that $\mathcal{I}$ is
generated by a finite collection of smooth differential forms of homogeneous degree,
so one obtains a smooth algebraic variety in local fiber coordinates
of $\Gr_n(TM)$.
See \cite{McKay2018} for more examples, additional insight, and historical
context.
\end{rmk}

\subsection{Differential Ideals and Integral Elements}
To be precise, an \emph{integral element} of $\mathcal{I}$ at $p \in M$ is a
linear subspace $e \subset T_pM$ such that $\varphi|_e = 0$ for all $\varphi
\in \mathcal{I}_n$.  That is, the $n$-forms in $\mathcal{I}$ provide a
collection of functions that cut out a variety, $\Var_n(\mathcal{I}) \subset
\Gr_n(TM)$.  These functions vary smoothly in $M$ and are homogeneous in the
fiber variables.    

There is a maximal dimension $n$ for which $\Var_n(\mathcal{I})$ is
locally non-empty, which is the case of interest.  If an independence condition
$\boldsymbol{\omega}$ is specified, we also require $\boldsymbol{\omega}|_e
\neq 0$, which forces $\Var_n(\mathcal{I})$ to lie in the open subset of
$\Gr_n(TM)$ for which that condition holds.  (For example, in the case of
the contact system, the condition
$\boldsymbol{\omega} = \mathrm{d}x^1 \wedge \cdots \wedge\mathrm{d}x^n \neq 0$ holds
in the same neighborhood where \eqref{eqn:contactP} makes sense.)

Because $\mathcal{I}_n$ is finitely generated by smooth functions, Sard's
theorem guarantees an
open, dense subset $\Var_n^o(\mathcal{I})\subset \Var_n(\mathcal{I})$ defined
as the smooth subbundle of $\Gr_n(TM)$ that is cut out smoothly by smooth
functions.  

\begin{defn}[K\"ahler-ordinary]
Integral elements in $\Var_n^o(\mathcal{I})$ are called
\emph{K\"ahler-ordinary}.
\end{defn}

A single connected component of $\Var_n^o(\mathcal{I})$ is denoted $M^{(1)}$. We
allow ourselves to redefine $M$ so that $\varpi:M^{(1)} \to M$ is a smooth
bundle.  

Let $s$ denote the dimension of each fiber of the projection $M^{(1)} \to M$,
so $t=nr-s$ is the corresponding codimension of $T_eM^{(1)}_p$ in
$T_e\Gr_n(T_pM)$.
That is, the projective bundle $A = \ker \varpi_* = TM^{(1)} \subset
T\Gr_n(TM)$ is a
tableau in the sense of Remark~\ref{rmk:subspacesaretableaux}, as
each fiber $A_e = T_e M^{(1)}_p$ is a linear subspace of $T_e \Gr_n(T_pM)$.
Because
$M^{(1)}$ is a smooth manifold, we have
\begin{lemma}
$K \in A_e$ implies $\arctan_e(K) \in M^{(1)}$. 
\end{lemma}
That is, we have a well-defined vector bundle $A = \ker \varpi_* \subset
TM^{(1)}$ over $M^{(1)}$.
\begin{defn}[K\"ahler-regular]
If $e$ is a K\"ahler-ordinary integral element and the Cartan characters
of each tableau $A$ are constant in an open neighborhood of $e$, then $e$ is
called \emph{K\"ahler-regular}.
\end{defn}
That is, the K\"ahler-regular integral elements form a dense open subset of the
K\"ahler-ordinary integral elements, which form a dense open subset of the
whole variety $\Var_n(\mathcal{I})$ of integral elements.

So that we may apply the results of Section~\ref{sec:tableau} without treating
the Cartan characters of $A_e$ as functions of $e$, we redefine $M^{(1)}$ to be a single
connected component of K\"ahler-regular integral elements, and we again allow
ourselves to redefine $M$ so that $\varpi:M^{(1)} \to M$ is a smooth bundle.

Such $M^{(1)}$ is called the \emph{first prolongation} of $(M,\mathcal{I})$,
though it is clear from the definition that there could be multiple first
prolongations, depending on which components of $\Var_n(\mathcal{I})$ are under consideration.

To generalize the notation and results of Part~\ref{part:back} to $M^{(1)}$,
define the restricted tautological bundles 
\begin{equation}
\begin{split}
V &= \taut|_{M^{(1)}} = \{ \mathbb{P}e \otimes \mathbb{C}\}_{e \in M^{(1)}}, \\
V^* &= \taut^*|_{M^{(1)}} = \{ \mathbb{P}e^* \otimes \mathbb{C}\}_{e \in M^{(1)}}, \\
W &= (\mathbb{X}/\taut)|_{M^{(1)}} = \{ \mathbb{P}(T_pM/e) \otimes
\mathbb{C}\}_{e \in M^{(1)}}, \\
V^\perp &= \taut^\perp|_{M^{(1)}} = \{ \mathbb{P}e^\perp \otimes \mathbb{C}\}_{e \in M^{(1)}}
\label{eqn:mybundles}
\end{split}
\end{equation}
\textbf{Warning!}  These are now complex projective bundles, not vector spaces
as in Section~\ref{sec:tableau}! Sometimes, it is convenient to think of $A  =
\ker \varpi^* \subset TM^{(1)}$ as being a complex projective bundle, too, in
which case we consider it to be a subbundle of the projective bundle $W \otimes
V^*$.  Of course, the notation has been developed to be consistent regardless.

An \emph{integral manifold} of $\mathcal{I}$ is an immersion $\iota:N \to M$
such that $\iota^*(\varphi)=0$ for all $\varphi \in \mathcal{I}$.  (If an
independence condition $\boldsymbol{\omega}$ is specified, we require that
$\iota^*(\boldsymbol{\omega}) \neq 0$, too.)  When we are considering a
particular K\"ahler-regular component $M^{(1)} \subset \Var_n(\mathcal{I})$ as above, we say $N$ is an
\emph{ordinary} integral manifold provided that $\iota_*(TN) \subset M^{(1)}$.
All of the observations from Section~\ref{sec:immersions} apply, but
$\iota^{(1)}(N)$ lies in the submanifold $M^{(1)}$, and $\iota^{(1)}_*(TN)$
lies in the subbundle $A$.   The overall goal is to construct all ordinary
integral manifolds of $(M,\mathcal{I})$ through the careful study of the
geometry of a K\"ahler-regular first prolongation $M^{(1)}$.

\subsection{Prolongation and Spencer Cohomology}\label{sec:spencer}
Suppose that $\iota: N\to M$ is an ordinary integral manifold of $\mathcal{I}$.
By Theorem~\ref{thm:equiv}, the 1-forms $\theta^a$ spanning $J_e$ must vanish
for each $e \in \iota^{(1)}(N)$.  The tautological form $(\omega^i)$ on 
$\mathcal{F}_{\taut}$ pulls back to a nondegenerate frame $(\eta^i)$ on
$N$, since $\iota^{(1)}$ is an immersion.

Therefore, if $\iota^{(1)}:N \to M^{(1)}$ actually exists, we have
\begin{equation}
\begin{split}
\iota^{(1)*}(\theta^a) &= 0,\\
\iota^{(1)*}(\mathrm{d}\theta^a) &= 0,\\
\iota^{(1)*}(\omega^1 \wedge \cdots \wedge \omega^n)  &=
\eta^1 \wedge \cdots \wedge \eta^n \neq 0. 
\end{split}
\end{equation}
However, working on the frame bundle of $M^{(1)}$, these forms satisfy a more
general equation, called \emph{Cartan's structure equation}:
\begin{equation}
\mathrm{d}\theta^a \equiv \pi^a_i \wedge \omega^i + \frac12 T^a_{i,j}\,\omega^i
\wedge \omega^j,\ \mod \{\theta^b\}.
\label{eqn:cartanstr}
\end{equation}
The derivative of $\theta^a$ must take this form, because $\theta^a$ and
$\omega^i$ are semi-basic with respect to the bundle $\varpi:M^{(1)} \to M$,
whereas $\pi^a_i \in A$ is vertical, so $\mathrm{d}\theta^a$ cannot involve a
totally vertical 2-form. See discussion of connections and principal bundles in
\cite{Kobayashi1963}.

Let us now describe the meaning of each of the terms in \eqref{eqn:cartanstr},
with respect to the ordinary integral manifold $\iota:N \to M$.
Using the dual coframe $z_a \leftrightarrow \theta^a$ for
$W \leftrightarrow V^\perp$, we can see
that $\pi = \pi^a_i(z_a \otimes \omega^i)$ lies in $A$. (Hence, it is called
the \emph{tableau term}.) 
In particular, it must be that 
\begin{equation}
\iota^{(1)*}(\pi^a_i) = P^a_{i,j} \eta^j
\label{eqn:iotapi}
\end{equation}
for some function $P^a_{i,j}$ that must satisfy $P^a_{i,j}\eta^i \wedge \eta^j
    =0$, so $P^a_{i,j} = P^a_{j,i}$. 
That is, the homomorphism $P$ lies in the fiber of $W \otimes (V^* \otimes
V^*)$ over $e$, as 
\begin{equation}
P \in A \otimes V^* \subset (W \otimes V^*) \otimes V^* = W \otimes (V^* \otimes V^*).
\label{eqn:A2}
\end{equation}
Moreover, the existence of an immersion $\iota^{(1)}: N \to M^{(1)}$ requires
that the \emph{torsion term} $w_a T^a_{i,j}\,\omega^i \wedge \omega^j$ can be
removed in \eqref{eqn:cartanstr}; otherwise, it cannot be that
$\iota^{(1)*}\mathrm{d}\theta^a = 0$ as required.
That is, it must be possible to rewrite $\pi^a_i \mapsfrom \pi^a_i + Q^a_{i,j}
\omega^j$ for $Q \in A \otimes V^*$ such that any $T^a_{i,j}$ term is absorbed.
Note that this absorption of torsion is an algebraic property of the tableau
$A$.
In summary, we have Lemma~\ref{lem:spencer2}.
\begin{lemma}
Let $\delta: A \otimes V^* \to W \otimes \wedge^2 V^*$ denote the composition
of skewing $\otimes^2 V^* \to \wedge^2 V^*$ and inclusion $A \to W \otimes V^*$, and write $A^{(1)} = \ker \delta$ and $H^2(A) =
\coker \delta$: 
\begin{equation}
0 \to A^{(1)} \to A \otimes V^* \overset{\delta}{\to} W \otimes \wedge^2V^* \to
H^{2}(A)\to 0.
\end{equation}
For any ordinary integral manifold $N$, the homomorphism $P$ of
\eqref{eqn:iotapi} and \eqref{eqn:A2} lies in 
$A^{(1)}$, and the pullback of torsion $T$ is zero in $H^2(A)$.
\label{lem:spencer2}
\end{lemma}
Writing $\delta$ in a chosen coframe, it is easy to check that 
\begin{equation}
\dim A^{(1)} \leq s_1 + 2s_2 + \cdots + n s_n.
\label{eqn:Cartansineq}
\end{equation}
The case of equality is considered in Section~\ref{sec:edsinv}.

The exterior differential system $\mathcal{I}^{(1)}$ on $M^{(1)}$
generated as 
\begin{equation}
\mathcal{I}^{(1)} = \pair{ \theta^a, \mathrm{d}\theta^a } =
\varpi^*(\mathcal{I}) + \mathcal{J}
\label{eqn:I1}
\end{equation}
is called the (first) \emph{prolongation} of $(M, \mathcal{I})$, and we are
back where we started at the beginning of Section~\ref{sec:eds}.  We can construct $M^{(2)}
\subset \Gr_n(TM^{(1)})$, and repeat the entire process for $E \in
M^{(2)}$ over $e \in M^{(1)}$ that was used for $e \in M^{(1)}$ over $p \in M$.
Lemma~\ref{lem:spencer2} essentially says that $A^{(1)}$ is the tableau 
bundle $T M^{(2)} \subset T\Gr_n(TM^{(1)})$.  
Thus, we can construct $M^{(3)}$ over $M^{(2)}$ and re-apply
Lemma~\ref{lem:spencer2}, and so on.  
By the definition of $M^{(1)}$ and \eqref{eqn:I1}, we have 
\begin{cor}
Every ordinary integral manifold $N$ of $(M^{(1)}, \mathcal{I}^{(1)})$ is also an ordinary
integral manifold of $(M,\mathcal{I})$.  However, the converse might fail,
as the smooth connected locus of $M^{(1)}$ may be a strict subset of
$\Var_n(\mathcal{I})$.
\end{cor}
Overall, we achieve exact sequences that summarize the entire situation of the
tangent spaces of an immersed ordinary integral manifold $N$ of
$\mathcal{I}$, $\mathcal{I}^{(1)}$, $\mathcal{I}^{(2)}$, $\mathcal{I}^{(3)}$, \ldots 
\begin{equation}
\begin{split}
0 \to A &\to W \otimes \wedge^1V^* \to H^{1}(A)\to 0,\\
0 \to A^{(1)} \to A \otimes V^* &\overset{\delta}{\to} W \otimes \wedge^2V^* \to
H^{2}(A)\to 0,\\
0 \to A^{(2)} \to A^{(1)} \otimes V^* &\overset{\delta}{\to} W \otimes \wedge^3V^* \to
H^{3}(A)\to 0,\\
&\vdots\\
0 \to A^{(n-1)} \to A^{(n-2)} \otimes V^* &\overset{\delta}{\to} W \otimes
\wedge^{n}V^* \to H^{n}(A)\to 0.
\end{split}
\label{eqn:spencers}
\end{equation}

The cokernels $H^1(A)$, $H^2(A)$, \ldots, $H^n(A)$ are the \emph{Spencer
cohomology} of the tableau $A$.  Even outside the context of exterior
differential systems, they are defined for formal tableaux $A \subset W \otimes
V^*$ via the exact
sequences \eqref{eqn:spencers} as 
\begin{equation}
H^k(A) =
\left(A \otimes (\otimes^{k-1} V^*)\right) / \left(W \otimes \wedge^{k}
V^*\right).
\end{equation}

Spencer cohomology detects functional obstructions to the solution of the
initial-value problem on $M^{(k)}$ in the form of \emph{torsion}; this is
explained nicely in \cite[Section 5.6]{Ivey2003}, and the reader is urged to
read their presentation.

Spencer cohomology was a major focus of the formal study of partial
differential equations and Lie pseudogroups in the mid-20th century; most
notably, \cite{Spencer1962, Quillen1964, Singer1965, Guillemin1966,
Goldschmidt1967, Gardner1967a, Guillemin1968,  Guillemin1968b, Guillemin1970}.  As it
happens, many of the major results of that era are easy to re-prove
under our regularity assumptions on $M^{(1)}$ and using the endovolutive notation from
Section~\ref{sec:tableau}, particularly when using the involutivity criteria in
Section~\ref{sec:edsinv} that were detailed in \cite{Smith2014a}.  We
demonstrate this in Parts~\ref{part:Xi} and \ref{part:eikonal}.

\section{Involutivity of Exterior Differential Systems}\label{sec:edsinv}
\begin{defn}[Cartan's test]
A tableau $A$ is called \emph{involutive} if equality holds in
Equation~\eqref{eqn:Cartansineq},
\[ s_1 + 2s_2 + \cdots + \ell s_\ell = \dim A^{(1)}\]
\end{defn}
\begin{defn}
A tableau $A$ is called \emph{formally integrable} if $H^k(A)=0$ for all $k
\geq 2$.
\end{defn}
Cartan's test comes from the following consequence of the Cartan--K\"ahler
theorem.\footnote{See \cite[Chapter III]{BCGGG} or \cite{Ivey2003} for more background on the
Cartan--K\"ahler theorem; it is not our focus here.}
\begin{thm}
Suppose that $(M,\mathcal{I})$ is an analytic exterior differential
differential system, that $M^{(1)}$ is a smooth sub-bundle, and that the
tableau bundle $A$ of $r \times n$ homomorphisms has constant\footnote{That is,
$M^{(1)}$ is K\"ahler-regular.}
Cartan characters $(s_1, s_2, \ldots, s_\ell)$ over $M^{(1)}$.
If $A$ is involutive and formally integrable, then through any point in $M$,
there is an analytic ordinary integral manifold $\iota:N \to M$.
Moreover, such $N$ are parametrized locally by $r$ constants, $s_1$ functions
of 1 variable, $s_2$ functions of 2 variables, \ldots, $s_\ell$ functions of
$\ell$ variables.
\label{thm:CK}
\end{thm}
Somewhat confusingly, the situation in Theorem~\ref{thm:CK} is called
\emph{involutivity of $(M,\mathcal{I})$};  that is, an EDS might fail to be
involutive even if its tableau is involutive, because there may be
nonzero \emph{torsion} in $H^k(A)$, meaning that $\mathcal{I}$ fails to be
formally integrable.  This means essentially that the ideal $\mathcal{I}$ is
being studied on the wrong manifold.

For a beautiful interpretation of Cartan's test that is relevant to the later
Sections of this course, read the introduction of \cite{Yang1987}.  In summary,
ordinary integral manifolds are constructed by decomposing the Cauchy
problem into a sequence of steps, each of which is determined and has solutions
using the Cauchy--Kowalevski theorem.

For fixed spaces $W$ and $V^*$, involutivity is a closed algebraic condition on
tableaux in $W \otimes V^*$.  Because the conditions come from Cartan's test,
which involves $W \otimes \wedge^2V^*$, it is not surprising that these
conditions are quadratic; however, writing down the precise ideal is a lengthy
argument.  Doing so was suggested in \cite[Chapter IV\S5]{BCGGG} and
accomplished for general tableaux in \cite{Smith2014a} following the outline in
\cite{Yang1987}.

\begin{thm}[Involutivity Criteria]
Suppose a tableau is given in generic bases as in \eqref{eqn:IB}.  The
tableau is \emph{involutive} if and only if there exists a basis of $W$ such that 
\begin{enumerate}
\item $\B^\lambda_i$ is endovolutive in that basis, and 

\item $\left( \B^\lambda_l \B^\mu_k - \B^\lambda_k \B^\mu_l\right)^a_b=0$ for all
$\lambda < l < k$ and $\lambda \leq \mu < k$ and all $a > s_l$.
\end{enumerate}
\label{thm:invcond}
\end{thm}

This theorem is our main computational tool in Part~\ref{part:Xi}.

\subsection{Moduli of Involutive Tableaux}\label{sec:moduli}
While it seems like a trivial (if lengthy) computation, consider carefully the
meaning of Theorem~\ref{thm:invcond}:  We can fix $r$, $n$, and Cartan
characters $s_1, \ldots, s_n$ and then write down an explicit ideal in
coordinates whose variety is all of the involutive tableaux with those
Cartan characters.  Hence, we can use computer algebra systems such as Macaulay2,
Magma, and Sage to decompose and analyze that ideal using Gr\"obner basis
techniques.  With
enough computer memory, we can answer the question ``What is the moduli of involutive
tableaux?''  By virtue of Theorem~\ref{thm:CK}, this is fairly close to answering
the question ``What is the moduli of involutive PDEs?''

For example, fix $r=n=3$ and $(s_1, s_2, s_3) = (3,2,0)$.  For some 
coefficients $x_0, \ldots, x_{15}$ in the ring $S$, an endovolutive
tableau must be of the form
\begin{equation}
(\pi^a_i) = 
\begin{pmatrix}
\alpha_0 & \alpha_3 & x_3 \alpha_0 + x_6 \alpha_1 + x_9 \alpha_2 + x_{12} \alpha_3 + x_{14}\alpha_4 \\
\alpha_1 & \alpha_4 & x_4 \alpha_0 + x_7 \alpha_1 + x_{10}\alpha_2 + x_{13}\alpha_3 + x_{15} \alpha_4\\
\alpha_2 & x_0\alpha_0 + x_1\alpha_1 +x_2\alpha_2 & x_5 \alpha_0 + x_8 \alpha_1 + x_{11} \alpha_2
\end{pmatrix},
\end{equation}
or in block form like \eqref{eqn:bigB},
\begin{equation}
(\B^\lambda_i) =
\begin{bmatrix}
\begin{pmatrix} 1 & 0 & 0\\  0 & 1 & 0\\ 0 & 0 & 1 \end{pmatrix}
&
\begin{pmatrix} 0 & 0 & 0\\  0 & 0 & 0\\ x_0 & x_1 & x_2 \end{pmatrix}
&
\begin{pmatrix} x_3 & x_6 & x_9\\  x_4 & x_7 & x_{10}\\ x_5 & x_8 & x_{11}
\end{pmatrix}\\
\begin{pmatrix} 0 & 0 & 0\\  0 & 0 & 0\\ 0 & 0 & 0 \end{pmatrix}
&
\begin{pmatrix} 1 & 0 & 0\\  0 & 1 & 0\\ 0 & 0 & 0 \end{pmatrix}
&
\begin{pmatrix} x_{12} & x_{14} & 0\\  x_{13} & x_{15} & 0\\ 0 & 0 & 0
\end{pmatrix}
\end{bmatrix}.
\label{eqn:moduliblock}
\end{equation}
Involutivity is an affine quadratic ideal $\mathscr{G}$ on $\mathbb{C}(x_0, \ldots, x_{15})$ generated by
the last rows of $\B^1_2\B^1_3 - \B^1_3\B^1_2$  
and $\B^1_2\B^2_3 - \B^1_3\B^2_2$, so:
\begin{equation}
\mathscr{G} = 
\begin{cases}
x_{0} x_{3} + x_{1} x_{4} + x_{2} x_{5} -  x_{0} x_{11},\\
x_{0} x_{6} + x_{1} x_{7} + x_{2} x_{8} -  x_{1} x_{11},\\
x_{0} x_{9} + x_{1} x_{10},\\
x_{0} x_{12} + x_{1} x_{13} -  x_{5},\\
x_{0} x_{14} + x_{1} x_{15} -  x_{8}.
\end{cases}
\end{equation}
The complete primary decomposition of this ideal reveals two components.  The
maximal component has dimension 12, and it is described by the fairly boring
prime ideal $\{x_0, x_1, x_5, x_8\}$.  The other component has dimension 11 and
its prime ideal is generated by 27 polynomials.  
See \url{http://goo.gl/jGTnMU} for how to compute this in SageMathCell.

Many of your favorite involutive second-order scalar PDEs in three independent
variables live somewhere in this variety; see \eqref{eqn:example_sym_array} and
Section~\ref{sec:exwave}.
Up to some notion of equivalence, this
is essentially the moduli space of such equations.  As seen in
Part~\ref{part:Xi}, their characteristic varieties are obtained by combining
$\mathscr{G}$ with the rank-1 ideal $\mathscr{R}$ on $\mathbb{C}[x_0, \ldots,
x_{15}, a_0, \ldots, a_{4}]$.  

However, there is still some ambiguity to be resolved, as it may be that a
given abstract tableau admits several endovolutive bases with apparently
distinct coordinate descriptions.

\subsection{Cauchy retractions}\label{sec:cauchy}
Before proceeding to Part~\ref{part:Xi}, it is worthwhile to mention Cauchy
retractions, which are much simpler than---and quite distinct from---elements of the
characteristic variety.
To confuse matters, many references call these ``Cauchy
characteristics.''
For any differentially closed ideal $\mathcal{I}\subset \Omega^\bullet M$,
the Cauchy retractions are the vectors that preserve $\mathcal{I}$; that is, 
$\mathfrak{g} = \{ v \in TM : v \lhk \mathcal{I} \subset \mathcal{I}\}
$. 
Because $\mathcal{I}$ is differentially closed, 
the annihilator bundle $\mathfrak{g}^\perp \subset T^*M$ is the smallest
Frobenius ideal in $\Omega^\bullet(M)$ that contains $\mathcal{I}$.  Then, for
any integral manifold $\iota:N \to M$,  the subspaces $\mathfrak{g}\cap
\iota^{(1)}(N)$ form an integrable distribution; that is,
$\mathfrak{g}^\perp_N$ is Frobenius as well \cite{Gardner1967a}.

Because
$\mathfrak{g}^\perp$ is a Frobenius system---a system of ODEs---it is common to 
redefine $(M,\mathcal{I})$ so that it is free of Cauchy retractions before
proceeding to study its integral manifolds.
The distinction between $\mathfrak{g}^\perp$ and the characteristic variety
$\Xi$ is explored further in
\cite{Smith2014b}.

\part{Characteristic and Rank-one Varieties}\label{part:Xi}
Thank you for taking the time to read the enormous amount of background in
Parts~\ref{part:back} and \ref{part:manifolds}.  We are ready to define and
deconstruct a fascinating mathematical object that lies at the heart of PDE
theory.

Here we stand:  We have an exterior differential system $\mathcal{I}$ on $M$.
Perhaps this EDS arose from a system of PDEs on $M$ and is equipped with 
an independence condition $\boldsymbol{\omega}$.
The EDS yields a smooth K\"ahler-regular subbundle $M^{(1)} \subset
\Gr_n(TM)$, where any $e \in M^{(1)}$ is an integral element of the original
EDS.  As a manifold in its own right, $M^{(1)}$ is equipped with tautological
bundles $V$, $V^*$, $W$, and $A$ from \eqref{eqn:mybundles}.
Moreover, $A$ is a subbundle of $W \otimes V^*$, so it is a
tableau bundle.  Its symbol $\sigma$ gives a short-exact sequence of bundles,
\begin{equation}
0 \to A \to W \otimes V^* \overset{\sigma}\to H^1(A) \to 0.
\label{eqn:tableausequence}
\end{equation}
An integral manifold is an immersion $\iota:N \to M$ such that $\iota_*(T_xN)
\in M^{(1)}_{\iota(x)}$ for all $x \in N$.  Let $\iota^{(1)}:N \to M^{(1)}$ denote the
map $x \mapsto e = \iota_*(T_xN)$.

As you read this part, compare it to \cite[Section 4.6]{Ivey2003} and
\cite[Chapter V]{BCGGG}.
The reader will note that we do not assume that $\mathcal{I}$ is a linear
Pfaffian system, nor do we build a prolonged EDS $\mathcal{I}^{(1)}$ using the
contact system.  Instead we are working with the tautological bundles per
Remark~\ref{rmk:contact}.

\section{The Characteristic Variety}
The original motivation for the characteristic variety is to see where the
initial-value problem becomes ambiguous.  That is, given an initial condition
for our PDE on a local submanifold of dimension $n{-}1$, when would the
$n$-dimensional solutions for that initial condition fail to be unique?
We express this condition in terms of integral elements.

\subsection{via Polar Extension}\label{sec:polarXi}
For an integral element $e' \in \Var_{n-1}(\mathcal{I})$, we consider its
space\footnote{The polar space is a vector space thanks to the assumption that
$\mathcal{I}_n$ is a finitely-generated $C^\infty(M)$-module, because that
assumption implies that the polar equations over $p \in M$ are a linear subspace of
$T_p^*M$.}
of integral extensions, called the \emph{polar space}, 
\begin{equation}
H(e') = \{ v ~:~ e=e' + \pair{v} \in \Var_{n}(\mathcal{I}) \}\subset TM 
\end{equation}
and the \emph{polar equations} comprise its annihilator,  
\begin{equation}
H^\perp (e') = \{ e'\lhk \varphi~:~ \varphi \in \mathcal{I}_n  \} \subset T^*M.
\end{equation}
The \emph{polar rank} is $r(e') = \dim H(e') - \dim e' -1$.
If $r(e')=-1$, then $e'$ admits no extensions.
If $r(e')=0$, then $e'$ admits a unique extension to some $e \in
\Var_n(\mathcal{I})$.

The case of interest is $r(e')>0$, meaning that $e'$ admits many extensions, so
the initial-value problem from $e'$ to $e = e' + \pair{v}$ is ambiguous.  For any $e \in
M^{(1)}$, we can identify a hyperplane $e' \in \Gr_{n-1}(e)$ with $\xi \in
\mathbb{P}e^*$ via $e' = \ker \xi$.  Because $e \in M^{(1)} \subset \Gr_n(TM)$
where $n$ is the maximal dimension of integral elements of $\mathcal{I}$, the
function $r$ cannot be positive on an open set of $\mathbb{P}e^*$, so the case
$r(e')>0$ is a closed condition.  Moreover, the function $r:\mathbb{P}e^* \to
\mathbb{N}$ is the rank of a linear system of equations, so it defines a
Zariski-closed projective algebraic variety.  We choose to study that algebraic variety
projectively over $\mathbb{C}$.  Hence, the typical definition of the \emph{characteristic variety of $e$} is 
\begin{equation} 
\Xi_e = \{ \xi \in \mathbb{P}e^*\otimes \mathbb{C} : r(\xi^\perp) > 0 \} \subset
V^*_e.
\label{eqn:Xi0}
\end{equation}
This initial definition is refined in Section~\ref{sec:incidence} to produce a
scheme.  To study properly the ambiguity of the initial-value problem, we want to
assign a multiplicity to each $\xi \in \Xi_e$ and decompose $\Xi$ into
irreducible components based on the structure of the space $H(\xi^\perp)$. 

\subsection{via Rank-one Incidence}\label{sec:incidence}
For both computational and theoretical purposes, it would be convenient to 
tie the polar space $H(e')$ to the geometry of the tableau $A_e$ of an
extension $e$ of $e'$. 
The discussion of polar pairs in Section~\ref{sec:polarpairs} links these two
objects, to provide another interpretation of the initial-value problem that is
much more convenient than \eqref{eqn:Xi0}.

Fix $e \in M^{(1)}$, and suppose that both $e$ and $\tilde{e}$ are integral
extensions of $e' = \ker \xi$ for some $\xi \in e^*$.  By the definition of
$H(e')$, it must be that $\tilde{e}$ lies in $\Var_n(\mathcal{I}) \cap
\Pol_1(e)$, 
but we do not know whether $\tilde{e}$ lies in the particular maximal smooth component of
$\Var_n(\mathcal{I})$ that we call $M^{(1)}$.
However, the results of Section~\ref{sec:polarpairs} guarantee that $\tilde{e}$
is detected by $A_e$ even if $\tilde{e}$ is not in $M^{(1)}$, in the following
way.

\begin{lemma}
Fix $e \in M^{(1)}$, and suppose that both $e$ and $\tilde{e}$ are integral
extensions of $e' = \ker \xi$ for some $\xi \in e^*$.  
Let $w$ be such that $\tilde{e} = e' + \pair{w}$.
Then $w \otimes \xi \in A_e$, and there is an open 1-parameter family of
integral extensions of $e'$ near $e$ in $M^{(1)}$ 
that also represent $[w \otimes \xi]$. 
\label{lem:AseesXi}
\end{lemma}
\begin{proof}
Because $\tilde{e} \in \Pol_1(e)$, Lemma~\ref{lem:ranknullity} yields a
particular line $[K]$ of rank-1 homomorphisms in $(T_pM/e) \otimes e^*$ representing
$\tilde{e}$. 
Because $H(e')$ is a vector space\footnote{Here we see again why it is helpful
for an EDS to be finitely generated.} such that $w \in H(e')$ and $w \not\in e$, 
the rank-1 projective homomorphism $[K]$ takes the form of $[w \otimes \xi]$
for some $w \in H(e')/e$.

By Lemma~\ref{lem:connected}, there is a continuous 1-parameter
family of other polar pairs $e_\tau$ of $e$, with $e_\tau \cap e = e'$,
converging to $\tilde{e}$, all of which share the rank-1 projective
homomorphism $[w \otimes \xi]$.

That is, as a line of rank-1 homomorphisms, $[w \otimes \xi]$ is contained in
$(H(e')/e) \otimes e^*$, as a subspace of $(T_pM /e)\otimes e^*$.  
Applying $\arctan_e$, this implies that $e_\tau
\subset H(e')$ for all $\tau$.  By the definition of $H(e')$, this means
$e_\tau \in \Var_n(\mathcal{I})$ for all $\tau$.  But, the $e_\tau$ follow a
continuous curve, and $e_0 = e$ lies in the open subset $M^{(1)}$.  Therefore,
all $e_\tau$ for an open set of sufficiently small $\tau$.  Differentiating, we
see that the line $[w \otimes \xi]$ is contained in the tangent space of the
fiber of $M^{(1)}$ at $e$, namely $A_e$.
\end{proof}

On the other hand, for fixed $e$ and $\xi$, there are various distinct
$\tilde{e}$ corresponding to linearly independent $w$.
With Figure~\ref{fig:polarincidence} in mind, it is easy to see that 
\begin{equation}
\dim \mathbb{P} \{ w \in T_pM/e ~:~ w \otimes \xi \in A_e\} = r(\xi^\perp).
\label{eqn:rxi}
\end{equation}

Recall the rank-1 ideal $\mathscr{R}$ from
Section~\ref{sec:tableau}. Here it applies to vector bundles.  As a set, the rank-1
subvariety of the tableau is
\begin{equation}
\Cone 
=  A \cap \Var \mathscr{R}
= A \cap \{  w \otimes \xi\ :\ w \in W,\ \xi \in V^*  \}. 
\label{eqn:Cone}
\end{equation}
As a set, the \emph{characteristic variety} $\Xi$ is the projection of $\Cone$
to $V^*$.  More precisely, $\Xi$ is the \emph{scheme}\footnote{We must study
$\Xi$ along with its various components and multiplicities, so it is better to
think of it as a scheme than as a simple-minded variety.} defined by
the \emph{characteristic ideal} $\mathscr{M}$ on $V^*$ that is obtained from
the rank-1 ideal $\mathscr{R}$ on $A \subset W \otimes V^*$ in the following way:
For any $\xi \in V^*$, define $\sigma_\xi:W \to H^{1}$ by
$\sigma_\xi(w) = \sigma(w \otimes \xi)$.    Note that $\dim\ker\sigma_\xi =
r(\xi^\perp)$ by \eqref{eqn:rxi} and \eqref{eqn:Cone}, but this does not
account for multiplicity within $\Cone$ itself. Then the scheme $\Cone$ is the incidence correspondence\footnote{
For more background on the utility of incidence correspondences in algebraic
geometry, see the 2013 Columbia Eilenberg lecture series
by Joe Harris, \cite{Harris2013}. A YouTube link is in the bibliography.}
of $\Xi$ for the symbol map $\sigma_\xi$. See
Figure~\ref{fig:incidence}.  
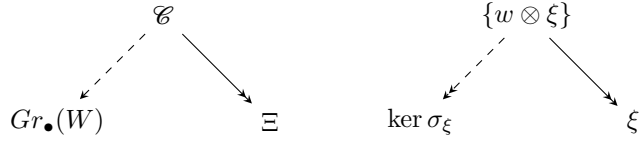
\begin{figure}[ht]
\centering\begin{tikzpicture}[node distance=2cm, auto]
  \node (C) {$\Cone$};
  \node[below left of=C] (G) {$Gr_\bullet(W)$};
  \node[below right of=C] (Xi) {$\Xi$};
  \draw[->>, above right] (C) to node {} (Xi);
  \draw[->, dashed, above left] (C) to node {} (G);
  \node[right of=Xi] (G1) {$\ker \sigma_\xi$};
  \node[above right of=G1] (C1) {$\{ w \otimes \xi \}$};
  \node[below right of=C1] (Xi1) {$\xi$};
  \draw[->>, above right] (C1) to node {} (Xi1);
  \draw[->>, dashed, above left] (C1) to node {} (G1);
\end{tikzpicture}
\caption{The rank-1 variety $\Cone$ is the incidence correspondence for the
characteristic variety $\Xi$, but the scheme multiplicities in $\Xi$ should be obtained as 
in \eqref{eqn:charidealv}.
}\label{fig:incidence}
\end{figure}

This interpretation is \emph{amazing}.  Suddenly, two completely elementary
ideas from Section~\ref{sec:tableau}---tableaux of matrices and rank-1
matrices---come together to give a concise description of the most subtle
structure in PDE theory.  

However, the scheme components and multiplicities are still not obvious from
Figure~\ref{fig:incidence}; they must be obtained by examining the degree of
the equations defining $\ker \sigma_\xi$.  The powerful third interpretation in
Section~\ref{sec:inveig} provides this detail. But first an example.

\subsection{Example: The Wave Equation}\label{sec:exwave}
Consider the PDE $f_{11} + f_{22} = f_{33}$.  To do this, we
consider the manifold $M = \mathbb{R}^{3+1+3+5} \subset
 = \mathbb{R}^{13} = \mathbb{J}^2(\mathbb{R}^3,\mathbb{R})$ with coordinates $x^1$,$x^2$,$x^3$, $f$,
$p_1$,
$p_2$, $p_3$, $p_{11}$, $p_{12}$, $p_{13}$, $p_{22}$, $p_{23}$.
Consider the exterior differential system generated by
\begin{equation}
\begin{split}
\theta^0 &= \mathrm{d}u - p_1 \mathrm{d}x^1 - p_2 \mathrm{d}x^2 - p_3
\mathrm{d}x^3,\\
\theta^1 &=
\mathrm{d}p_1-p_{11}\mathrm{d}x^1-p_{12}\mathrm{d}x^2-p_{13}\mathrm{d}x^3,\\
\theta^2 &=
\mathrm{d}p_2-p_{12}\mathrm{d}x^1-p_{22}\mathrm{d}x^2-p_{23}\mathrm{d}x^3,\\
\theta^3 &= \mathrm{d}p_3-p_{13}\mathrm{d}x^1-p_{23}\mathrm{d}x^2-(p_{11}+p_{22})\mathrm{d}x^3
\end{split}
\end{equation}
Let $\omega^i = \mathrm{d}x^i$ for $i=1,2,3$, so 
the derivatives are computed as 
\begin{equation}
\mathrm{d}
\begin{pmatrix}
\theta^0 \\ \theta^1 \\ \theta^2 \\ \theta^3
\end{pmatrix}
\equiv 
\begin{pmatrix}
0 & 0 & 0\\
\pi^1_1 & \pi^1_2 & \pi^1_3\\
\pi^2_1 & \pi^2_2 & \pi^2_3\\
\pi^3_1 & \pi^3_2 & \pi^3_3
\end{pmatrix}
\wedge
\begin{pmatrix}
\omega^1 \\ \omega^2 \\ \omega^3
\end{pmatrix}
\mod \{ \theta^0, \theta^1, \theta^2, \theta^3\}
\end{equation}
where $\pi^1_2=\pi^2_1$, $\pi^1_3 = \pi^3_1$, $\pi^2_3 = \pi^3_2$, and $\pi^3_3
= \pi^1_1 + \pi^2_2$.

Changing bases, this tableau is equivalent to an endovolutive one of the form
\begin{equation}
(\pi^a_i) = 
\begin{pmatrix}
\alpha_0 & \alpha_3 & \alpha_4 \\
\alpha_1 & \alpha_4 & \alpha_2 + \alpha_3 \\
\alpha_2 & \alpha_0 & \alpha_1 \\
\end{pmatrix}
\end{equation}
Or in block form
\begin{equation}
(\B^\lambda_i) =
\begin{bmatrix}
\begin{pmatrix} 1 & 0 & 0\\  0 & 1 & 0\\ 0 & 0 & 1 \end{pmatrix} &
\begin{pmatrix} 0 & 0 & 0\\  0 & 0 & 0\\ 1 & 0 & 0 \end{pmatrix} &
\begin{pmatrix} 0 & 0 & 0\\  0 & 0 & 1\\ 0 & 1 & 0 \end{pmatrix}\\
\begin{pmatrix} 0 & 0 & 0\\  0 & 0 & 0\\ 0 & 0 & 0 \end{pmatrix} &
\begin{pmatrix} 1 & 0 & 0\\  0 & 1 & 0\\ 0 & 0 & 0 \end{pmatrix} &
\begin{pmatrix} 0 & 1 & 0\\  1 & 0 & 0\\ 0 & 0 & 0 \end{pmatrix}
\end{bmatrix}
\end{equation}
Note that the third row of both 
$\B^1_2\B^1_3 - \B^1_3\B^1_2$ and 
$\B^1_2\B^2_3 - \B^1_3\B^2_2$ are zero, so the tableau is involutive by
Theorem~\ref{thm:invcond}.

The rank-1 condition is
\begin{equation}
\begin{split}
0 &= \alpha_0 \alpha_4 - \alpha_1 \alpha_3, \\ 
0 &= \alpha_0 \alpha_0 - \alpha_2 \alpha_3, \\
0 &= \alpha_0 \alpha_1 - \alpha_2 \alpha_4, \\
0 &= \alpha_1 \alpha_1 - \alpha_2 \alpha_2 - \alpha_2 \alpha_3,\\
0 &= \alpha_3 \alpha_1 - \alpha_0 \alpha_4, \\
0 &= \alpha_3 \alpha_2 + \alpha_3 \alpha_3 - \alpha_4 \alpha_4,\ \text{and} \\
0 &= \alpha_4 \alpha_1 - \alpha_0 \alpha_2 - \alpha_0 \alpha_3.\\
\end{split}
\end{equation}

After a simple change of basis, this becomes the example \eqref{eqn:example_tab}
-- \eqref{eqn:C}, seen throughout the earlier sections.

\section{Guillemin Normal Form and Eigenvalues}\label{sec:inveig}
In this section, we reinterpret $\Cone$ and $\Xi$ as properties of
the endomorphisms $\B^\lambda_i$.  This section is the key to all of the more
advanced results that follow.
Our main computation tool is the structure of an endovolutive tableau discussed
in Section~\ref{sec:endo}, where $W$ and $V$ and $A$ are now 
bundles over $M^{(1)}$.  

The incidence correspondence of Figure~\ref{fig:incidence}
is rephrased in Lemma~\ref{lem:eigen}.
\begin{lemma}
If $\xi \in \Xi$, $v \in V$, and $w \in \ker \sigma_\xi \subset
W$, then
\begin{equation}
\B(\xi)(v)w = \xi(v) w.
\label{eigen}\end{equation}
In particular, $w$ is an eigenvector of $\B(\xi)(v)$ for all $v$.
\label{lem:eigen}
\end{lemma}
\begin{proof}
Fix generic bases $(u^i)$ and $(z_a)$ and $(u_i)$, so that $\xi = \xi_i u^i$ and $w = w^a
z_a$ and $v = v^i u_i$.  
Set $\pi = w \otimes \xi \in \Cone \subset A$, so $\pi^a_i = w^a\xi_i$ for all $a,i$, and this
$\pi$ must satisfy the symbol relations \eqref{eqn:symrels}.
In particular, 
$w^a\xi_i = B^{a,\lambda}_{i,b}w^b\xi_\lambda$ for $a > s_i$.
Therefore
\begin{equation}
\begin{split}
\B(\xi)(v)w &= 
\sum_{a \leq s_i} \xi_i v^i w^a z_a + 
\sum_{a>s_i} B^{a,\lambda}_{i,b} w^b \xi_\lambda v^i  z_a\\
&=
\sum_{a \leq s_i } \xi_i v^i  w^az_a+ \sum_{a > s_i} \xi_i v^i w^az_a\\
&=
\sum_{a, i}  \xi_i v^i w^az_a = \xi(v) w.
\end{split}
\end{equation}
(Here we see the utility of including the first summand in
Equation~\eqref{eqn:IB}.)
\end{proof}

Recalling the decomposition \eqref{eqn:VUY} and \eqref{eqn:VUYdual}, 
Lemma~\ref{lem:backeigen} provides a sort of converse of Lemma~\ref{lem:eigen}.
\begin{lemma}
Suppose that $A$ is an endovolutive tableau.
Fix $\varphi \in Y^\perp \cong U^*$ and suppose that $w \in \Wu^-(\varphi)$ is
an eigenvector of $\B(\varphi)(v)$ for every $v \in V$.  Then there
is a $\xi \in \Xi$ over $\varphi \in Y^\perp$ such that $w  \in
\Wu^1(\varphi)$, so $w \otimes \xi \in A$.
\label{lem:backeigen}
\end{lemma}
\begin{proof}
For each $v \in V$, let $\xi(v)$ denote the eigenvalue corresponding to $v$, so that
$\xi(v)w=\B(\varphi)(v)w$.  Because $\B(\varphi)(v)w$ is linear in $v$, so is
$\xi(v)$.  Then $\xi = \xi_i u^i \in V^*$.
Therefore, $\B(\varphi)(\cdot)w = w \otimes \xi$.  In particular, 
the rank-1 condition implies that
\begin{equation}
\sum_{\lambda \leq \mu} \varphi_\lambda \B^\lambda_\mu w
= \xi_\mu w 
=
\sum_{\lambda \leq \mu} \xi_\lambda \B^\lambda_\mu w,\quad \forall \mu \leq \ell.
\label{eqn:Wup2}
\end{equation}
This is the same expression as in \eqref{eqn:Wup}, so by comparing
recursively over $\mu=1,2,\ldots,\ell$, we see that
$\xi_\lambda = \varphi_\lambda$ for all $\lambda$, so $w \in \Wu^1(\varphi)
\subset \Wu^-(\varphi)$.
\end{proof}
Lemma~\ref{lem:backeigen} deserves a warning:  
There may be multiple $\xi$ over the same $\varphi$, for perhaps there are
different eigen\emph{vectors} $w \in \Wu^-(\varphi)$ admitting different sequences of eigen\emph{values}
$\xi_\varrho$, for $\varrho > \ell$, associated to the same $\varphi$.
Moreover, it is
not (yet) clear that a mutual eigenvector $w$ exists for every such $\varphi$.

But overall it is clear that there is some relationship
between the eigenvalues of $\B^\lambda_i$ and the characteristic variety of an
endovolutive tableau $A$.  This relationship is made precise for involutive
tableau using a result from \cite{Guillemin1968}.  

\begin{thm}[Guillemin normal form]
Suppose that $A$ is involutive. 
For every $\varphi \in Y^\perp$ and $v \in V$, the restricted homomorphism
$\B(\varphi)(v)|_{\Wu^1(\varphi)}$ is an endomorphism of $\Wu^1(\varphi)$. 
Moreover, for all $v, \tilde{v} \in V$, 
\begin{equation} \left[ 
\B(\varphi)(v), 
\B(\varphi)(\tilde{v})\right]\Big|_{\Wu^1(\varphi)}
=0.
\label{eqn:gnf1}
\end{equation}
\label{thm:gnf1}
\end{thm}
Compare Theorem~\ref{thm:gnf1} to Lemma 4.1 in \cite{Guillemin1968} and Proposition 6.3 in Chapter VIII of \cite{BCGGG}.  
Theorem~\ref{thm:gnf1} is known as \emph{Guillemin normal form} because it implies that
the family of homomorphisms $\B(\varphi)(\cdot)$ can be placed in simultaneous
Jordan normal form on $\Wu^1(\varphi)$.  It is the ``normal form'' alluded to
in Section~\ref{sec:regular}.
We defer the proof of Theorem~\ref{thm:gnf1} to Section~\ref{sec:gnf} so we may first see its important consequences.

\begin{cor}
If $A$ is involutive, then for each $\varphi \in Y^\perp$, 
there exists some $w$ satisfying the hypotheses of Lemma~\ref{lem:backeigen}.
That is, the projection map $\Xi \to Y^\perp$ is onto.
In particular, if $A$ is nontrivial and involutive, then $\Xi$ is nonempty.
\label{cor:invbackeigen}
\end{cor}
\begin{proof}
Because we are working over $\mathbb{C}$, the commutativity condition 
\eqref{eqn:gnf1} guarantees that common eigenvectors
exist for  the commutative algebra $\{ \B(\varphi)(v)~:~ v \in V\}$.
\end{proof}

\begin{lemma}
Suppose that $A$ is an involutive tableau. 
Then the map of projective varieties induced by $\Xi \to Y^\perp$ is a \emph{finite} branched
cover.  In particular, both $\hat\Xi$ and $Y^\perp$ have affine fiber dimension
$\ell$.
\label{lem:dimXi}
\end{lemma}
\begin{proof}
Fix $\varphi \in Y^\perp$.  The set of $\xi$ over $\varphi$ is nonempty by
Corollary~\ref{cor:invbackeigen}.
If it were true that the set of $\xi$ projecting to a particular $\varphi$ were infinite,
then the parameter $\xi_i$ would take infinitely many values in some expression
of the form
\begin{equation}
\det \left( \sum_{\lambda} \varphi_\lambda \B^\lambda_i 
- \xi_i I \right) =0.
\end{equation}
But, the matrix $\sum_\lambda \varphi_\lambda \B^\lambda_i \in \End(\Wu^-_1)$
can have at most $s_1$ eigenvalues.
\end{proof}

Here we arrive at an easy\footnote{It is easy in the sense that we have the explicit polynomials of $\mathscr{M}$ in
hand, and they are recognizable as the familiar eigenvector equations.
The reader should compare \eqref{eqn:charidealv} to the descriptions
provided in \cite{BCGGG} and \cite{Ivey2003}.  Both references defer their
decomposition of $\Xi$ to the abstract Grothendeick--Riemann--Roch theorem.
Hence, neither reference indicates how to compute the scheme by hand for
general tableaux.  While details are given in \cite{BCGGG} in the simple case
of rectangular tableaux, a complete description is achieved here because of the
normal form provided by Theorem~\ref{thm:invcond}.}
proof of the main theorem regarding the structure of
$\Xi$.
\begin{thm}
If $A$ is involutive, then $\dim \Xi = \ell-1$ and $\deg \Xi = s_\ell$. 
\label{thm:degXi}
\end{thm}
\begin{proof}
We work in endovolutive coordinates.
From Lemma~\ref{lem:dimXi}, we already know that $\dim\Xi=\ell-1$.

Fix a generic point $\xi \in \Xi$ over $\varphi \in Y^\perp$.  
Let $\Cone_\xi = (\ker \sigma_\xi) \otimes \xi$ denote the fiber over $\xi$ in
$\Cone$.
To understand the scheme $\Xi$, we must determine the degree of the condition
defining $\Cone_\xi$. Note that $\Cone_\xi$ must be a subvariety of $\Wu^1(\varphi)\otimes \xi$, and
$\Wu^1(\varphi)$ is a linear subspace of $W$, so the degree of $\Xi$ is the
degree of some condition on $\Wu^1(\varphi)$.

By Lemma~\ref{lem:eigen} and \eqref{eqn:Cone}, the condition that $\Cone_\xi$ is nontrivial is precisely the condition 
that 
\begin{equation}
\det \left(\sum_{\lambda} \xi_\lambda \B^\lambda_i - \xi_i I\right) =0,\quad
\forall i.
\label{eqn:charideal}
\end{equation}
Since we may restrict our attention to $\Wu^1(\varphi)\otimes \xi$, the
condition \eqref{eqn:charideal} for $i \leq \ell$ is automatic by
\eqref{eqn:Wupeig}.  Hence,
only
these terms contribute to the non-linear part of the ideal:
\begin{equation}
\det \left(\sum_{\lambda} \xi_\lambda \B^\lambda_\varrho - \xi_\varrho I\right)=0,\quad
\forall \varrho > \ell.
\label{eqn:charidealX}
\end{equation}
So, without coordinates, the defining equations of $\Cone_\xi$ are
\begin{equation}
\det \left( \B(\xi)(v) - \xi(v) I\right) = 0,\quad \forall v \in Y.
\label{eqn:charidealv}
\end{equation}
For a particular $v$, this is the characteristic polynomial of
$\B(\xi)(v)$ as an endomorphism of $\Wu^1(\varphi)$. 
By involutivity and Theorem~\ref{thm:gnf1}, all $\B(\xi)(v)$ for $v \in
Y$ admit the same Jordan-block form, so they admit the same factorization type for their respective
characteristic polynomials.   That means it suffices to consider a single $v$. 
By definition, the characteristic polynomial of $\B(\xi)(v)|_{\Wu^1(\varphi)}$ has degree
$\dim \Wu^1(\varphi)$ at generic $\varphi$.
Therefore, $\deg \Xi = s_\ell$ follows from Lemma~\ref{lem:dimW1}.
\end{proof}

Theorems~\ref{thm:gnf1} and \ref{thm:degXi} provide a powerful interpretation
of the form of an involutive tableau seen in Theorem~\ref{thm:invcond} and
Figure~\ref{fig:figBend}; the first $\ell$ columns represent a projection of
$\Xi$, as in Lemma~\ref{lem:dimXi}, and the rank-1 incidence correspondence
in Figure~\ref{fig:incidence} is precisely the eigenvector condition on the
appropriate subspaces.  It is peculiar and interesting that these results were
discovered in the opposite order historically, as explored in
Section~\ref{sec:gnf}.

The proof of Theorem~\ref{thm:degXi}---in particular
Equation~\eqref{eqn:charidealv}---gives a precise understanding of $\Xi$ as a scheme.
Specifically, the characteristic scheme (in the sense of PDE) is merely a scheme of
characteristic equations (in the sense of linear algebra)!
The components of $\Xi$ correspond to the various Jordan blocks apparent in
\eqref{eqn:charidealv}. The multiplicity of each component is the dimension
of that generalized eigenspace. The sheets of the finite branched cover $\Xi \to
Y^\perp$ come from different generalized eigenspaces where the first $\ell$
eigenvalues match.  See Section~\ref{sec:examples} for how to compute this.

\section{Examples}\label{sec:examples}
\subsection{Zero-dimensional examples}
Consider some cases of involutive tableaux with
$(s_1, s_2, s_3) = (4,0,0)$.
\begin{equation}
(\pi^a_i) =
\begin{pmatrix}
\colorbox{io}{$\pi^1_1$} & \pi^1_2 & \pi^1_3  \\
\colorbox{io}{$\pi^2_1$} & \pi^2_2 & \pi^2_3  \\
\colorbox{io}{$\pi^3_1$} & \pi^3_2 & \pi^3_3  \\
\colorbox{io}{$\pi^4_1$} & \pi^4_2 & \pi^4_3  \\
\end{pmatrix}.
\end{equation}
Or, in endovolutive block form: 
\begin{equation}
(\B^\lambda_i) = 
\begin{bmatrix}
I_4 & \B^1_2 & \B^1_3 
\end{bmatrix}.
\end{equation}
The characteristic ideal $\mathcal{M}$ will have degree $s_\ell=4$
and projective dimension $\ell-1=0$.  That is, $\Xi$ will be 4 points, counted
with multiplicity.
The involutivity condition is $0 = \B^1_2\B^1_3 - \B^1_3\B^1_2$ (all rows);
that is, the matrices commute.
Thus the matrices $\B^1_2$ and $\B^1_3$ must have compatible Jordan-block
forms; they span a commutative algebra.
In these examples, we will use colors to emphasize the distinct generalized
eigenspaces.

One possibility is that the matrices are diagonal with distinct Jordan blocks:
\begin{equation}
A  = 
\left\{\begin{pmatrix}
\alpha_1 & \la\alpha_1  & \ma \alpha_1 \\
\alpha_2 & \lb\alpha_2  & \mb \alpha_2 \\
\alpha_3 & \lc\alpha_3  & \mc \alpha_3 \\
\alpha_4 & \ld\alpha_4  & \md \alpha_4 
\end{pmatrix}~:~ \alpha_a \in \mathbb{C}\right\}.
\end{equation}
In this case, the rank-1 variety is
\begin{equation}
\Cone = \left\{
\begin{bmatrix}\textcolor{c1}{1}\\0\\0\\0\end{bmatrix}\otimes[1:\la:\ma],\ 
\begin{bmatrix}\textcolor{c2}{1}\\0\\0\end{bmatrix}\otimes[1:\lb:\mb],\ 
\begin{bmatrix}0\\0\\\textcolor{c3}{1}\\0\end{bmatrix}\otimes[1:\lc:\mc],\  
\begin{bmatrix}0\\0\\0\\\textcolor{c4}{1}\end{bmatrix}\otimes[1:\ld:\md]
\right\}.
\end{equation}
Each point $\xi \in \Xi$ has multiplicity 1.

Another possibility is that they are diagonal, but there is an two-dimensional
eigenspace.
\begin{equation}
A = 
\left\{
\begin{pmatrix}
\alpha_1 & \la\alpha_1  & \ma \alpha_1 \\
\alpha_2 & \la\alpha_2  & \ma \alpha_2 \\
\alpha_3 & \lc\alpha_3  & \mc \alpha_3 \\
\alpha_4 & \ld\alpha_4  & \md \alpha_4 
\end{pmatrix}~:~ \alpha_a \in \mathbb{C}\right\}.
\end{equation}
In this case, the rank-1 cone is 
\begin{equation}
\Cone = \left\{
\begin{bmatrix}\textcolor{c1}{*}\\\textcolor{c1}{*}\\0\\0\end{bmatrix}\otimes[1:\la:\ma],\ 
\begin{bmatrix}0\\0\\\textcolor{c3}{1}\\0\end{bmatrix}\otimes[1:\lc:\mc],\ 
\begin{bmatrix}0\\0\\0\\\textcolor{c4}{1}\end{bmatrix}\otimes[1:\ld:\md]\right\}.
\end{equation}
One point $\xi \in \Xi$ has multiplicity 2; in particular,
the fiber $\ker \sigma_\xi$ for $\xi=[1:\la:\ma]$ should be seen as a
$\mathbb{P}^1$.
This is reflected clearly in
\eqref{eqn:charidealv}, because $\xi = [\xi_1:\xi_2:\xi_3]=[1:\la:\ma]$ is a
root of degree 2 for any $v$:
\begin{equation}
\begin{split}
0 &= 
\det\left(\xi_1(v^2\B^1_2 + v^3\B^1_3)  - (\xi_2v^2 + \xi^3v_3)I \right)\\
&=
\left|
v^2\begin{pmatrix}
\la{-}\la & 0 & 0 & 0 \\
0 & \la{-}\la & 0 & 0 \\
0 & 0 & \lc{-}\la & 0 \\
0 & 0 & 0 & \ld{-}\la
\end{pmatrix}
+
v^3\begin{pmatrix}
\ma{-}\ma & 0 & 0 & 0 \\
0 & \ma{-}\ma & 0 & 0 \\
0 & 0 & \mc{-}\ma & 0 \\
0 & 0 & 0 & \md{-}\ma
\end{pmatrix}
\right|\\
&= v^2(\la - \la)^2(\lc-\la)(\ld-\la) 
  +v^3(\ma - \ma)^2(\mc-\ma)(\md-\la).
\end{split}
\end{equation}

Another possibility is that there is a $2\times 2$ block:
\begin{equation}
A = 
\left\{
\begin{pmatrix}
\alpha_1 & \la\alpha_1 +\alpha_2 & \ma \alpha_1 + \alpha_2 \\
\alpha_2 & \la\alpha_2  & \ma \alpha_2 \\
\alpha_3 & \lc\alpha_3  & \mc \alpha_3 \\
\alpha_4 & \ld\alpha_4  & \md \alpha_4 
\end{pmatrix}~:~ \alpha_a \in \mathbb{C}\right\}.
\end{equation}
In this case, the rank-1 cone is 
\begin{equation}
\Cone = \left\{
\begin{bmatrix}\textcolor{c1}{1}\\\textcolor{c1}{0}\\0\\0\end{bmatrix}\otimes[1:\la:\ma],\ 
\begin{bmatrix}0\\0\\\textcolor{c3}{1}\\0\end{bmatrix}\otimes[1:\lc:\mc],\ 
\begin{bmatrix}0\\0\\0\\\textcolor{c4}{1}\end{bmatrix}\otimes[1:\ld:\md]\right\}.
\end{equation}
Note that the fiber over of $\mathscr{C}$ over $\Xi$ has dimension 1 in each
case; however, the first point has multiplicity 2.
We see that the dimension of the fiber is insufficient to measure the
multiplicity of the scheme $\Xi$, because the incidence correspondence involves
the ideal $\mathscr{R}$.  We can see this because of the structure of the
rank-1 matrices:  the upper $2 \times 2$ minors vanish if and only if
$\alpha_2\alpha_2=0$, 
so the fiber $\ker \sigma_\xi$ for $\xi=[1:\la:\ma]$ should be seen as a $\mathbb{P}^0$ of degree 2.
This is reflected clearly in
\eqref{eqn:charidealv}, because $\xi = [\xi_1:\xi_2:\xi_3]=[1:\la:\ma]$ is a
root of degree 2 for any $v$:
\begin{equation}
\begin{split}
0 &= 
\det\left(\xi_1(v^2\B^1_2 + v^3\B^1_3)  - (\xi_2v^2 + \xi^3v_3)I \right)\\
&=
\left|
v^2\begin{pmatrix}
\la{-}\la & 1 & 0 & 0 \\
0 & \la{-}\la & 0 & 0 \\
0 & 0 & \lc{-}\la & 0 \\
0 & 0 & 0 & \ld{-}\la
\end{pmatrix}
+
v^3\begin{pmatrix}
\ma{-}\ma & 1 & 0 & 0 \\
0 & \ma{-}\ma & 0 & 0 \\
0 & 0 & \mc{-}\ma & 0 \\
0 & 0 & 0 & \md{-}\ma
\end{pmatrix}
\right|\\
&= v^2(\la - \la)^2(\lc-\la)(\ld-\la) 
  +v^3(\ma - \ma)^2(\mc-\ma)(\md-\la).
\end{split}
\end{equation}

Finally, consider the case where both types of multiplicity occur.
For example, 
\begin{equation}
A = 
\left\{
\begin{pmatrix}
\alpha_1 & \la\alpha_1 +\alpha_2 & \ma \alpha_1 + \alpha_2 \\
\alpha_2 & \la\alpha_2  & \ma \alpha_2 \\
\alpha_3 & \la\alpha_3  & \ma \alpha_3 \\
\alpha_4 & \ld\alpha_4  & \md \alpha_4 
\end{pmatrix}~:~ \alpha_a \in \mathbb{C}\right\}.
\end{equation}
In this case, the rank-1 cone is 
\begin{equation}
\Cone = \left\{
\begin{bmatrix}\textcolor{c1}{*}\\\textcolor{c1}{0}\\\textcolor{c1}{*}\\0\end{bmatrix}\otimes[1:\la:\ma],\ 
\begin{bmatrix}0\\0\\0\\\textcolor{c4}{1}\end{bmatrix}\otimes[1:\ld:\md]\right\}.
\end{equation}
The scheme structure of $\Xi$ is apparent here.  
The point $\xi = [1:\la:\ma]$
appears in two components, which correspond to the factorization of  
\begin{equation}
\begin{split}
0 &= 
\det\left(\xi_1(v^2\B^1_2 + v^3\B^1_3)  - (\xi_2v^2 + \xi^3v_3)I \right)\\
&=
\left|
v^2\begin{pmatrix}
\la{-}\la & 1 & 0 & 0 \\
0 & \la{-}\la & 0 & 0 \\
0 & 0 & \la{-}\la & 0 \\
0 & 0 & 0 & \ld{-}\la
\end{pmatrix}
+
v^3\begin{pmatrix}
\ma{-}\ma & 1 & 0 & 0 \\
0 & \ma{-}\ma & 0 & 0 \\
0 & 0 & \ma{-}\ma & 0 \\
0 & 0 & 0 & \md{-}\ma
\end{pmatrix}
\right|\\
&= v^2(\la - \la)^2(\la-\la)(\ld-\la) 
  +v^3(\ma - \ma)^2(\ma-\ma)(\md-\la).
\end{split}
\end{equation}
From the perspective of $\Cone$, these components correspond to the rank-1
matrices 
\begin{equation}
\begin{pmatrix}
\alpha_1 & \la\alpha_1  & \ma \alpha_1 \\
\textcolor{c1}{0} &  \textcolor{c1}{0}  & \textcolor{c1}{0} \\
\alpha_3 & \la\alpha_3  & \ma \alpha_3 \\
0 & 0   & 0 
\end{pmatrix}.
\end{equation}
The fiber should be seen as two components,  a $\mathbb{P}^1$ and 
a $\mathbb{P}^0$.  Overall, this point has multiplicity 3.

\begin{rmk}
For readers interested in hydrodynamic integrability criteria, take a moment to
compute the secant varieties $\Sec_k(\Cone)$ and $\Sec_k(\Xi)$, $k=2,3$, in each of these cases.  
The secant variety is all linear combinations of $k$ points from
the given variety.
One can consider both the embedded secant variety within $A$ and $V^*$,
respectively, as well as the Grassmannian secant variety within $\Gr_k(A)$ and
$\Gr_k(V^*)$, respectively
Note that
hyperbolic systems of conservation laws have $s_1=n$ and take the
non-degenerate diagonal form of the first example, over $\mathbb{R}$. 
\end{rmk}

\subsection{One-dimensional examples}
Consider an involutive tableau with
$(s_1, s_2, s_3) = (2,1,0)$.
\begin{equation}
(\pi^a_i) =
\begin{pmatrix}
\colorbox{io}{$\pi^1_1$} & \colorbox{io}{$\pi^1_2$} & \pi^1_3  \\
\colorbox{io}{$\pi^2_1$} & \pi^2_2 & \pi^2_3  \\
\end{pmatrix}.
\end{equation}
Or, in endovolutive block form, 
\begin{equation}
(\B^\lambda_i) =
\begin{bmatrix}
\begin{pmatrix}
1 & 0 \\
0 & 1
\end{pmatrix} & \begin{pmatrix}
0 & 0 \\
x_{0} & x_{1}
\end{pmatrix} &  \begin{pmatrix}
x_{2} & x_{3} \\
x_{4} & x_{5}
\end{pmatrix} \\
\begin{pmatrix}
0 & 0 \\
0 & 0
\end{pmatrix} & \begin{pmatrix}
1 & 0 \\
0 & 0
\end{pmatrix} & \begin{pmatrix}
x_{6} & 0 \\
0 & 0
\end{pmatrix}
\end{bmatrix}.
\end{equation}
The characteristic ideal $\mathcal{M}$ will have degree $s_\ell=1$
and projective dimension $\ell-1=1$.  That is, $\Xi$ will be a single curve.

For the sake of concreteness, let us assume that the coefficients are:
\begin{equation}
(\B^\lambda_i) =
\begin{bmatrix}
\begin{pmatrix}
1 & 0 \\
0 & 1
\end{pmatrix} & \begin{pmatrix}
0 & 0 \\
\frac19 & 0
\end{pmatrix} &  \begin{pmatrix}
5 & 0 \\
1 & 5
\end{pmatrix} \\
\begin{pmatrix}
0 & 0 \\
0 & 0
\end{pmatrix} & \begin{pmatrix}
1 & 0 \\
0 & 0
\end{pmatrix} & \begin{pmatrix}
9 & 0 \\
0 & 0
\end{pmatrix}
\end{bmatrix},
\end{equation}
so that 
\begin{equation}
(\pi^a_i) = 
\begin{pmatrix}
\alpha_{0} & \alpha_{2} & 5 \alpha_{0} + 9 \alpha_{2} \\
\alpha_{1} & \frac{1}{9} \alpha_{0} & \alpha_{0} + 5 \alpha_{1}
\end{pmatrix}.
\end{equation}
The rank-1 ideal is just $\alpha_0\alpha_0 -9 \alpha_1\alpha_2=0$.
Write a generic element of $\Cone$ as $[\alpha_0:\alpha_1:\alpha_2] =
[3\tau:1:\tau^2]$, like so: 
\begin{equation}
\begin{pmatrix}
3\tau & \tau^2 & 15 \tau + 9 \tau^2 \\
1 & \frac{1}{3} \tau & 5 + 3\tau 
\end{pmatrix}.
\end{equation}
Thus, a generic element of $\xi$ is of the form $\xi =
[3:\tau:15+9\tau]$ with
fiber $\begin{bmatrix}3\tau\\1\end{bmatrix}$.

Using \eqref{eqn:charidealv}, the characteristic scheme of $\xi =
[3:\xi_2:\xi_3]$ is generated by
$0 = \det\left( \xi_1 v^3 \B^1_3 + \xi_2 v^3 \B^2_3 - \xi_3v^3 I_2\right)$,
restricted to the space $\Wu^1(\xi)\subset W$, which is 1-dimen\-sional.  Write
$\tau$ for $\xi_2$; so we are trying to find $\xi = [3:\tau:\xi_3]$ over
$\varphi=[3:\tau:0]$ as in Lemma~\ref{lem:backeigen}. The space
$\Wu^1(\varphi)$ is the space spanned by 
$\begin{bmatrix}3\tau\\1\end{bmatrix}$.  Hence, the single linear sheet of the characteristic variety over
$[3:\tau:0]$ is given by $[3:\tau:15+9\tau]$.

\subsection{One-dimensional exercise}
Now is the time go back and re-read the example \eqref{eqn:C} and see how it fits into
Sections~\ref{sec:exwave} and \ref{sec:moduli}. The wave-equation example offers a single $\mathbb{P}^1$ whose fiber is also a
$\mathbb{P}^1$.  By choosing appropriate
coefficients, you should be able to produce
examples with $(s_1,s_2,s_3)=(3,2,0)$ with various other components and multiplicities.

In principle, you can choose any Cartan characters, and choose coefficients subject to Theorem~\ref{thm:invcond} to build examples in this way.
See the Sage code at \url{https://bitbucket.org/curieux/symbol_sage}, which
can generate and analyze any such example (given sufficient memory).

\section{Results of Guillemin and Quillen}\label{sec:gnf}
As in the analogy Section~\ref{sec:regular}, normal forms often
reveal shortcuts to other advanced ideas.

Guillemin's proof of Theorem~\ref{thm:gnf1} made use of two results derived
from Quillen's thesis \cite{Quillen1964}.  In this section, we see how these 
results become easier using Theorem~\ref{thm:invcond}. 
(Note that Theorem~\ref{thm:invcond} and Theorem~\ref{thm:gnf1} are \emph{not}
equivalent.  Theorem~\ref{thm:invcond} is strictly stronger; it is easy to construct endovolutive tableaux that satisfy
the conclusion of \eqref{eqn:gnf1} but are not involutive.  See \cite{Smith2014a}.)

Recall the Spencer cohomology groups from Section~\ref{sec:spencer}.  For any
$\varphi \in V^*$, wedging by $\varphi$ gives a map $W \otimes \wedge^k V^* \to
W \otimes \wedge^{k+1} V^*$.  This induces a map on the quotient spaces,
$H^k(A) \to H^{k+1}(A)$.
\begin{thm}[Quillen's Exactness Theorem]
Suppose $A$ is an involutive tableau, and that $\varphi \not \in \Xi_A$.  
Then the sequence of maps by $\wedge\varphi$,
\[ 
0 \to A \to H^1(A) \to H^2(A) \to \cdots \to H^n(A) \to 0,
\]
is exact.
\label{thm:quillen}
\end{thm}

In \cite{Quillen1964}, this theorem is proven using enormous commutative diagrams.  In our context, with
Theorem~\ref{thm:invcond} in hand, we can prove an easy version of Quillen's
result, in the form of Lemma~\ref{lem:quil}.
Lemma~\ref{lem:quil} is a consequence of Corollary~\ref{cor:guilA}, which for us is an
easy corollary of Theorem~\ref{thm:invcond}.  This corollary is called
Theorem~A in \cite{Guillemin1968}, where it was proved using a large diagram
chase using Quillen's exactness theorem, Theorem~\ref{thm:quillen}. 
\begin{cor}[Quillen, Guillemin]
Consider the subspace $U = \pair{u_1,\ldots, u_{\ell}} \subset V$ for a generic
basis $(u_i)$ of $V$, as in \eqref{eqn:VUY}. 
If $A$ is involutive, then $A|_U$ is involutive, and the natural map
between prolongations $A^{(1)} \to \left(A|_U\right)^{(1)}$ is bijective.
\label{cor:guilA}
\end{cor}
\begin{proof}
The first part is an immediate consequence of Theorem~\ref{thm:invcond}, as the
quadratic condition still holds if the range of indices $\lambda, \mu, i, j$ is
truncated at $\ell$ (or greater).  In particular, the generators $(
\pi^a_\lambda
)_{a \leq s_\lambda}$ of $A$ are preserved.  

The second part is similarly immediate, using the proof of
Theorem~\ref{thm:invcond} given in \cite{Smith2014a}:
the contact relation $\pi^a_\mu = Z^{a}_{\mu,i} u^i$ 
for $a \leq s_\lambda$ gives coordinates $Z^a_{\mu,i}$ to the prolongation $A^{(1)} \subset A \otimes V^*$,
and the $s_1 + 2s_2 + \cdots + \ell s_\ell$ independent generators are
precisely those $Z^a_{\mu, \lambda}$ with $a \leq s_\mu$ and $\lambda \leq \mu
\leq \ell$.  Since they involve no indices $i > \ell$, these generators remain independent when the range of indices is
truncated at $\ell$. 
\end{proof}

Now we come to our simplified version of Theorem~\ref{thm:quillen}.  Compare
Lemma~\ref{lem:quil} to the exact sequence $(3.4)_2$ in  \cite{Guillemin1968}.
\begin{lemma}
Recall that $U^\perp$ is a complement to $Y^\perp \subset V^*$, so that $V^* =
Y^\perp \oplus U^\perp$ as in \eqref{eqn:VUY} and \eqref{eqn:VUYdual}.
For $A$ involutive, the sequence
\[0 \to W \otimes S^2 U^\perp \to H^1 \otimes U^\perp \overset{\delta}{\to} H^2\]
is exact.
\label{lem:quil}
\end{lemma}
\begin{proof}
This proof is just an explicit description of the maps in a basis and
an application of Corollary~\ref{cor:guilA}.  Let $(u^i)$ be a basis for $V^*$
such that $(u^\lambda)$ is a basis for $Y^\perp$ and $(u^\varrho)$ is a basis for
$U^\perp$, using the index convention \eqref{eqn:index} from Section~\ref{sec:tableau}.

The sequence makes sense because we can split the Spencer sequence \eqref{eqn:spencers} as 
$W \otimes V^* = A \oplus H^1$ by identifying the space $H^1$ with $\{ \sum_{a > s_i} \pi^a_i (z_a
\otimes u^i)\} \subset W\otimes V^*$, which is the space spanned by the unshaded entries in
Figure~\ref{fig:figtab}.
Using this identification, two elements $\sum_{a > s_i} \pi^a_i (z_a \otimes u^i)$ and $\sum_{a > s_i}
\hat{\pi}^a_i(z_a \otimes u^i)$ of $W \otimes V^*$ are equivalent in $H^1$ if 
and only if $\pi^a_i - \hat{\pi}^a_i = \sum_{b \leq
s_\lambda}B^{a,\lambda}_{i,b}z^b_i$ for some $\{ z^a_i: a\leq s_i\}$, the
shaded entries in Figure~\ref{fig:figtab}.  In other words, the projection $W
\otimes V^* \to H^1$ is defined by \eqref{eqn:symrels}, and the projection $W
\otimes V^* \to A$ is defined by the projection onto the orange generator
components in Figure~\ref{fig:figtab}, those  $\pi^a_\lambda$ with $a \leq
s_\lambda$.

Since $s_\varrho =0$ for all $\varrho > \ell$, the inclusion $W \otimes U^\perp
\subset W \otimes V^*$ is an inclusion $W \otimes U^\perp \subset H^1$.   Hence, 
the inclusion is understood as \begin{equation}
W \otimes S^2U^\perp \subset ( W \otimes U^\perp ) \otimes U^\perp \subset H^1 \otimes
U^\perp.
\end{equation}

An element of $H^1\otimes U^\perp$ is written in $W \otimes V^* \otimes U^\perp$ as
\begin{equation}
P=
\sum_{a > s_\lambda} P^a_{\lambda,\varsigma} 
(z_a \otimes u^\lambda \otimes u^\varsigma) 
+
\sum_{a>0} P^a_{\varrho,\varsigma} 
(z_a \otimes u^\varrho \otimes u^\varsigma).
\end{equation}
The image $\delta(H^1 \otimes U^\perp)$ in $H^2$ is 
\begin{equation}
\delta( H^1 \otimes V^*) \subset \delta(W \otimes V^* \otimes V^*) \subset W
\otimes \wedge^2 V^*,
\end{equation}
so $\delta P \in W \otimes \wedge^2 V^*$ is of the form
\begin{equation}
\delta P = \sum_{a > s_\lambda} P^a_{\lambda,\varsigma} 
(z_a \otimes u^\lambda \wedge u^\varsigma) 
+
\sum_{a>0} \frac12 \left( P^a_{\varrho,\varsigma}  -
P^a_{\varsigma,\varrho}\right)
(z_a \otimes u^\varrho \wedge u^\varsigma).
\label{eqn:deltaP}
\end{equation}

Recall that $H^2 = \frac{W \otimes \wedge^2 V^*}{\delta_\sigma(A \otimes
V^*)}$. So, $\delta P  \equiv 0$ in $H^2$ if and only
if there is some $T\in A \otimes V^*$ such that $\delta_\sigma(T) = \delta(P)$
    in $W\otimes \wedge^2V^*$.
Looking at \eqref{eqn:deltaP}, it is apparent that such $T$ must have
$\delta_\sigma(T|_U)=0$, as $\delta(P)$ has no $Y^\perp \wedge Y^\perp$ terms.  By involutivity and Corollary~\ref{cor:guilA}, we
consider the involutive tableau
\begin{equation}
0 \to A|_U \to W \otimes Y^\perp \overset{\sigma|_U}{\to} H^1_U \to 0
\end{equation}
with prolongation
\begin{equation}
0 \to \left(A|_U\right)^{(1)} \to A|_U \otimes Y^\perp \overset{\delta_{\sigma}|_U}\to W \otimes \wedge^2Y^\perp
\to H^2_U \to 0.
\end{equation}
Therefore, $T|_U \in A|_U \otimes Y^\perp$ lies in the kernel of
$\delta_{\sigma}|_U$, so $T|_U \in \left(A|_U\right)^{(1)}$.  Therefore, 
Corollary~\ref{cor:guilA} tells us $T \in A^{(1)}$.  That is, 
$\delta(P)\equiv 0 \in H^2$ if and only if $\delta(P)=\delta_\sigma(T)=0$.

Therefore, $\delta(P)\equiv 0 \in H^2$ if and only if $P^a_{\lambda,\varsigma}=0$ and
$P^a_{\varrho,\varsigma}  = P^a_{\varsigma,\varrho}$ on these index ranges.
This occurs if and only if $P=P^a_{\varrho,\varsigma} (z_a \otimes u^\varrho
\otimes u^\varsigma)$, meaning $P \in W \otimes S^2U^\perp$.
\end{proof}

We are ready to prove Theorem~\ref{thm:gnf1}.
The structure of the proof is identical to the original proof in
\cite{Guillemin1968}.
\begin{proof}[Proof of Theorem~\ref{thm:gnf1}]
Suppose that $w \in \Wu^1(\varphi)$, so that $\pi=\B(\varphi)(\cdot)w = w \otimes
\varphi + J$ for some $J \in W \otimes U^\perp$ with 
$J_\varrho = J^a_\varrho z_a \in \Wu^-(\varphi)$ for all $\varrho$. 
First, we must show that the span of the columns $J_\varrho$ of $J$ lies in $\Wu^1(\varphi)$.

Consider the element $-J \otimes \varphi = -J^a_\varrho \varphi_\lambda (z_a
\otimes u^\lambda \otimes u^\varrho) \in H^1 \otimes U^\perp$.  Because $z \otimes
\varphi + J \in A$, it must be that $z \otimes \varphi \otimes \varphi 
\in W\otimes V^* \otimes V^*$ represents the same point in $H^1 \otimes U^\perp$.
So, we can compute
\begin{equation}
-J^a_\varrho \varphi_\lambda (z_a \otimes u^\lambda \wedge
u^\varrho ) \equiv z \otimes \varphi \wedge \varphi = 0 \in H^2.
\end{equation} 
By Corollary~\ref{cor:guilA}, there exists $Q = Q^a_{\varrho,\varsigma}
(z_a \otimes u^\varsigma \otimes u^\varrho) \in W \otimes S^2 U^\perp$ such that $
-J \otimes \varphi - Q \in A \otimes U^\perp$.  That is, writing $Q_{\varrho} =
Q^a_{\varrho,\varsigma} (z_a \otimes u^\varsigma) \in W \otimes Y^\perp$, we have
 $J_\varrho \otimes \varphi + Q_{\varrho}\in A$ for all $\varrho$, meaning $J_\varrho \in \Wu^1(\varphi)$
for all $\varrho$.
Therefore, for any $v \in V$, we have $\B(\varphi)(v)z = \varphi(v)z +
J(v) \in \Wu^1(\varphi)$.

Now, mapping again, $\B(\varphi)(\cdot)J_{\varrho} = J_\varrho \otimes \varphi +
Q_\varrho$, so $\B(\varphi)(u_\varsigma)J_\varrho = Q_{\varrho,\varsigma}$,
which is already known to be symmetric in $\varrho,\varsigma$.  Therefore,
\begin{equation}
\begin{split}
\B(\varphi)(\tilde{v})\B(\varphi)(v)z 
&=
\B(\varphi)(\tilde{v})\left( \varphi(v)z + J(v) \right) \\
&=
\varphi(v) \B(\varphi)(\tilde{v})z  +
u^\varrho(v)\B(\varphi)(\tilde{v})J_\varrho \\
&=
\varphi(v)\left(\varphi(\tilde{v})z  + J(\tilde{v})\right) +
u^\varrho(v)\left( \varphi(\tilde{v})J_\varrho +
Q_{\varrho}(\tilde{v})\right)\\
&=
\varphi(v)\varphi(\tilde{v})z + 
\varphi(v)J(\tilde{v}) + \varphi(\tilde{v})J(v) +
Q(v,\tilde{v}).
\end{split}
\end{equation}
This is symmetric in $v,\tilde{v}$, giving the commutativity condition
\eqref{eqn:gnf1}
\end{proof}

It is interesting to see the inversion of logic that happened here.  In the
original literature, the overall implications are
\[\ref{thm:quillen} \to \ref{lem:quil} \to \ref{cor:guilA} \to \ref{thm:gnf1}.\]
But, the arguments here give the overall implications \[\ref{thm:invcond} \to
\ref{cor:guilA} \to \ref{lem:quil} \to \ref{thm:gnf1}.\]

However, we can write a shorter proof of Theorem~\ref{thm:gnf1} that relies
Theorem~\ref{thm:invcond} more directly, avoiding the general results of
Quillen.
For motivation, consider the following trivial corollary of
Theorem~\ref{thm:invcond} that is obtained by setting $\lambda=\mu$.
\begin{cor}
Suppose an involutive tableau is given in a generic, endovolutive basis as in \eqref{eqn:IB}, so
that Theorem~\ref{thm:invcond} holds.  Then 
$\B(u^\lambda)(v)$ is an endomorphism of
$\Wu^-(u^\lambda)$ such that for all $v, \tilde{v} \in Y$, 
\[ [\B(u^\lambda)(v), \B(u^\lambda)(\tilde{v})] = 0.\]
\label{cor:comm}
\end{cor}

\begin{proof}[Alternate Proof of Theorem~\ref{thm:gnf1}]
Fix $\varphi \in Y^\perp$, and suppose that $w \in \Wu^1(\varphi)$. 
We must verify that all maps $\B(\varphi)(v)$ preserve $\Wu^1(\varphi)$ and
that they commute. 
Note that the definition of $\Wu^1(\varphi)$ in Equation~\ref{eqn:Wupinv}
depends on the choice of subspace $Y^\perp$ but not on its basis, so we may verify
these conditions using any basis we like.

First a trivial case:  if it happens that $\varphi \in \Xi\cap Y^\perp$, then
$\B(\varphi)(v)w = \varphi(v) w \in \Wu^1(\varphi)$ is a rescaling, and it is
immediate that $[\B(\varphi)(v),\B(\varphi)(\tilde{v})]= 0$.

Otherwise, we have $\varphi \not\in \Xi$.  Then we may choose a generic basis of
$V^*$ in which $\varphi = u^1$.  Moreover, we may use that basis to construct
an endovolutive basis of $W$.
By Corollary~\ref{cor:comm}, it suffices to prove in this basis that
$\Wu^1(u^1)$ is preserved by every $\B(u^1)(v)$. 
Write $\B(\varphi)(\cdot)w =  w \otimes u^1 + J$, and
examine \eqref{eqn:Wup} on a column $J_\varrho$ of $J$.  For each
$\mu=1, \ldots, \ell$, we must verify
\begin{equation}
0 = \left( \B^1_\mu - \delta^1_\mu I \right) J_\varrho 
=
\left( \B^1_\mu - \delta^1_\mu I \right) \B^1_\varrho w
=
\left( \B^1_\mu\B^1_\varrho - \delta^1_\mu\B^1_\varrho\right) w.
\end{equation}

If $\mu = 1$, then this is immediate, since $\B^1_1 = I_{s_1}$.

If $\mu \neq 1$, then we are verifying $0=(\B^1_\mu \B^1_\varrho -0)w$.  
Note that $\B^1_\mu w =0$, since $\B(\varphi)(\cdot)w = w \otimes \varphi + J =
w \otimes u^1 + J$.
Moreover, by Theorem~\ref{thm:invcond},  we have
\begin{equation}
0 = \left(\B^1_\mu \B^1_\varrho - \B^1_\varrho \B^1_\mu\right)^a_bw^b 
=\left(\B^1_\mu \B^1_\varrho\right)^a_bw^b 
\end{equation}
for $a > s_\mu$.  Therefore, $\B^1_\mu\B^1_\varrho$ lies in $\Wu^-(\mu)$.
On the other hand, note that the output of $\B^1_\mu$ lies in $\Wu^+_\mu$
by the construction of the maps $\B^\lambda_\mu$ from the reduced symbol in
Section~\ref{sec:endo}.  Combining these, we see that $\B^1_\mu \B^1_\varrho w$
lies in $\Wu^-_\mu \cap \Wu^+_\mu = 0$.

Hence, the space $\Wu^1(\varphi)$ is preserved by $\B(\varphi)(v)$ for all $v$.
By Corollary~\ref{cor:comm}, they commute.
\end{proof}

On the theoretical side, it would be interesting to see how many of the
hard classical theorems in the subject can be re-proven with elementary
techniques.  Specifically, the proof of
Lemma~\ref{lem:quil} suggests an elementary proof of Quillen's exactness
theorem.  The other hard theorem is the integrability of the
characteristic variety, and a proof of that theorem using Guillemin's original
formulation is the subject of \cite{Guillemin1970}.  That result was applied
immediately to study primitive Lie pseudogroups.

\section{Prolongation}\label{sec:prolongation}
How does the characteristic scheme change under prolongation? 
The short answer is that it does not!
This does not depend on endovolutivity or involutivity.

Recall that $A^{(1)}$ is a tableau within $A \otimes V^*$.  An element of
$A^{(1)}$ is $P \in A \otimes V^*$.  Using any bases for $V,W,A$,
we may write $P$ as $P^a_{i,j} z_a \otimes u^i \otimes u^j$, with the
additional condition that $P^a_{i,j} = P^a_{j,i}$ from \eqref{eqn:A2}.
Let $\Cone^{(1)}$ denote the rank-1 elements of $A^{(1)}$, and let
$\Xi^{(1)}$ denote its projection to $V^*$, as in Section~\ref{sec:incidence}.

\begin{thm}
If $\pi \otimes \xi \in \Cone^{(1)}$, then $\pi =
w \otimes \xi \in \Cone$ for some $w \in \ker \sigma_\xi$.  Conversely, if $w
\otimes \xi \in \Cone$, then $(w \otimes \xi) \otimes \xi \in \Cone^{(1)}$.
In particular, $\Xi \cong \Xi^{(1)}$ as schemes.
\label{thm:proXi}
\end{thm}
\begin{proof}
Suppose that $\pi \otimes \xi \in \Cone^{(1)}$ for some $\pi \in A$ and $\xi
\in V^*$.  That is, $P \in A^{(1)}$ and $P=\pi \otimes \xi$, so $P^a_{i,j} = \pi^a_i\xi_j$, and $\pi^a_i\xi_j =
\pi^a_j\xi_i$ for all $a,i,j$.   

Let $\underline{\lambda}$ be the minimum index such that
$\xi_{\underline{\lambda}} \neq 0$.
Then $\pi^a_{\underline{\lambda}}\xi_i = \pi^a_i\xi_{\underline{\lambda}}$, so
column $i$ of $(\pi^a_i)$ is a multiple---namely
$\xi_{i}/\xi_{\underline{\lambda}}$---of column $\underline{\lambda}$ for all
$i$.  Therefore, $(\pi^a_i)$ is rank-1, and there is some $w$ with $\pi = w
\otimes \xi$.  The converse is immediate.
\end{proof}
\begin{rmk}
Theorem~\ref{thm:proXi} is used sometimes as a method for computing the characteristic
variety, as follows:  Given a tableau $(\pi^a_i)$ whose entries might depend on $e \in
M^{(1)}$, consider $(\xi_i) \mapsto (\pi^a_i\xi_j - \pi^a_j\xi_i)$ as a map
$V^* \to W \otimes \wedge^2V^*$; that is, a map from $\mathbb{C}^n$
to $\mathbb{C}^{r\binom{n}{2}}$.  For a general point in $\xi \in V^*$,
this map has rank at least 1.  Its rank falls to 0 if and only if $\xi \in \Xi$.
But, this method is inefficient.  If you have $(\pi^a_i)$ in hand and want to compute
$2\times 2$ minors of something, you would save ink by computing the $2 \times 2$ minors of
$(\pi^a_i)$ itself to find $\Cone$.
\end{rmk}

\section{Characteristic Sheaf}
For a single endomorphism, the characteristic polynomial and the Jordan block
decomposition of generalized eigenspaces together reveal all of the information
that is independent of coordinates.

The ultimate conclusion of the preceding sections is that, 
for an abstract tableau $A$, 
the characteristic sheaf $\mathscr{M}$
knows the dimensions $n$, $r$, $(s_1, \ldots, s_n)$, as well as all of the
dimensions and relationships among the mutual eigenspaces of the various symbol
maps.  The rank-1 cone $\Cone$ knows the algebraic relationships among the
sequences of eigenvalues (which we call $\xi$), and it also knows on which
subspaces the symbol maps commute and on which fail to commute.  
In summary, $\mathscr{M}$ and $\Cone$
together know everything about an abstract tableau $A$ that is independent of
coordinates.\footnote{We revealed this fact using special bases, but as with
traditional Jordan normal form, there is an abstract
structure independent of basis that is easiest to see by building an adapted
basis.}
Moreover, they are invariant under prolongation!  

If the abstract tableau $A$ is a smooth projective bundle, then this applies to 
involutive K\"ahler-regular exterior differential systems in the smooth
category. 

If this formal perspective is appealing, then one might as well dispense with
tableaux, symbols, Grassmann bundles, and differential ideals, and instead study
the sheaf $\mathscr{M}$ directly, with modern algebraic tools such as
\cite{Eisenbud2005}.  Consider $\mathscr{M}$ as an ideal in \[C^\infty(M^{(1)})[u_1,
\ldots, u_n],\] and consider its free resolution.  The Hilbert syzygy theorem
states that there is a finite free resolution that is characterized by its
Hilbert polynomial $h_{\mathscr{M}}(d)$.  Of course,
Theorem~\ref{thm:degXi} is reading the leading term of $h_{\mathscr{M}}(d)$!

One might ask how the involutivity of $A$ can be detected as an algebraic
property of $\mathscr{M}$.  The answer is tied to Castelnuovo--Mumford
regularity, which measures the growth of the Hilbert polynomial.  This
computation is equivalent to the Cartan characters in Cartan's test!

While it is not necessarily a useful \emph{computational} tool versus 
differential forms or tableaux, this perspective allows a broader view of the
techniques in PDE analysis, and it suggests that future progress 
in the field will emphasize on invariant algebraic techniques.

For more on this perspective, see \cite{Malgrange2003}, \cite[Chapter
VIII]{BCGGG}, and the notes by Mark Green from the 2013 conference \emph{New
Directions in Exterior Differential Systems} in Estes Park, Colorado, which are
based on the perspective in \cite{Carlson2009}.

\part{Eikonal Systems}\label{part:eikonal}
In Part~\ref{part:Xi}, we studied the characteristic scheme defined over
$M^{(1)} \subset \Gr_n(TM)$.  In this part, we turn our attention to the
characteristic scheme as pulled back to an integral manifold $\iota:N
\to M$.  This is where the meaning of $\Xi$ as ``directions with an ambiguous
initial value problem'' has clear implications for the internal structure of 
solutions of a differential equation, as the eikonal system yields
intrinsic foliations of integral manifolds $N$.

\section{General Eikonal Systems}\label{sec:eikgeneral}
First, let us consider the general notion of ``eikonal equations'' of a
projective variety, without specific regard to the characteristic variety.

Consider a smooth manifold $N$ of dimension $n$.  Here are three ways to
produce a smooth local hypersurface $H \subset N$.

\begin{enumerate}
\item\label{eik1} The implicit function theorem says that a smooth hypersurface $H \subset N$ is defined locally by a smooth
function $f:N \to \mathbb{R}$, where $T_x H = \ker \mathrm{d}f$ for all $x \in H$.

\item\label{eik2} By the Frobenius theorem, this is equivalent to having a local smooth section
$\varphi$ of $T^*N=\Omega^1(N)$ such that $\mathrm{d}\varphi \equiv 0 \mod \varphi$, for
then $\varphi$ is a rescaling of some $\mathrm{d}f$.  

\item\label{eik3} 
We can also look at the Frobenius theorem from the perspective of
Cartan--K\"ahler theory\footnote{Although Theorem~\ref{thm:CK} applies as
stated only in the analytic category, it can be extended to the
smooth category in this case.  This sort of extension is explored in
Section~\ref{sec:yang}.}, as in Theorem~\ref{thm:CK}.
To make a local function $f:N \to \mathbb{R}$ or a local section $\varphi$ of
$T^*N$, consider the jet space $\mathbb{J}^1(N,\mathbb{R})$, which is
isomorphic to the bundle $T^*N \times \mathbb{R}$.  Jet space is an open
neighborhood (or local linearization) of $\Gr_n(N \times \mathbb{R})$ equipped with
local coordinates $(x^i, p_i, y) = (x^1, \ldots, x^n, p_1, \ldots, p_n, y)$ and
a contact system $\mathcal{J}$ generated by $\Upsilon = \mathrm{d}y - p_i
\mathrm{d}x^i$ and $\mathrm{d}\Upsilon$, as in Section~\ref{sec:contact}. 
In these local coordinates, set the independence condition $\boldsymbol{\omega}
=\mathrm{d}x^1\wedge\cdots\wedge\mathrm{d}x^n \neq 0$.
Any
$n$-dimensional integral manifold of the exterior differential system $(T^*N
\times \mathbb{R}, \mathcal{J}, \boldsymbol{\omega})$ corresponds to a function
$y = f(x^1,\ldots,x^n)$ with $p_i = \frac{\partial f}{\partial x^i}$, so we may
take $\varphi = \mathrm{d}f = \frac{\partial f}{\partial x^i}\mathrm{d}x^i$.
It is easy to see that this exterior differential system has no torsion and has
a K\"ahler-regular tableau with Cartan characters $s_1 =
s_2 = \cdots = s_n = 1$.  That is, integral manifolds are parametrized by 1 function of $n$
variables (hardly a surprise). 
\end{enumerate}

Now, consider a projective subbundle $\Sigma_N \subset \mathbb{P}T^*N$,
meaning it is defined smoothly by homogeneous functions in the local fiber
variables $(p_i)$ of $T^*N$.  
We want a test that tells us whether there exist hypersurfaces $H$ for which
$\mathrm{d}f \in \Sigma_N$ everywhere.  Specifically, we want a theorem like
the following.
\begin{thm}
Suppose that the eikonal system (defined below) of $\Sigma_N$ is involutive.  
Then for any smooth point $[\varphi] \in (\Sigma_N)_{x}$, there is a smooth hypersurface $H
\subset N$ such that $(T_xH)^\perp  = [\varphi]$ and such that
$(T_{\tilde{x}}H)^\perp$ lies in the smooth locus of $(\Sigma_N)_{\tilde{x}}$ for all $\tilde{x} \in H$. 
\label{thm:eikgeneral}
\end{thm}
Because the hypersurface $H$ and the 1-form $\varphi$ are not chosen \emph{a priori}, this condition is
difficult to interpret using the above formulations \ref{eik1} and \ref{eik2} of
hypersurfaces; however, the third formulation on $T^*N \times \mathbb{R}$ is
well-suited to this theorem.  Consider the inclusion $\psi: \hat\Sigma_N \times
\mathbb{R} \to \mathbb{J}^1(N, \mathbb{R})$.  (Recall that $\hat{~}$ indicates
the affine de-projectivization of a projective variety, resulting in a cone.) The \emph{eikonal system} of
$\Sigma_N$ is the exterior differential system $\Eik(\Sigma_N) =
\psi^*(\mathcal{J})$ on $\hat\Sigma_N \times \mathbb{R}$; that is,
$\Eik(\Sigma_N)$ is generated by $\psi^*(\Upsilon)$ and
$\psi^*(\mathrm{d}\Upsilon)$ and has independence condition $\mathrm{d}x^1
\wedge \cdots \wedge\mathrm{d}x^n \neq 0$.  An integral manifold of
$\Eik(\Sigma_N)$ corresponds to a hypersurface in $N$ whose tangent
space in annihilated by a section of $\hat\Sigma_N$. 

We do not prove involutivity of $\mathcal{E}(\Sigma_N)$ in any significant case
here; it is typically extremely deep and difficult, and references are provided
below.  However, the situation in Theorem~\ref{thm:eikgeneral} has several
interesting consequences and interpretations.

\begin{cor}
Suppose that the eikonal system of $\Sigma_N$ is involutive.
Let $\ell-1$ denote the projective fiber dimension of $\Sigma_N$.
The hypersurfaces guaranteed by Theorem~\ref{thm:eikgeneral} depend on
$\ell$ functions of 1 variable.
\label{cor:eikcoords}
\end{cor}
\begin{proof}
Fix $[\varphi] \in (\Sigma_N)_{x}$. We work locally\footnote{In fact, we work
microlocally in the bundle.  Microlocally means that we are working over a
contractible neighborhood of the base space with a local trivialization of the
bundle, and also within a neighborhood in the fiber.}
near
$\varphi$, so we may assume $N$ is open, connected, and simply connected, and
that $T^*N = N \times \mathbb{R}^n$.
Because $\hat\Sigma_N$ is smooth with affine fiber dimension
$\ell$ in $T^*N$, we may choose local coordinates $(q_1, \ldots, q_n)$ on
each fiber of $T^*N$ near $\varphi$ such that $\hat\Sigma_N$ is defined by $q_{\ell+1} =
\cdots = q_n =0$ near $\varphi$.  

For each $\lambda=1,\ldots,\ell$, let $\sigma^\lambda \in (\Sigma_N)_{x}$ denote
the lines of 1-forms specified as \[(0, \ldots, 0, q_\lambda, 0, \ldots 0),\] nonzero in the
$\lambda$ slot, in these
coordinates.  By Theorem~\ref{thm:eikgeneral}, there is a local hypersurface
$H_\lambda \subset N$ and a corresponding local function $x^\lambda$ such that
$\mathrm{d}x^\lambda \sim \sigma^\lambda$.   Complete $x^1, \ldots, x^\ell$ to
a local coordinate system $(x^i)$ on $N$, and let $p_i$ be the canonical
Darboux coordinates (that is, roughly corresponding to $\frac{\partial y}{\partial x^i}$) on 
the fiber of $T^*N$.   Note that $p_i(\mathrm{d}x^\ell) = \delta^\lambda_i$ by
construction, so $\hat\Sigma_N$ is defined by $p_{\ell+1} = \cdots = p_n =0$.
(Note that the open neighborhood of $T^*N$ around $\varphi$ may have shrunk
during this process, which is why this is  microlocal.)

Therefore, the contact system on $T^*N \times \mathbb{R}$ is generated in a
neighborhood of $\varphi$ by
$\Upsilon = \mathrm{d}y - p_i \mathrm{d}x^i$, which pulls back to 
$\hat\Sigma_N \times \mathbb{R}$ as 
\[ \psi^*(\Upsilon) =  \mathrm{d}y - p_\lambda \mathrm{d}x^\lambda. \]
The corresponding tableau is the space of $1\times \ell$ matrices with entries 
$\mathrm{d}p_\lambda$ for $\lambda = 1, \ldots, \ell$, so its has 
$s_1 = s_2 = \cdots = s_\ell = 1$.
\end{proof}

This is an interesting proof, using all three perspectives of hypersurfaces.  The implicit function theorem on the fiber
provides local coordinates on the base by involutivity. Then, the Frobenius theorem 
on the base produces contact coordinates on the fiber that are compatible with the original fiber coordinates.
It is easy to adapt this proof to the following corollary, which is
useful for constructing coordinates in some situations, as in
\cite{Smith2014b}.
\begin{cor}
For any $\Sigma_N$, let $\pair{\Sigma_N}$ denote its linear span, which is
itself a projective subbundle of $\mathbb{P}T^*N$.
If $\Eik(\Sigma_N)$ is involutive, then $\Eik(\pair{\Sigma_N})$
is involutive.
\end{cor}

We will now examine several interpretations of the eikonal system that tie together various
branches of geometry.  Compare Sections~\ref{sec:lagr}, \ref{sec:poisson}, and
\ref{sec:intchar} to \cite[V\S3(vi)]{BCGGG}.

\subsection{as Lagrangian Geometry}\label{sec:lagr}
The $\mathbb{R}$ term in $T^*N \times \mathbb{R}$ plays little role for the
eikonal system $\mathcal{E}(\Sigma_N)$.  It is there merely to make obvious
the relationship between the eikonal equations and hypersurfaces.

Instead, consider the symplectic manifold $T^*N$ with symplectic 2-form
$\mathrm{d}\Upsilon$, which is expressed in local coordinates as
$\mathrm{d}\Upsilon = - \mathrm{d}p_i \wedge \mathrm{d}x^i$ according to
Darboux's theorem.  The \emph{Lagrangian Grassmannian} $LG(N)$ is the
bundle over $T^*N$ whose fiber is all the Legendrian $n$-planes
\begin{equation}
LG_{\varphi}(N) = \{ e \in \Gr_n(T_{\varphi}T^*N)~:~ \mathrm{d}\Upsilon|_e = 0\},
\quad\forall \varphi\in T^*N
\end{equation}
Each fiber is isomorphic to the homogeneous space $LG(n,2n)$, which is the 
variety of $n$-planes in $\mathbb{R}[x^1, \ldots, x^n, p_1, \ldots p_n]$ on
which $\mathrm{d}p_i\wedge \mathrm{d}x^i=0$.
If we consider a plane $e \in LG(n,2n)$ for which $\mathrm{d}x^1 \wedge \cdots
\wedge\mathrm{d}x^n \neq 0$, then $\mathrm{d}p_i = P_{i,j}(e)\mathrm{d}x^i$ on $e$
with $P_{i,j}=P_{j,i}$.  Hence, the non-vertical open neighborhood of
$LG(n,2n)$ is identified with the space of symmetric $n \times n$ matrices,
$\Sym^2(\mathbb{R}^n)$.

Suppose the de-projectivized affine subvariety $\hat\Sigma_N \subset T^*N$ is defined smoothly by
homogeneous functions in the local fiber variables $(p_i)$ of $T^*N$.  From
this perspective, the eikonal system $\mathcal{E}(\Sigma_N)$ is measuring the
intersection of $\Gr_n(T_\varphi\Sigma_N)$ with $LG_\varphi(N)$ for all
$\varphi \in \Sigma_N$.
\begin{cor}
The eikonal system $\mathcal{E}(\Sigma_N)$ is involutive if and only if there are
local coordinates of $T^*N$ near $\varphi \in \hat\Sigma_N$ in which
the non-vertical open set in $\Gr_n(T\Sigma_N) \cap LG(N)$ is described as the $n\times n$
symmetric matrices $P_{i,j}(e)$ that vanish outside the upper-left $\ell \times
\ell$ part.
\label{cor:eiksymp}
\end{cor}
\begin{proof}
If the eikonal system $\Eik(\Sigma_N)$ is involutive, then we may
construct coordinates as in Corollary~\ref{cor:eikcoords} such that the
de-projectivized affine variety $\hat\Sigma_N$
is defined by $p_\varrho=0$ for all $\varrho > \ell$, so $T_\varphi \hat\Sigma_N$
is defined by $\mathrm{d}p_\varrho =0$ for all $\varrho > \ell$.
In such coordinates, the open neighborhood of the Lagrangian Grassmannian takes
the block form
\begin{equation}
\left.\begin{pmatrix}
\mathrm{d}p_\lambda \\ \mathrm{d}p_\varrho
\end{pmatrix}\right|_e
=
\left.\begin{pmatrix}
P_{\lambda,\mu}(e) & P_{\lambda, \varsigma}(e)\\
P_{\varrho,\mu}(e) & P_{\varrho, \varsigma}(e)
\end{pmatrix}
\begin{pmatrix}
\mathrm{d}x^\mu \\  \mathrm{d}x^\varsigma
\end{pmatrix}\right|_e,\ \text{such that $P_{i,j} = P_{j,i}$,}
\end{equation}
using our index convention \eqref{eqn:index} from Section~\ref{sec:tableau}.
The condition $e \in T\Sigma_N$ implies $\mathrm{d}p_\varrho =0$, so the lower
blocks are zero.  The matrix is symmetric, so the upper-right block is zero.

Conversely, suppose such coordinates exist.  Then $T\hat\Sigma_N$ satisfies 
the closed 1-forms $\mathrm{d}p_\varrho=0$, and the dimensions match, so
$\Sigma_N$ satisfies $p_\varrho = \text{constant}$.  Since the equations
defining $\Sigma_N$ are homogeneous, it must be $p_\varrho = 0$.
Using these coordinates for $T^*N \times \mathbb{R}$ and $\mathcal{J}$ yields 
$\psi^*(\Upsilon) = \mathrm{d}y - p_\lambda \mathrm{d}x^\lambda$, as in
Corollary~\ref{cor:eikcoords}, which is
involutive with the correct Cartan characters and gives the desired hypersurfaces in
Theorem~\ref{thm:eikgeneral}.
\end{proof}

Compare this to Proposition~3.22 in \cite[Chapter V]{BCGGG}.
For more symplectic and Lagrangian geometry, see \cite{Bryant1993}.

\subsection{as Poisson Brackets}\label{sec:poisson}
If $T^*N$ describes the state of a physical system, a function $F: T^*N \to
\mathbb{R}$ is called an \emph{observable} \cite{Sattinger1986}.
The \emph{Poisson bracket} of observables is the operation given in local
coordinates by
\begin{equation}
\begin{split}
\poisson{F,G} &= \sum_i \left( \frac{\partial F}{\partial p_i}  \frac{\partial G}{\partial x^i} -  \frac{\partial G}{\partial p_i} \frac{\partial F}{\partial x^i}\right)\\
&=
\sum_i \mathrm{d}F \wedge \mathrm{d}G \left( \frac{\partial}{\partial p_i}, \frac{\partial}{\partial x^i}\right).
\end{split}
\label{eqn:poissonbracket}
\end{equation}
The Poisson bracket plays a fundamental role in Hamiltonian mechanics and the
relationship between symmetries and conservation laws in physics.  This is
because \eqref{eqn:poissonbracket} is a Lie bracket on $C^\infty(T^*N)$. 
(See \cite{Bryant1993} for details.)

Suppose that $O$ is some subspace of $C^\infty(T^*N)$, so that $O$ is a
nonempty set of smooth observables that is closed under linear
combinations. Suppose also that $\poisson{F, G} \in
O$ for all $F, G \in O$.  Then, $O$ is a Lie subalgebra of $C^\infty(T^*N)$
with respect to the Poisson bracket.

Because $\Sigma_N \subset \mathbb{P}T^*N$ is a projective variety in each fiber, the
de-projectivized affine subvariety $\hat\Sigma_N \subset T^*N$ is defined
smoothly by observables that take the form of homogeneous functions in the local fiber variables $(p_i)$ of
$T^*N$.  For convenience, let us make the additional assumption that the
homogeneous functions are algebraic of degree $d$ in $(p_i)$, so that $\hat\Sigma_N$ is defined
smoothly near $\varphi \in \hat\Sigma_N$ for $\varphi \neq 0$ by a set of equations in multi-index form
\begin{equation}
0 = F^\varrho(x,p) = \sum_{|I|=d} f^{\varrho,I}(x)p_I,\ 
\text{for $\varrho = \ell+1, \ldots, n$.}
\label{eqn:Xigens}
\end{equation}
\begin{cor}
Let $O$ denote the module in $S = C^\infty(N)[p_1, \ldots, p_n]$
generated by \eqref{eqn:Xigens}.
The eikonal system $\mathcal{E}(\Sigma_N)$ is involutive if and only if 
$\poisson{O, O} \subset O$. That is, $\mathcal{E}(\Sigma_N)$ is involutive if
and only if the module $O$ is a Lie algebra with respect to the Poisson
bracket.
\label{cor:eikpoisson}
\end{cor}
A proof---which does not depend on the polynomial form \eqref{eqn:Xigens}---can
be derived from Corollary~\ref{cor:eiksymp} along with the observation that the
Poisson bracket can be defined in a coordinate-free way as the operator such that 
\begin{equation}
\poisson{F, G}
(\mathrm{d}\Upsilon)^{\wedge n} = n\, \mathrm{d}F \wedge \mathrm{d}G \wedge
(\mathrm{d}\Upsilon)^{\wedge (n-1)}.
\end{equation}

Equations of the form \eqref{eqn:Xigens} appear in analysis
as systems of homogeneous first-order PDEs on $u:\mathbb{R}^n \to \mathbb{R}$ of the form 
\begin{equation}
0 = F^\varrho(x,u, \nabla u) = \sum_{|I|=d} f^{\varrho,I}(x)
\frac{\partial u}{\partial x^I},\ 
\text{for $\varrho = \ell+1, \ldots, n$.}
\end{equation}
A famous example is the $n{-}\ell = 1$ characteristic equation for the wave equation
of Section~\ref{sec:exwave}: 
\begin{equation}
0 = - (u_t)^2 + c^2( (u_x)^2 + (u_y)^2).
\end{equation}
This is generalized to any involutive EDS in Section~\ref{sec:intchar}.

\section{Involutivity of the Characteristic Variety}\label{sec:intchar}
We would like to apply the entire discussion from Section~\ref{sec:eikgeneral}
to the case where $\Sigma_N$ is a characteristic variety, but first we must
establish that $\Xi$ is well-defined in $\mathbb{P}T^*N$.

Suppose that $\iota:N \to M$ is a connected integral manifold of
an involutive exterior differential system $(M,\mathcal{I})$, and that
$\iota^{(1)}(N)$ lies in 
$M^{(1)}$, a smooth and K\"ahler-regular component of
$\Var_n(\mathcal{I})$, as in Section~\ref{sec:eds}.

Fix $x \in N$, and suppose $\iota(x) = p \in M$ and $\iota^{(1)}(x)=e \in
M^{(1)}$.  For $\xi \in \Xi_e \subset V^*_e$, we can consider the pullback
$\iota^{(1)*}(\xi) \in \mathbb{P}T^*_xN \otimes \mathbb{C}$.  In a basis 
$(\eta^i)$ of $T^*_xN$, we can write a representative as $\xi = \xi_i \eta^i$ for coefficients $\xi_i
\in \mathbb{C}$. As a bundle over $N$, we have $\iota^{(1)*}(\xi) = \xi_i\eta^i \in \mathbb{P}T^*N
\otimes \mathbb{C} = \taut^*_N$.  In this sense, we can pull back
the characteristic variety---as a set---to $N$.  

More precisely, recall that $\Xi$ has degree $s_\ell$ and affine fiber dimension
$\ell$, but it is a scheme defined by the characteristic sheaf $\mathscr{M}$.
For any local section $(u_i)$ of the coframe bundle
$\mathcal{F}_{\taut^*} \to M^{(1)}$, we can write the characteristic sheaf
$\mathscr{M}$ as a homogeneous ideal in the module $C^\infty(M^{(1)})[u_1,
\ldots, u_n]$.    At each $e = \iota^{(1)}(x) \subset M^{(1)}$, the coframe
$(u_i)$ is just a complex basis of $e$. Therefore, we obtain a basis for $T_xN$ of the
form $\eta_i = \left(\iota^{(1)}_*\right)^{-1}(u_i)$. 
That is, in some neighborhood
of $x$, the section $(\eta_i)$ of $\mathcal{F}^*N$ is well-defined.
Moreover the stalks of the sheaf $C^\infty(M^{(1)})$ can be pulled back, as
$\iota^{(1)*}(f)$ is well-defined for any $f$ defined in a neighborhood of $e$.
Therefore, we can pull back both the coefficients and the coordinates to
define the homogeneous ideal $\mathscr{M}_N$ in $C^\infty(N)[\eta_1, \ldots,
\eta_n]$.  Let $\Xi_N \subset \mathbb{P}T^*N \otimes \mathbb{C}$ be the scheme defined by $\mathscr{M}_N$.

Now, the entire discussion from Section~\ref{sec:eikgeneral} applies where
$\Sigma_N$ is any particular component of $\Xi_N$.  We focus our attention on
the maximal smooth locus $\Xi_N^o$ of $\Xi_N$.  We know additionally that $\Xi_N$
takes the polynomial form \eqref{eqn:Xigens} as derived from
\eqref{eqn:charidealv}, so it has degree $s_\ell$ and fiber dimension $\ell-1$
at smooth points, as a complex projective variety.

\begin{thm}[Guillemin--Quillen--Sternberg]
Suppose that $N$ is an ordinary integral manifold of an involutive exterior
differential system $\mathcal{I}$ with character $\ell$ and Cartan integer $s_\ell$.  The eikonal system of the smooth locus of
the (complex) characteristic variety, $\Eik(\Xi_N^o)$, is involutive.  At
smooth points in $\Xi_N$, the characteristic hypersurfaces
are parametrized by 1 function of $\ell$ variables.
\label{thm:intchar1}
\end{thm}
Note that our definition of $\Xi_N$ is the \emph{complex} characteristic
variety.\footnote{Recall that, in the complex case, the distinction between
elliptic and hyperbolic second-order PDEs does not occur, because there is only one nondegenerate
signature.}  Theorem~\ref{thm:intchar1} is called the ``integrability of characteristics.''
Cartan demonstrated several examples of this phenomenon in \cite{Cartan1911}.
The proof appears in \cite{Guillemin1970}, where a major step is the
application of Theorem~\ref{thm:gnf1}.  Hence, this result appears to rely in an
essential way on all three facets of the characteristic variety seen in
Part~\ref{part:Xi}.

The converse of Theorem~\ref{thm:intchar1} is \emph{not} true; it is easy to write down non-involutive
exterior differential systems for which $\Eik(\Xi_N)$ is involutive.

However, in \cite{Gabber1981}, Ofer Gabber proved a more general form of
Theorem~\ref{thm:intchar1} that was conjectured in \cite{Guillemin1970} 
and that removes practically all of the technical
assumptions.  Phrased as Theorem~\ref{thm:intchar2}, Gabber's theorem recalls
the ideas of Section~\ref{sec:poisson}.

\begin{thm}[Gabber]
Let $S$ be a filtered ring whose graded ring $gr(S)$ is a Noetherian
commutative algebra over $\mathbb{Q}$. Let $M$ be a $gr(S)$-ideal that is finitely generated
as an $S$-module.  
Then $\poisson{\sqrt{M}, \sqrt{M}} \subset
\sqrt{M}$
\label{thm:intchar2}
\end{thm}

In our context, Gabber's theorem applies to the case where
$S=C^\infty(N)[p_1, \ldots, p_n]$, the ring of polynomials in local
fiber variables of $T^*N$, filtered by degree.  Then, $gr(S)$ is the ring of
homogeneous polynomials, graded by degree, which admits a Poisson structure 
like \eqref{eqn:poissonbracket}.  The $gr(S)$-ideal $M$ is the characteristic
sheaf $\mathscr{M}_N$, which by \eqref{eqn:charidealv} is defined by homogeneous polynomials if the
original exterior differential system is involutive.  By Hilbert's Nullstellensatz, the radical ideal
$\sqrt{M}$ defines the generic component $\Xi_N^o$. Thus, the conclusion
$\poisson{\sqrt{M}, \sqrt{M}} \subset \sqrt{M}$ invokes
Corollary~\ref{cor:eikpoisson} to say that the eikonal system
$\Eik(\Xi_N^o)$ is involutive.

From the general discussion of eikonal systems surrounding
Theorem~\ref{thm:eikgeneral}, the interpretation of these theorems is apparent,
in the form of Corollary~\ref{cor:eikcor}.
\begin{cor}
Suppose that $N$ is an ordinary integral manifold of an involutive exterior
differential system $\mathcal{I}$ with character $\ell$ and Cartan integer $s_\ell$.
Then $N$ admits a local---possibly complex---coordinate system $(x^1, \ldots
x^n)$ such that $\mathrm{d}x^1, \cdots, \mathrm{d}x^\ell \in \Xi_N$.
\label{cor:eikcor}
\end{cor}

In \cite{Smith2014b}, the linear span of the characteristic variety, $\pair{\Xi_N}$ is studied in
comparison to the Cauchy retraction space $\mathfrak{g}^\perp_N =
\iota^*(\mathfrak{g})$, where 
$\mathfrak{g}^\perp$ is the maximum Frobenius system within $\mathcal{I}$, as in Section~\ref{sec:cauchy}.

Suppose that the affine fiber dimension of $\pair{\Xi_N}$ is $L$ and that the affine
fiber dimension of $\mathfrak{g}^\perp_N$ is $\nu$.    These spaces are nested, so $\ell \leq
L \leq \nu \leq n$.
\begin{cor}
Suppose that $N$ is an ordinary integral manifold of an involutive exterior
differential system $\mathcal{I}$ with character $\ell$ and Cartan integer $s_\ell$.
Then $N$ admits a local---possibly complex---coordinate system $(x^1, \ldots,
x^n)$ such that $\mathrm{d}x^1, \ldots, \mathrm{d}x^\ell \in \Xi_N$, such that 
$\mathrm{d}x^{\ell+1}, \ldots, \mathrm{d}x^L \in \pair{\Xi_N}$, and such that 
$\mathrm{d}x^{L+1}, \ldots, \mathrm{d}x^\nu \in \mathfrak{g}^\perp_N$.
\label{cor:coords}
\end{cor}
Corollary~\ref{cor:coords} is a simple result, but its proof relies on building 
a coframe of $N$ in which the nilpotent parts of the commuting
symbol maps $\B^\lambda_i$ are identified clearly; that is, it depends in an
essential way on Theorems~\ref{thm:intchar1} and \ref{thm:invcond}.
The key point is that it reinforces the following remark.
\begin{rmk}[General Dogma of the Characteristic Variety]
An exterior differential system $(M,\mathcal{I})$ is a geometric object over
$M$, meaning that its key properties are coordinate-invariant.   On each
K\"ahler-regular component $M^{(1)}$, knowing this
geometry is equivalent to knowing the characteristic scheme and rank-1 variety 
over $M^{(1)}$, which are prolongation-invariant.  Moreover, the geometry of an EDS imposes a geometry on its
solutions, $\iota:N \to M$, and this imposition is also dictated by the
characteristic scheme and rank-1 variety.
Therefore, exterior differential systems can be classified up to coordinate equivalence as 
``parametrized families of manifolds $N$ with associated characteristic
geometry.''
\label{rmk:dogma}
\end{rmk}
Remark~\ref{rmk:dogma} is not a theorem; it is an attitude.

To make this remark robust for a general exterior differential
system, the scheme
separating $\Var_n(\mathcal{I})$ into its components $M^{(1)}$---each component smooth
with its own fixed Cartan characters over some subvariety of $M$---would
have to be studied, and very little progress has been made at that level of
abstraction.  Nonetheless, whenever some property of PDEs is encountered,
Remark~\ref{rmk:dogma} urges us to ask ``is this property really invariant, or
an artifact of my coordinates?'' which is best answered by asking ``can this
property be reinterpreted using the characteristic scheme?''
Sections~\ref{sec:yang} and \ref{sec:open} discuss progress of this type.

\section{Yang's Hyperbolicity Criterion}\label{sec:yang}
One of the great frustrations of the Cartan--K\"ahler theorem is that it relies
on the Cauchy--Kowalevski theorem, so it applies only in the analytic category.
One can see its dramatic failure in the smooth category in \cite{Lewy1957a}.
However, this frustration has been escaped in some special cases by exploiting
the structure\footnote{If we
take the broadest possible interpretation of Remark~\ref{rmk:dogma} to heart,
then \emph{any} possible escape from analyticity ought to arise from the structure of $\Xi$.
However, the reader is cautioned again that a dogma is not a theorem.} 
of $\Xi$.  
For example
\begin{enumerate}
\item ODE systems.  Suppose that $(M,\mathcal{I})$ is involutive over $C^\infty$ and
that $\Xi = \emptyset$.  Then $\ell = 0$, so the tableau $A$ is the trivial
(irrelevant) subspace of $W \otimes V^*$. The prolonged system
$\mathcal{I}^{(1)}$ on $M^{(1)}$ is Frobenius, and $M^{(1)}$ is merely a copy
of $M$ whose fiber is the unique element of an integrable distribution.  That
integrable distribution is the Cauchy retraction space $\mathfrak{g}$ of
$\mathcal{I}$ as in Section~\ref{sec:cauchy},
so it must have been that $\mathcal{I} = \mathfrak{g}^\perp$.  The flow-box
theorem foliates $M$ by solutions in the smooth category.  (Actually, in the
Lipschitz category, by standard ODE theory!) If $N$ is a leaf of this
foliation, then removing Cauchy retractions on the original exterior
differential system $(M, \mathcal{I})$ yields the exterior differential system
$(N, 0)$.

\item Empty systems. Suppose that $(M,\mathcal{I})$ is involutive over
$C^\infty$ and that $\Xi = V^*$ with $(s_1, s_2, \ldots, s_n) = (r,r, \ldots,
r)$. Then, the tableau $A$ is the total space $W \otimes V^*$.  Therefore,
$M^{(1)}$ is an open domain in $\Gr_n(TM)$, so $\mathcal{I} = 0$, and there is
no condition whatsoever\footnote{The most extreme and amusing exploitations of
the flexibility of $\Gr_n(TM)$ come from the homotopy principle
\cite{Gromov1986, Eliashberg2002}.} on integral manifolds $\iota:N \to M$;
however, the prolongation $\iota^{(1)}:N \to M^{(1)}$ would have to satisfy the
contact ideal, forcing some regularity on $N$.  We studied this EDS in
Section~\ref{sec:grass}.

\end{enumerate}

A less trivial special case is presented in \cite{Yang1987}, which is the
subject of this section.\footnote{As it happens, the attempt to understand
\cite{Yang1987} in the context of \cite[Chapter VIII]{BCGGG} was the
inspiration for computing the details shown in \cite{Smith2014a} and the entire
approach of these notes.}

A tableau $A \subset W \otimes V^*$ is called \emph{determined} if $s_1 = s_2 = \cdots = s_{n-1}=r$
and $s_n = 0$. That is, $s = (n-1)r$, so $t=r$, and $H^1(A) \cong W$.  Cartan's test shows that a
determined tableau is always involutive, so we may assume that $A$ is written in
endovolutive form as in Theorem~\ref{thm:invcond}. The only nontrivial
symbol endomorphisms in \eqref{eqn:bigB} are $\B^\lambda_\lambda = I_{r\times
r}$ and $\B^\lambda_n$ for $\lambda = 1, \ldots, n-1$, like this:
\begin{equation}
(\B^\lambda_i) = 
\begin{bmatrix}
I_r & 0 & 0 & \cdots & 0 & \B^1_n\\
 & I_r & 0 & \cdots & 0 & \B^2_n\\
 & & \ddots &  & \vdots  & \vdots\\
 & &  & I_r & 0 & \B^{n-2}_n\\
 & &  & & I_r & \B^{n-1}_n
\end{bmatrix}.
\end{equation}
The quadratic
involutivity condition is trivial, which is why Cartan's test passes
automatically.
\begin{lemma}
Suppose $A$ is determined and written in endovolutive bases. 
Identify $H^1(A)$ with $W$, and use our endovolutive basis of $W$ for
both.  Then for any $\varphi \in V^*$, the symbol map $\sigma_\varphi:w \mapsto
\sigma(w \otimes \varphi)$ from
Section~\ref{sec:incidence} is 
\begin{equation}
\sigma_\varphi = \left(\varphi_\lambda \B^\lambda_n - \varphi_n I\right).
\end{equation}
Then \begin{equation}
\ker \sigma_\varphi = \ker \left(\varphi_\lambda \B^\lambda_n
- \varphi_n I\right),
\end{equation}
and the characteristic ideal $\mathscr{M}$ is generated by 
\begin{equation}
\det \sigma_\varphi = \det \left(\varphi_\lambda \B^\lambda_n
- \varphi_n I\right).
\end{equation}
In particular, $\xi \in \Xi$
if and only if $\xi_n$ is an eigenvalue of $\xi_\lambda \B^\lambda_n$.
\label{lem:determined}
\end{lemma}
\begin{proof}
The first two equations are immediate from our block form.
From Part~\ref{part:Xi}, we know that $w \otimes \xi \in A$ if and only if
$\B(\xi)(v)w = \xi(v) w$ for all $v$.  Therefore, we compute in our
endovolutive basis 
\begin{equation}
\begin{split}
\xi(v)w 
&= \B(\xi)(v)w\\
&= \xi_\lambda v^i \B^\lambda_i(w) \\
&= (\xi_\lambda v^\lambda) w + \xi_\lambda v^n \B^\lambda_n w\\
&= ( \xi(v) - \xi_n v^n)w + \xi_\lambda v^n \B^\lambda_n w.
\end{split}
\end{equation}
That is, $\xi_n w = \xi_\lambda\B^\lambda_n  w$.
\end{proof}
\begin{cor}
Consider a determined tableau as in Lemma~\ref{lem:determined}. Fix an integral element $e$.  
Suppose that $e'$ is a \emph{real} hyperplane in $e$ such that $(e')^\perp
\otimes \mathbb{C} = \varphi \in V^*$ and $\varphi \not\in \Xi$.  Then
$\sigma_\varphi:W \to H^1(A)$ is an isomorphism. 
\label{cor:determinedinv}
\end{cor}
\begin{proof}
By Lemma~\ref{lem:determined}, we have $\ker \sigma_\varphi \neq 0$ if and only
if $\varphi \in \Xi$.
\end{proof}
\begin{defn}
Suppose $e'$ is a real hyperplane in $e$ corresponding to the real covector 
$\varphi = (e')^\perp \in \mathbb{P}e^*$.  The real hyperplane is called
\emph{space-like} if the following conditions hold.
\begin{enumerate}
\item $\varphi\otimes \mathbb{C} \not \in \Xi_e$.  
\item For any $\eta \in \mathbb{P}e^*$, there is a real basis of $W$ in which  
$(\sigma_\varphi)^{-1}(\sigma_\eta):W \to W$ is real and diagonal.
\item The above choice of basis is a smooth function of $[\eta] \in e^*/\varphi =
(e')^*$.
\end{enumerate}
A determined tableau $A \subset W \otimes V^*$ is called \emph{determined hyperbolic} if $V$ admits a
(real) space-like
hyperplane.
\end{defn}

Here is a simple example using our notation from Lemma~\ref{lem:determined}.
Fix $n=3$.  To meet the first condition, suppose that $\varphi = 1u^1+ 0u^2 + 0u^3$ is
not in $\Xi$.  Then $\sigma_\varphi = \B^1_3$, and $0$ is not an eigenvalue of
$\sigma_\varphi$, which of course implies that $\sigma_\varphi = \B^1_3$ is
invertible.  Say $\eta = 0u^1 + 1u^2 + \tau u^3$, so that $\sigma_\eta = \B^2_3 - \tau
I_r$.  The second condition is that  $(\B^1_3)^{-1}\left(\B^2_3 - \tau
I_r\right)$ is diagonalizable using some change-of-basis $g_\tau$.  The third condition
is that $g_\tau$ is continuous in the projective variable $\tau$.
Suppose moreover that we take our basis such that the basis-change at $\tau=0$
is $g_0=I$.
Then we have the condition that $(\B^1_3)^{-1}\B^2_3$ is a diagonal matrix, $D$.   
This puts restrictions on the possible forms of these matrices.
For example, $\ker \B^2_3 = \ker D$ and $\im \B^2_3 \subset \im \B^1_3$.

\begin{defn}
A tableau $A \subset W\otimes V^*$ is called \emph{hyperbolic} if $V$ admits a flag given by a basis $(u^1, \ldots,
u^n)$ of $V^*$ such that each of the sequential initial value problems from
$\pair{u^i, \ldots, u^n}^\perp$ to
$\pair{u^{i+1}, \ldots, u^n}^\perp$ has a hyperbolic determined tableau.
\end{defn}

\begin{thm}[Yang]
Theorem~\ref{thm:CK} applies in the smooth category, if $A$ is hyperbolic.
\end{thm}

The proof proceeds by replacing the Cauchy--Kowalevski initial-value problem
with the Cauchy initial-value problem for determined first-order quasilinear
hyperbolic PDEs.  See \cite{Yang1987} and Appendix A of \cite{Kamran1989} for
more details.

Clearly the definition of \emph{hyperbolic} depends on the geometry of $\Xi$
and the symbol maps $\B^\lambda_i$;  however, to the author's knowledge no one has
succeeded in writing down the explicit criteria on $\B^\lambda_i$ or $\Cone$ or $\Xi$ for
general hyperbolicity.  Hence, Yang's condition is not yet
available to computer algebra systems.  If that can be accomplished, it means we can identify a
subvariety of the moduli of involutive tableaux---as in
Section~\ref{sec:moduli}---that admit solutions in the smooth category.

One well-understood special case is when $\ell =1$, so $\Xi_e$ contains
$s_1$ real points (with multiplicity).  If the number of distinct points is
sufficiently large (greater than $n$), then this is the situation for hyperbolic
systems of conservation laws, as in \cite{Tsarev1991}.  The eikonal system is
rigid, so each solution is foliated by $s_1$ characteristic hypersurfaces.
Multiplicity corresponds to nilpotent pieces of the generalized eigenspaces of
the symbol endomorphisms $\B^1_i$.  See again Section~\ref{sec:examples}.

\section{Open Problems and Future Directions}\label{sec:open}
Our perspective here has been simple-minded---focusing on matrices and their
computable properties---to gain intuition of
$\Xi$ and $\Eik(\Xi)$ as rapidly as possible.  
The articles \cite{Smith2014a} and \cite{Smith2014b} are founded on this
perspective, but reveal additional detail in the structures discussed here.
For more modern and sophisticated treatment, please see 
\cite{Malgrange2003}, \cite{Kruglikov2007}, and
\cite{Carlson2009}. Additionally, Chapters V--VIII of \cite{BCGGG} contain
significantly more results than we have summarized here. 

To conclude, here are some interesting questions which---to the author's present
knowledge---are open subjects that represent the major
theoretical gaps in the subject of exterior differential systems.
They are worth serious consideration as research projects, and offer great
opportunities for collaboration between analysts, differential geometers,
algebraic geometers, and scientific programmers.  

\begin{enumerate}
\item \textbf{Variety of involutive tableaux.}  For given $r$, $n$, and Cartan
characters $(s_1, \ldots, s_\ell)$, what is the variety of
involutive tableaux (with fixed coefficients)?  Can we compute its dimension or
degree or Hilbert polynomial?  Section~\ref{sec:moduli} demonstrates a
first step toward understanding the variety of involutive tableaux, as
Theorem~\ref{thm:invcond} gives the ideal in certain bases.  However, to answer
the question completely, one would need to examine how the coefficients in
\eqref{eqn:moduliblock} vary under arbitrary changes of basis in $V^*$ and $W$.

\item \textbf{Special hyperbolic integrability criteria.}  Solution techniques (such as
Lax pairs, inverse scattering, hydrodynamic reduction, and B\"acklund
transformations) play a key r\^ole in the analysis of wave-like PDEs, especially
those coming from physics and geometry.  Given that these techniques are
coordinate-invariant, Remark~\ref{rmk:dogma} suggests that they should all
be expressible as algebraic conditions on $\mathscr{M}$.  Expressing those
conditions in an abstract way over $\Xi$ and $M^{(1)}$ would allow more
systematic geometric approach to many of the \emph{ad hoc} methods in the
analysis of PDEs.\footnote{Indeed, the central theme of the conference for
which these notes were prepared was to express Ferapontov's notion of
hydrodynamic integrability in terms of algebro-geometric structures in the
Lagrangian Grassmannian.  The notion of hydrodynamic integrability is tied
completely to the secant variety of $\Cone$.}

\item \textbf{Elliptic systems.}  Consider the classical results regarding elliptic
regularity of quasilinear elliptic operators.  This is another form of
``special integrability criteria.''  How far can the notion of elliptic
regularity be extended to general exterior differential systems?  Certainly the
conditions of involutivity, $\pair{\Xi}=V^*$, and $\Xi_\mathbb{R} = \emptyset$
are necessary, and one can directly translate the classical theorems to an EDS
written specifically to describe a quasilinear second-order elliptic
operator in local coordinates, but what other technical assumptions can be
dropped?  Some discussion appears in \cite[Chapter X\S3]{BCGGG}.

\item \textbf{Moduli of involutive tableaux.}  Refining the first problem in light of
the second and third problems,  
can we identify invariant sub-varieties of the variety of involutive tableaux?
Dogma~\ref{rmk:dogma} indicates that we should be able to identify
subvarieties, such as hyperbolic tableaux, elliptic tableaux, systems
satisfying special integrability conditions, and so on.  What does it mean when
these sub-varieties intersect?  Lewy showed that there are involutive PDEs with
no solution in the smooth category \cite{Lewy1957a}, which cannot happen in the
analytic category.  
Where do the Lewy examples fall in this
variety?  Are there other subvarieties that have not been observed in
classical equations?  If there is any organizing geometry behind the ``nearly
impenetrable jungle'' of involutive PDEs, this is where we should look.

\item \textbf{Weakness of involutivity of characteristics.}  Note that
Theorem~\ref{thm:intchar2} does not regard the involutivity of an exterior
differential system in any direct way; the assumption of involutivity of
$\mathcal{I}$ enters Theorem~\ref{thm:intchar2} only because we know that
$\mathscr{M}_N$ is an ideal of homogeneous polynomials from
\eqref{eqn:charidealv}.  Thus, we expect that the condition ``$\Xi_N$ is the
characteristic scheme of an exterior differential system $\mathcal{I}$, and
$\Eik(\Xi_N)$ is involutive'' is much weaker than ``$\Xi_N$ is the
characteristic scheme of an exterior differential system $\mathcal{I}$, and
$\mathcal{I}$ is involutive.''  The gap between these two statements is
extremely important to explore, as it goes to the heart of the question about
how involutivity leads to solutions of the initial-value problem for a system
of PDEs.
To put this a different way, can we construct an embedded variety $\Xi_N \subset
\mathbb{P}T^*N$ that is involutive, but for which there is \emph{no} involutive
exterior differential system for which $\Xi$ is the characteristic variety?

\item \textbf{Global integrability of the characteristic variety.}
If $A$ is involutive, then the system $\Eik(\Xi_N^o)$ is involutive on an
ordinary integral manifold, $N$.  However, it is not clear whether 
$\Xi^o$ is involutive as a bundle over $M^{(1)}$ itself in any reasonable way
that considers all $N$ simultaneously.  That is, consider
the EDS on $M^{(1)}$ generated by $\mathcal{I}^{(1)} + \pair{\xi}$ for some
section $\xi$ of $\Xi^o \subset V^* \subset \mathbb{P}TM^{(1)} \otimes
\mathbb{C}$. 
Under what circumstances is this involutive?  Can
Gabber's theorem \ref{thm:intchar2} be adopted to this case?  This has theoretical implications
for special integrability conditions (above), because it would allow one to
count special solutions among all solutions from $M^{(1)}$ directly.  Additionally, given its
algebraic nature, can Gabber's theorem provide solutions for certain types
of PDEs with low regularity, bypassing the Lewy examples with various
additional conditions?

\item \textbf{Prolongation theorems.} Does prolongation always uncover solutions of
an exterior differential system, if we remove
the regularity assumptions on $M^{(1)}$ and consider the many components of
the scheme $\Var_n(\Var_n( \cdots (\mathcal{I})\cdots ) )$?  
As experts are well aware, this is has been
the key open question in the subject for most of a century.  (See
\cite[Chapter~VI]{BCGGG}.)
In the context of this monograph, the question is related to whether the block form of involutive
tableau \eqref{eqn:bigB} and the involutivity conditions of
Theorem~\ref{thm:invcond} can be extended from
$M^{(1)}$ to non-smooth points in $\Var_n(\mathcal{I})$?  Because of the
interaction of Guillemin normal form and involutivity with Spencer cohomology
as in Section~\ref{sec:gnf}, such an extension of the endovolutive block form
could be helpful in an effort to construct (or prove the non-existence of)
counterexamples.

\item \textbf{Representation theory of Lie pseudogroups.}  Lie pseudogroups are
subgroups of the diffeomorphism pseudogroup whose trajectories are the
solutions of involutive PDEs. See \cite{Olver2009}. Just as Jordan form (in the guise of
the Levi decomposition) is the key first step toward understanding the
representation of Lie groups, it is reasonable to expect that the endovolutive
block-form \eqref{eqn:bigB} and  Theorem~\ref{thm:invcond} can serve as the
foundation of a representation theory of Lie pseudogroups.  Any results
regarding the 
``moduli of involutive tableaux'' can be applied to Lie pseudogroups with those
tableaux. Indeed, the first
application of Theorem~\ref{thm:intchar1} was the classification of the
primitive Lie pseudogroups \cite{Guillemin1966}.

\end{enumerate}

\subsection*{Acknowledgments}
Thanks to Robert L. Bryant, Niky Kamran, and
Deane Yang for always humoring my interest in these details,
to Ian Morrison for helping me appreciate the role of incidence
correspondences in algebraic geometry, and
to Giovanni
Moreno for arranging this course and for encouraging me repeatedly to complete
these notes. 
The punny photograph in Figure~\ref{fig:pp} is public domain from the U.S. Fish and Wildlife Service.

~\\

\input{smith-warsaw-Xi.bbl}

\end{document}

%% file: smith-warsaw-Xi.bbl
\providecommand{\etalchar}[1]{$^{#1}$}
\providecommand{\bysame}{\leavevmode\hbox to3em{\hrulefill}\thinspace}
\providecommand{\noopsort}[1]{}
\providecommand{\mr}[1]{\href{http://www.ams.org/mathscinet-getitem?mr=#1}{MR~#1}}
\providecommand{\zbl}[1]{\href{http://www.zentralblatt-math.org/zmath/en/search/?q=an:#1}{Zbl~#1}}
\providecommand{\jfm}[1]{\href{http://www.emis.de/cgi-bin/JFM-item?#1}{JFM~#1}}
\providecommand{\arxiv}[1]{\href{http://www.arxiv.org/abs/#1}{arXiv~#1}}
\providecommand{\doi}[1]{\url{http://dx.doi.org/#1}}
\providecommand{\MR}{\relax\ifhmode\unskip\space\fi MR }
\providecommand{\MRhref}[2]{%
  \href{http://www.ams.org/mathscinet-getitem?mr=#1}{#2}
}
\providecommand{\href}[2]{#2}